\documentclass{article}
\usepackage{url}
\usepackage{graphicx, float}
\usepackage{amsmath,amssymb,amsfonts}%
\usepackage{amsthm}%
\usepackage{mathtools}
\usepackage{authblk}

\mathtoolsset{showonlyrefs=true}
\usepackage{hyperref}
\usepackage{amsrefs}
\usepackage[margin=1.0in]{geometry}

\newcommand{\R}{\mathbb{R}}

\newcommand{\D}{\mathcal{D}}
\newcommand{\E}{\mathcal{E}}
\newcommand{\sgn}{\mathrm{sgn}}

\newcommand{\sk}{\langle k \rangle}
\newcommand{\skp}{\langle k' \rangle}
\newcommand{\skm}{\langle k-k' \rangle}

\newcommand{\vk}{\langle \xi \rangle}
\newcommand{\vkp}{\langle \xi' \rangle}
\newcommand{\vkm}{\langle \xi - \xi' \rangle}

\newcommand{\dk}{\langle c \lambda_k t \rangle}
\newcommand{\dkp}{\langle c \lambda_{k'} t \rangle}
\newcommand{\dkm}{\langle c \lambda_{k-k'} t \rangle}

\newtheorem{theorem}{Theorem}[section]
\newtheorem{lemma}[theorem]{Lemma}
\newtheorem{proposition}[theorem]{Proposition}
\newtheorem{corollary}[theorem]{Corollary}

\theoremstyle{definition}

\theoremstyle{remark}
\newtheorem{remark}{Remark}[section]

\numberwithin{equation}{section}

\title{Enhanced Dissipation, Taylor Dispersion, and Inviscid Damping of Couette flow in the Boussinesq system on the Plane}
\author[1]{Ryan Arbon}
\affil[1]{Department of Mathematics, University of California, Los Angeles, 90095, CA. Email: rarbon@math.ucla.edu}

\begin{document}
\maketitle




\begin{abstract}
We consider the quantitative asymptotic stability of the stably stratified Couette flow solution to the 2D fully dissipative nonlinear Boussinesq system on $\R^2$ with large Richardson number $R > 1/4$, viscosity $\nu$ and density dissipation $\kappa$. For an initial perturbation $(\omega_{in}, \theta_{in})$ of size $\mu^{1/2 + \epsilon}$ in a low-order anisotropic Sobolev space, for $\mu$ roughly $\min(\nu, \kappa)\left(1 - O(1/\sqrt{R})\right)$ and $\nu$, $\kappa$ comparable, we demonstrate asymptotic stability with explicit enhanced dissipation and Taylor dispersion rates of decay. We also give inviscid damping estimates on the velocity $u$ and the density $\theta$. This is the first result of its type for the Boussinesq system on the fully unbounded domain $\R^2$. We also translate some known linear results from $\mathbb{T} \times \R$ to $\R^2$, and we give an alternative theorem for the nonlinear result.
\end{abstract}

\maketitle

\pagestyle{plain}

\setcounter{tocdepth}{1}
\tableofcontents



\section{Introduction}\label{introduction}
Consider the fully dissipative 2D Boussinesq system, posed on $(x,y) \in \R^2$:
\begin{equation}\label{full_dissipation}
    \begin{cases}
        &\partial_t v + (v \cdot \nabla) v - \nu \Delta v + \nabla P = - \rho \mathfrak{g} e_2,\\
        &\partial_t \rho + v \cdot \nabla \theta - \kappa \Delta \rho = 0,\\
        &\nabla \cdot v = 0,\\
        &v(0,x,y) = v_{in}(x,y), \; \; \rho(0,x,y) = \rho_{in}(x,y).
    \end{cases}
\end{equation}
In \eqref{full_dissipation}, $v = (v^1, v^2)$ is the fluid velocity, $\rho$ is the scalar temperature (or density depending on model applications), $P$ is the fluid pressure, $e_2 = (0, 1)$, $\nu$ is the fluid viscosity, $\kappa$ is a dissipation parameter in the temperature, and $\mathfrak{g}$ is the gravitational constant. In this paper, we are chiefly concerned with the stability properties of the \textit{stably stratified Couette flow} solution to \eqref{full_dissipation}, namely the equilibrium flow
$$v = u_E \coloneqq (y,0), \quad \rho = \theta_E \coloneqq 1 - b y,$$
with $b > 0$ a parameter. Let $P_E$ be the pressure associated with the stably stratified Couette flow. We now introduce the perturbations
$$v = u + u_E, \quad \rho = b \theta + \theta_E, \quad P = P_E + p,$$
so that \eqref{full_dissipation} becomes the following system of equations for the perturbation:\begin{equation}\label{velocity_perturb_eqn}
    \begin{cases}
        &\partial_t u + y \partial_x u + u \cdot \nabla u + (u^2,0) + \nabla p = \nu \nabla u - R \theta e_2,\\
        & \partial_t \theta + y \partial_x \theta ++ u \cdot \nabla \theta= \kappa \Delta  \theta + u^2,\\
        &\nabla \cdot u = 0,\\
        &u(0,x,y) = u_{in}(x,y), \; \; \theta(0,x,y) = \theta_{in}(x,y),
    \end{cases}
\end{equation}
where $R \coloneqq b \mathfrak{g}$ is the Richardson number. Throughout this paper, we will assume that the Miles-Howard criterion is satisfied, meaning that $R > 1/4$ \cites{howard,miles}. Letting $\nabla^\perp \coloneqq (-\partial_y , \partial_x)$ and upon defining the vorticity $\omega \coloneqq \nabla^\perp \cdot u$, we obtain the vorticity formulation of \eqref{velocity_perturb_eqn}:
\begin{equation}\label{vorticity_perturb_eqn}
    \begin{cases}
        &\partial_t \omega + y \partial_x \omega + u \cdot \nabla \omega = \nu \nabla \omega - R \partial_x \theta,\\
        & \partial_t \theta + y \partial_x \theta + u \cdot \nabla \theta = \kappa \Delta  \theta + \partial_x \psi,\\
        &u = \nabla^\perp \psi, \quad \Delta \psi = \omega,\\
        &\omega(0,x,y) = \omega_{in}(x,y), \; \; \theta(0,x,y) = \theta_{in}(x,y).
    \end{cases}
\end{equation}
We consider the scenario where $\nu$ and $\kappa$ are comparable, as determined by $R$. Specifically, we assume that
\begin{equation}\label{diffusion_assumption}
    \frac{ \max\{\nu, \kappa\} }{ \min\{\nu, \kappa\} } \leq 4 \sqrt{R} -1 - \epsilon
\end{equation}
for some fixed $\epsilon > 0$ and we set
\begin{equation}\label{mu_definition}
    \mu \coloneqq \min\{\nu, \kappa\} \left(1 - \frac{1}{4\sqrt{R}} - \frac{1}{4\sqrt{R}} \frac{\max\{\nu, \kappa\}}{\min\{\nu, \kappa\}}\right).
\end{equation}
The condition \eqref{diffusion_assumption} and definition \eqref{mu_definition} were first used in \cite{cotizelati_delzotto} and then in \cite{bianchini_cotizelati_dolce_2}. Our main object in examining the system \eqref{vorticity_perturb_eqn} is to produce a \textit{quantitative stability estimate}. Quantitative stability theory, and especially the quantitative stability of shear flows in fluid-related systems, has grown widely as a topic of research in the last several years, with advances in understanding (at the nonlinear level) the Navier-Stokes/Euler equations \cites{arbon_bedrossian , bedrossian2023stability, stabilityoverview, Beekie_He_2024,masmoudi_zhao_2, masmoudi_zhao,g_wang_w_wang_2024}, the Boussinesq system \cites{bedrossian_zelati_bianchini_dolce, masmoudi_said-houari_zhao, niu_zhao_2024, wang_2024, zhai_zhao, zillinger}, the MHD equations \cites{Chen_Zi_2024, knobel_zillinger, fwang_zzhang_mhd_2024, zhao_zi, zhang_zi_2023}, and the Boltzmann system \cite{bedrossian2022boltzmann}, among others. The principal goal of quantitative stability can be described as follows.
Suppose that $u_E$ is an equilibrium solution to some PDE system which has a diffusive parameter $\nu$ (e.g. the viscosity in the Navier-Stokes system). Then given two norms $|| \cdot ||_X$ and $||\cdot||_Y$, determine a minimal $\gamma > 0$ (called the stability threshhold) such that
\begin{equation}
    ||u_{in} - u_E||_X \ll \nu^\gamma \implies \begin{cases}
        ||u(t) - u_E(t)||_Y \ll 1, \; \; \forall t \in (0,\infty),\\
        ||u(t) - u_E(t)||_Y \to 0 \; \; \textrm{as} \; \; t \to \infty,
    \end{cases}
\end{equation}
where $u(t) - u_E$ is the solution to the PDE at time $t$ with the initial data $u_{in} - u_E$. For an overview of quantitative stability for shear flows in the Navier-Stokes/Euler system, see \cite{stabilityoverview}. The norms $X$ and $Y$ are generally chosen in such a way as to extract additional information about the stability of the flow, such as a more precise decay rate (note that we may take the norm $Y$ to be time-dependent). In fluid dynamics problems, three important properties to capture are \textit{enhanced dissipation}, \textit{Taylor dispersion}, and \textit{inviscid damping}. We will describe each of these in turn.

Enhanced dissipation is a general phenomenon observed in the advection-diffusion of linear passive scalars by shear flows, where large $x$-frequencies $|k|$ decay with a dissipative time scale faster than that predicted by the heat equation \cite{cotizelati2023}. Here, the $x$ direction is the direction perpendicular to the shear flow. In two dimensions,  the $x$-frequency $k$ with $|k| \gg \nu$ decays like $\exp\left(\nu^{1/3} |k|^{2/3} t \right)$. This phenomenon was first analyzed by Kelvin, who found that perturbations of the linearized periodic Couette flow in the Navier-Stokes system decay at time scales $\nu^{-1/3}$ \cite{kelvin1887}. In the context of the Boussinesq equations, nonlinear enhanced dissipation has been observed for Couette flow in the finite channel $\mathbb{T} \times [-1,1]$ with Navier boundary conditions \cite{masmoudi_zhai_zhao}, for Couette flow on $\mathbb{T} \times \R$ \cites{zhai_zhao, zillinger}, and for Poiseuille flow in the finite channel with Navier boundary conditions \cite{wang_2024}. Enhanced dissipation at the linear level on $\mathbb{T} \times \R$ was examined in \cite{zillinger} and again in \cite{bianchini_cotizelati_dolce_2}. Three-dimensional enhanced dissipation on $\mathbb{T} \times \R \times \mathbb{T}$ was proven in \cite{cotizelati_delzotto} for the linear system.

Taylor dispersion refers to the decay of low frequencies in the advection-diffusion of a passive scalar by a shear flow. The decay rate for these frequencies is $\exp(t k^2/\nu )$. See \cites{aris1956, taylor1953, taylor1954} for some of the original physics-based work on the subject. As of this paper, Taylor dispersion for the velocity or vorticity in \eqref{vorticity_perturb_eqn} has yet to be described by the pure mathematics literature, since Taylor dispersion is an ultra-low frequency phenomenon, while much of the literature focuses on periodic-in-$x$ flows, thereby discarding the low frequencies. The work \cite{cotizelati2023} by Coti Zelati and Gallay demonstrates a hypocoercive scheme which simultaneously proves enhanced dissipation and Taylor dispersion results for advection-diffusion equations in cylinders of arbitrary dimension. The idea of this hypocoercive scheme (with modifications found in \cite{bedrossian2023stability}) was used in \cite{arbon_bedrossian} to present a unified description of enhanced dissipation and Taylor dispersion for Couette flow in the Navier-Stokes equations on unbounded domains. See also \cite{bedrossian2022boltzmann} for similar ideas in the context of the Boltzmann equations.

By analogy with Landau damping in plasma physics (see \cite{ryutov}), inviscid damping refers to the behavior of the velocity in \eqref{velocity_perturb_eqn} (or similar systems) when the viscosity $\nu$ is $0$ (or in the limit as $\nu \to 0$). Generally speaking, one observes algebraic-in-$t$ convergence of the velocity to its $x$-average, with different components decaying at potentially different rates. See \cites{bedrossian_masmoudi_2015, bedrossian2016enhanced} for results on Couette flow in the nonlinear Navier-Stokes/Euler equations. Inviscid damping for shear flows in the Boussinesq system \eqref{vorticity_perturb_eqn} has seen a series of new developments. Yang and Lin described inviscid damping for linearized Couette flow with exponentially stratified density on $\mathbb{T} \times \R$ using hypergeometric functions \cite{yang_lin}. Bianchini, Coti Zelati, and Dolce more recently employed an energy method and symmetrization techniques to prove an invisicd damping result for near Couette flow for exponentially stratified fluids and for the Boussinesq approximation \cite{bianchini_cotizelati_dolce}. In particular, it was found for the Boussinesq system on $\mathbb{T}\times\R$ with $\nu = \kappa = 0$ and $R > 1/4$, that perturbations of the Couette flow satisfy
\begin{equation}\label{periodic_inv_damping}
    \begin{split}
        &||P_{\neq} u^1||_{L^2} + ||P_{\neq} \theta||_{L^2}\lesssim \langle t \rangle^{-1/2}\left(||P_{\neq} \omega_{in}||_{L^2} + ||P_{\neq} \theta_{in}||_{H^1}\right),\\
        &||P_{\neq} u^2||_{L^2}\lesssim \langle t \rangle^{-3/2}\left(||P_{\neq} \omega_{in}||_{H^1} + ||P_{\neq} \theta_{in}||_{H^2}\right),\\
        &||P_{\neq} \omega_{in}||_{L^2} + ||P_{\neq} \nabla \theta||_{L^2} \gtrsim \langle t \rangle^{1/2}\left(||P_{\neq} \omega_{in}||_{L^2_x H^{-1}_y} + ||P_{\neq} \theta_{in}||_{H^1_x L^2_y}\right),\\
        &||P_{\neq} \omega||_{L^2} + ||P_{\neq} \nabla \theta||_{L^2} \lesssim \langle t \rangle^{1/2}\left(||P_{\neq} \omega_{in}||_{L^2} + ||P_{\neq} \theta_{in}||_{H^1}\right),
    \end{split}
\end{equation}
where $P_{\neq}$ denotes to projection to non-zero $x$ frequencies \cite{bianchini_cotizelati_dolce}. See \cites{cotizelati_nualart, cotizelati_nualart_2} for additional linear inviscid damping results proved using Whittaker functions. We note that in the inviscid setting, we observe \textit{growth} of the vorticity, in contrast to the decay of the velocity. See \cite{bedrossian_zelati_bianchini_dolce} for a discussion of the nonlinear inviscid damping and this \textit{shear buoyancy instability} caused by the growth of $\omega$ and $\nabla \theta$. See \cite{bianchini_cotizelati_dolce_2} for a general overview of the 2D Boussinesq stability problem on $\mathbb{T}\times \R$.

Unlike previous works, the setting for our paper is the fully unbounded domain $\R \times \R$ without any periodicity assumptions. Our goal is to demonstrate enhanced dissipation, Taylor dispersion, and inviscid damping. We will do so with respect to two sets of norms. The first of these norms will be as follows. Fix $m > 0,$ $n \in [0, \infty)$, and $J \in [1,\infty)$. The norm $V(0)$ for the initial perturbation will be
\begin{equation}\label{def_initial_norm}
    ||f||_{V(0)} \coloneqq \sum_{0 \leq j \leq 1}|| \langle \partial_x, \partial_y \rangle^n |\partial_x, \partial_y|^{-1/2}  \langle \partial_x \rangle^{m + 1/2} \langle \frac{\partial_x}{\mu}\rangle^{-j/3}\partial_y^j f ||_{L_{x,y}^2}
\end{equation}
and the target norm $V(t)$ is
\begin{equation}\label{def_target_norm}
    ||g||_{V(t)} \coloneqq \sum_{0 \leq j \leq 1}|| \langle \partial_x, \partial_y + t \partial_x\rangle^{n}  |\partial_x, \partial_y + t \partial_x|^{-1/2}  \langle \partial_x \rangle^{m + 1/2} \langle c \lambda(\mu, \partial_x) t \rangle^J \langle \frac{\partial_x}{\mu}\rangle^{-j/3} \partial_y^j g||_{L_{x,y}^2}
\end{equation}
where $c$ is a small constant independent of $\mu$, and $\lambda(\mu,\partial_x)$ is a Fourier multiplier defined on the Fourier side as
\begin{equation}\label{definition_of_lambda_k}
        \lambda(\mu, k) = \begin{cases}
        \mu^{1/3}|k|^{2/3}, & |k| \geq \mu\\
        \frac{|k|^2}{\nu}, & |k| \leq \mu.
    \end{cases}
    \end{equation}
 We will find a stability threshold of $\mu^{1/2 + \delta_*}$, where $\delta_* \in(0,1/12)$ is an arbitrarily small parameter. Currently, the best known result for Couette flow and a linearly stratified density on $\mathbb{T}\times \R$ gives a stability threshold of $\nu^{1/2}$ (assuming $\nu = \kappa$), with the size of the initial perturbation being measured as $||u_{in}||_{H^{s+1}} + ||\theta_{in}||_{H^{s+2}} \lesssim \nu^{1/2}$ for $s \geq 6$ \cites{bianchini_cotizelati_dolce_2, zhai_zhao}. In our case, we use a hypocoercive anisotropic Sobolev norm to obtain much lower regularity in the initial perturbation.

We are now prepared to state the main theorem.

\begin{theorem}\label{main_theorem}
     Suppose $(\omega_{in}, \theta_{in})$ are initial data for \eqref{vorticity_perturb_eqn} and that $R > 1/4$, $\epsilon \in (0,1/2)$ are given. Fix $m \in (1/2, \infty)$, $n \in [0,\infty)$, $J \in [1,\infty)$, and $\delta_* \in (0,1/12)$. Then there exists a constant $\delta = \delta( n,m, J, R, \epsilon, \delta_*) > 0$ independent of $\nu$ and $\kappa$ such that if
\begin{equation}
\begin{split}
        ||\omega_{in} ||_{V(0)} + \sqrt{R}|| \nabla \theta_{in}||_{V(0)} = \zeta \leq \delta \mu^{1/2 + \delta_*},
\end{split}
\end{equation}
    then for all $c = c(R, \epsilon) > 0$ sufficiently small (independent of $\mu$ and $\delta$) and all $\nu, \kappa \in (0,1)$ satisfying \eqref{diffusion_assumption}, the corresponding solution $(\omega,\theta)$ to \eqref{vorticity_perturb_eqn} satisfies the following estimates $\forall t\in [0,\infty)$:
\begin{equation}\label{big_estimate}
\begin{split}
        &||  \omega||_{V(t)} +   \sqrt{R} || \nabla \theta||_{V(t)} \leq 2 \langle t \rangle^{1/2}\zeta,\\
        &|| \partial_x u^1||_{V(t)}  + \sqrt{R}  || \partial_x \theta ||_{V(t)}\leq 2 \langle t \rangle^{-1/2}\zeta,\\
        &|| \partial_x |\partial_x, \partial_y + t \partial_x|^{-1} \partial_x u^2||_{V(t)} \leq 2\langle t \rangle^{-3/2}\zeta,\\
        &C \left(\int_0^t  \langle s \rangle ||\partial_x u^2||_{V(s)}^2 ds\right)^{1/2} \leq \zeta,
        \end{split}
\end{equation}
for some constant $C = C(R, \epsilon, c) > 0$ independent of $\nu$, $\kappa$, and $\delta$.
\end{theorem}
\begin{remark}
    We first comment on the inclusion of inverse powers of $|\partial_x, \partial_y + t \partial_x|$ in the norm $V(t)$ \eqref{big_estimate}. We note that, when considered as Fourier multipliers, $|\partial_x, \partial_y + t \partial_x|^{-1/2} \gtrsim \langle \partial_x, \partial_y + t \partial_x\rangle^{-1/2}.$ So then if $n \geq 1/2$, then the estimates on $(\omega, \sqrt{R} \nabla \theta)$ and $(\partial_x u^1, \sqrt{R} \partial_x \theta)$ can be interpreted as controlling strictly positive regularity. Taking $n \geq 3/2$ does the same for $\partial_x u^2$.

    Furthermore, one can make alternate inviscid damping statements on $u$ and $\theta$ with $\langle \partial_x, \partial_y + t \partial_x \rangle^{-1/2}$ replacing $|\partial_x, \partial_y + t \partial_x|^{-1/2}$ in $V(t)$. This moderately changes the conclusion of theorem. For example, one has
    $$|| \langle t \partial_x \rangle^{1/2}  |\partial_x|^{1/2} u^1||_{V(t)} +  \sqrt{R} ||\langle t \partial_x \rangle^{1/2}  |\partial_x|^{1/2} \theta||_{V(t)}\leq 2\zeta$$
    with the aforementioned replacement in $V(t)$.
    We note that we are also able to obtain inviscid damping, enhanced dissipation, and Taylor dispersion at the linear level, see Corollaries \ref{inviscid_damping_and_vorticity_growth} and \ref{enhanced_dissipation_linear} in Section \ref{linear_results}. These result eliminate the regularity requirement of $\langle \partial_x \rangle^{m + 1/2}$ present in the statement of Theorem \ref{main_theorem}.
\end{remark}

In addition to the norms $V(0)$ and $V(t)$ of \eqref{def_initial_norm} and \eqref{def_target_norm}, there exists another set of norms for which we can state an alternate theorem. Letting $\hat{f}$ denote the Fourier transform of a function in the $(x,k)$ Fourier pair, define 
\begin{equation}\label{def_initial_norm_2}
    ||f||_{W(0)} \coloneqq \sum_{0 \leq j \leq 1}|| \langle \partial_x, \partial_y \rangle^n |\partial_x, \partial_y|^{-1/2}  \langle \partial_x \rangle^{m}|\partial_x|^{1/2} \langle \frac{\partial_x}{\mu}\rangle^{-j/3}\partial_y^j f ||_{L_{x,y}^2} + ||\langle k, \partial_y \rangle^n |k|^{1/2} |k, \partial_y|^{-1/2} \hat{f}||_{L^\infty_k L^2_y}
\end{equation}
and
\begin{equation}\label{def_target_norm_2}
\begin{split}
        ||g||_{W(t)} &\coloneqq \sum_{0 \leq j \leq 1}|| \langle \partial_x, \partial_y + t \partial_x\rangle^{n}  |\partial_x, \partial_y + t \partial_x|^{-1/2}  \langle \partial_x \rangle^{m} |\partial_x|^{1/2} \langle c \lambda(\mu, \partial_x) t \rangle^J \langle \frac{\partial_x}{\mu}\rangle^{-j/3} \partial_y^j g||_{L_{x,y}^2}\\
        &\quad \quad +  || \langle k, \partial_y + t k\rangle^{n} |k|^{1/2} |k, \partial_y + t k|^{-1/2} \hat{g} ||_{L^\infty_k L^2_y}.
\end{split}
\end{equation}
Then we have the following alternate theorem.
\begin{theorem}\label{alt_theorem}
    Under the same assumptions and set-up as \ref{main_theorem} with $W(0)$ replacing $V(0)$, we have the following estimates $\forall t \in [0,\infty)$:
    \begin{equation}
\begin{split}
        &||  \omega||_{V(t)} +   \sqrt{R} || \nabla \theta||_{V(t)} \leq 4 \langle t \rangle^{1/2}\zeta,\\
        &|| \partial_x u^1||_{V(t)}  + \sqrt{R}  || \partial_x \theta ||_{V(t)}\leq 4 \langle t \rangle^{-1/2}\zeta,\\
        &|| \partial_x |\partial_x, \partial_y + t \partial_x|^{-1} \partial_x u^2||_{V(t)} \leq 4\langle t \rangle^{-3/2}\zeta,\\
        &C \left(\int_0^t  \langle s \rangle \sum_{0 \leq j \leq 1}|| \langle \partial_x, \partial_y + s \partial_x\rangle^{n}  |\partial_x, \partial_y + s \partial_x|^{-1/2}  \langle \partial_x \rangle^{m} |\partial_x|^{1/2} \langle c \lambda(\mu, \partial_x) s \rangle^J \langle \frac{\partial_x}{\mu}\rangle^{-j/3} \partial_y^j u^2(s)||_{L_{x,y}^2}^2 ds\right)^{1/2} \leq 4\zeta,\\
        &C || \langle s \rangle^{1/2} \langle k, \partial_y + s k\rangle^{n} |k|^{1/2} |k, \partial_y + s k|^{-1/2} u^2 ||_{L^\infty_k L^2([0,t]) L^2_y} \leq 4 \zeta.
        \end{split}
\end{equation}
\end{theorem}
\begin{remark}
    We note the presence of the $L^\infty_k$ norm in the statement of Theorem \ref{alt_theorem}. The purpose of this norm is to allow control of, roughly speaking, an additional half-derivative at low frequencies. Similar ideas have been incorporated in \cite{arbon_bedrossian} and \cite{bedrossian2022boltzmann}. Meanwhile, Theorem \ref{main_theorem} does not require this $L^\infty_k$ portion in the norm, since in Theorem \ref{main_theorem}, we are propagating control over $\langle \partial_x \rangle^{1/2}|\partial_x, \partial_x, \partial_y + t \partial_x|^{-1/2}$, which gives the additional half-derivative at low frequencies.
\end{remark}
The proofs of the main theorem \ref{main_theorem} and the alternate theorem \ref{alt_theorem} are quite similar. Sections \ref{outline_section}-\ref{nonlinear_section} will be focused largely on the proof of Theorem \ref{main_theorem}. Section \ref{alt_section} will sketch the proof of Theorem \ref{alt_theorem}, highlighting key differences from the proof of Theorem \ref{main_theorem}.

\section{Outline}\label{outline_section}

The key to proving Theorem \ref{main_theorem} will be the use of an energy method involving symmetrized variables in a moving reference frame, as in \cites{bedrossian_zelati_bianchini_dolce, bianchini_cotizelati_dolce, zhai_zhao}. We note also that symmetrized variable methods have found applications in other PDE systems, such as the MHD system \cites{fwang_zzhang_mhd_2024, zhao_zi}. Our principal vehicle will be Theorem \ref{symmetric_theorem}, from which Theorem \ref{main_theorem} will follow as a corollary in Section \ref{proof_of_main}. We begin by setting up the statement of Theorem \ref{symmetric_theorem}.

\subsection{Symmetrized Variables and the Moving Reference Frame}\label{symmetry_section}

We begin by recasting \eqref{vorticity_perturb_eqn} into a moving frame of reference. We make the change of variables
\begin{equation}\label{moving_reference_frame}
    X = x - yt, \;\; Y = y, \; \; \Omega(t,X,Y) = \omega(t,x,y), \; \; \Theta(t,X,Y) = \theta(t,x,y).
\end{equation}
In general, we use capital letters to refer to coordinates and variables in the moving reference frame. Let us define 
$$p(t,\partial_X, \partial_Y) \coloneqq -\partial_X^2 - (\partial_Y - t\partial_X)^2, \; \; \nabla_t := (\partial_X, \partial_Y - t \partial_X).$$
Then \eqref{vorticity_perturb_eqn} in the reference frame \eqref{moving_reference_frame} becomes
\begin{equation}\label{moving_reference_equations}
    \begin{cases}
        &\partial_t \Omega + U \cdot \nabla_t \Omega = - \nu p\Omega - R \partial_X \Theta,\\
        & \partial_t \Theta + U \cdot \nabla_t\Theta = -\kappa p \Theta - \partial_X p^{-1} \Omega,\\
        &U = \nabla_t^\perp \Psi, \quad - p \Psi = \Omega,\\
        &\Omega(0,X,Y) = \omega_{in}(x,y), \; \; \Theta(0,X,Y) = \theta_{in}(x,y).
    \end{cases}
\end{equation}
We now introduce the symmetrized variables $(z,q)$ in the stationary reference frame, or $(Z,Q)$ in the moving reference frame, as follows:
\begin{equation}\label{def_of_symmetrized_coords}
    \begin{split}
        &z \coloneqq \langle \partial_x \rangle^{1/2}|\nabla|^{-1/2} \omega, \;\; \; q \coloneqq \sqrt{R} \partial_x |\partial_x|^{-1} \langle \partial_x \rangle^{1/2} |\nabla|^{1/2}\theta,\\
        &Z \coloneqq \langle \partial_X \rangle^{1/2}p^{-1/4} \Omega, \; \; \; Q \coloneqq \sqrt{R} \partial_X |\partial_X|^{-1} \langle \partial_X \rangle^{1/2} p^{1/4}  \Theta.
    \end{split}
\end{equation}
We note that the definition \eqref{def_of_symmetrized_coords} differs slightly, at low-in-$(x,k)$ frequencies, from the symmetrization in \cites{bianchini_cotizelati_dolce, zhai_zhao}. The slightly more traditional symmetrization leads to Theorem \ref{alt_theorem}, see Section \ref{alt_section}. In terms of the symmetrized variables, \eqref{moving_reference_equations} becomes
\begin{equation}\label{Symmetrized_moving_reference_equations}
    \begin{cases}
        &\partial_t Z - \frac{1}{4}\frac{\partial_t p}{p}Z +\langle \partial_X \rangle^{1/2}p^{-1/4}\left(U \cdot \nabla_t \Omega\right) = - \nu p Z - \sqrt{R} |\partial_X|p^{-1/2} Q,\\
        & \partial_t Q + \frac{1}{4}\frac{\partial_t p}{p}Q+ \sqrt{R} \partial_X |\partial_X|^{-1} \langle \partial_X \rangle^{1/2} p^{1/4} \left(U \cdot \nabla_t\Theta\right) = -\kappa p Q + \sqrt{R} |\partial_X| p^{-1/2} Z,\\
        &U = \nabla_t^\perp \Psi, \quad - p \Psi = \Omega,\\
        &Z(0,X,Y) = Z_{in}(X,Y), \; \; Q(0,X,Y) = Q_{in}(X,Y).
    \end{cases}
\end{equation}
In \eqref{Symmetrized_moving_reference_equations}, we note that $Z_{in}(X,Y) = z_{in}(x,y) = \langle \partial_x\rangle^{1/2} |\nabla|^{-1/2} \omega_{in}(x,y)$ and $Q_{in}(X,Y) = q_{in}(x,y) = \sqrt{R} \partial_x |\partial_x|^{-1}\langle\partial_x\rangle^{1/2} |\nabla|^{1/2} \theta_{in}(x,y)$. Our quantitative stability norms for $(Z,Q)$ will be defined as
$$||F||_{\tilde{H}(0)} \coloneqq \sum_{0 \leq j \leq 1}|| \langle \partial_X, \partial_Y \rangle ^n \langle \partial_X \rangle^m \langle \frac{\partial_X}{\mu}\rangle^{-j/3}\partial_Y^j F ||_{L_{X,Y}^2} ,$$
$$||G||_{\tilde{H}(t)} = \sum_{0 \leq j \leq 1}|| \langle c \lambda(\mu, \partial_x) t \rangle^J \langle \partial_X, \partial_Y \rangle^n \langle \partial_X \rangle^m \langle \frac{\partial_X}{\mu}\rangle^{-j/3} (\partial_Y - t \partial_X)^j G||_{L_{X,Y}^2} .$$
We now state our main theorem for the symmetrized variables.
\begin{theorem}\label{symmetric_theorem}
    Suppose $(Z_{in}, Q_{in})$ are initial data for \eqref{Symmetrized_moving_reference_equations} and that $R > 1/4$, $\epsilon \in (0,1/2)$ are given. Then for all $m \in (0, \infty)$, $n \in [0,\infty)$, $J \in [1,\infty)$, and $\delta_* \in (0,1/12)$, there exists a constant $\delta = \delta(n, m,J,R, \epsilon, \delta_*) > 0$ independent of $\nu$ and $\kappa$ such that if
\begin{equation}
\begin{split}
        ||Z_{in}||_{\tilde{H}(0)} + ||Q_{in}||_{\tilde{H}(0)} = \zeta \leq \delta \mu^{1/2 +  \delta_*},
\end{split}
\end{equation}
    then for all $c = c(R, \epsilon) > 0$ sufficiently small (independent of $\mu$ and $\delta$) and all $\nu, \kappa \in (0,1)$ satisfying \eqref{diffusion_assumption}, the corresponding solution $(Z,Q)$ to \eqref{vorticity_perturb_eqn} satisfies the following estimate:
\begin{equation}\label{symm_estimate}
        ||Z||_{\tilde{H}(t)} + ||Q||_{\tilde{H}(t)} + C \left(\int_0^t ||\frac{\partial_X }{|\nabla_t|} Z||_{\tilde{H}(s)}^2 + || \frac{ \partial_X }{|\nabla_t|} Q||_{\tilde{H}(s)}^2 ds\right)^{1/2} \leq \zeta, \quad \forall t\in [0,\infty),
\end{equation}
for some constant $C = C(R, \epsilon, c) > 0$ independent of $\nu$, $\kappa$, and $\delta$.
\end{theorem}
In Section \ref{linear_section} we will prove enhanced dissipation, Taylor dispersion, and inviscid damping-type estimates for the linearization of the symmetrized system \eqref{Symmetrized_moving_reference_equations} in a hypocoercive norm (see \eqref{linear_system} and Proposition \ref{linear_proposition}). Second, in Section \ref{nonlinear_section} we will complete the proof of Theorem \ref{symmetric_theorem} by performing a bootstrap argument to extend a version of the linearized estimates to the nonlinear setting under a small initial data condition (see Lemma \ref{bootstrap_lemma}).

\subsection{Outline of the Linearized Problem}
As in \cites{bianchini_cotizelati_dolce}, we will examine the linear system pointwise in frequency space. Letting $(k,\eta)$ be the frequencies conjugate to $(X,Y)$, we begin by removing the nonlinearity from the symmetrized system \eqref{Symmetrized_moving_reference_equations} to obtain the following linear system
\begin{equation}\label{linear_system}
    \begin{cases}
        &\partial_t Z_k(t,\eta) =  \frac{1}{4} \frac{\partial_t p}{p} Z_k(t,\eta) -\sqrt{R} |k| p^{-1/2} Q_k(t,\eta) - \nu p Z_k(t,\eta),\\
        &\partial_t Q_k(t, \eta) = -\frac{1}{4} \frac{\partial_t p}{p}Q_k(t,\eta) + \sqrt{R} |k| p^{-1/2}Z_k(t,\eta) - \kappa p Q(t,\eta),
    \end{cases}
\end{equation}
where we have abused notation slightly, identifying $p$ with its Fourier multiplier $k^2 + (\eta - kt)^2$ and writing $Z_k, Q_k$ for the Fourier transforms of $Z, Q$. Noticing the decoupling of \eqref{linear_system} in frequency space, we introduce the pointwise (in frequency) energy functional for a solution pair $(Z_k(\eta), Q_k(\eta))$ of \eqref{linear_system} (or a solution pair of \eqref{Symmetrized_moving_reference_equations} at Fourier frequency $(k,\eta)$): 
\begin{equation}\label{linear_energy}
\begin{split}
 E_k[Z_k, Q_k] &\coloneqq  \left(N_k + c_\tau \mathfrak{J}_k\right) |Z_k|^2 + \left(N_k + c_\tau \mathfrak{J}_k\right) |Q_k|^2 + \frac{1}{2\sqrt{R}}\mathrm{Re}\left(\frac{\partial_t p}{ |k| p^{1/2}} \left(N_k + c_\tau \mathfrak{J}_k\right) Z_k\bar{Q}_k \right)\\
 & \quad + c_\alpha \alpha_k\biggl(\left(N_k + c_\tau \mathfrak{J}_k\right) (\eta - kt)^2 |Z_k|^2 + \left(N_k + c_\tau \mathfrak{J}_k\right)  (\eta - kt)^2 |Q_k|^2 \\
 &\quad \quad \quad \quad + \frac{1}{2\sqrt{R}}\mathrm{Re}\left(\frac{\partial_t p}{ |k| p^{1/2}} \left(N_k + c_\tau \mathfrak{J}_k\right)  (\eta - kt)^2 Z_k\bar{Q}_k \right)\biggr)\\
 &\quad+ c_\beta \beta_k \biggl( k(\eta -kt) |Z_k|^2 +k(\eta -kt) |Q_k|^2 - \frac{1}{2\sqrt{R}} \mathrm{Re}\left(\frac{\partial_t p}{|k| p^{1/2}}k(\eta-kt) Z_k \bar{Q}_k\right)\biggr),
\end{split}
\end{equation}
where
\begin{equation}\label{definition_of_alpha_k}
    \alpha_k \coloneqq
    \begin{cases}
        1, & |k| <  \mu,\\
        \mu^{2/3} |k|^{-2/3}, & |k| \geq  \mu,
    \end{cases}
\end{equation}
\begin{equation}\label{definition_of_beta_k}
    \beta_k \coloneqq
    \begin{cases}
         \mu^{-1}, & |k| < \mu,\\
         \mu^{1/3} |k|^{-4/3}, & |k| \geq \mu, 
    \end{cases}
\end{equation}
the constants $c_\alpha, c_\beta, c_\tau > 0$ are independent of $\nu$, $\kappa$, and $k$ (but potentially dependent on $R$) to be specified later, the operator $N_k$ is the Fourier multiplier
\begin{equation}\label{definition_of_n_k}
N_k\coloneqq \exp\left(-\frac{\partial_t p}{|k| p^{1/2}} \frac{1}{2\sqrt{R} -1}\right),
\end{equation}
and $\mathfrak{J}_k$ is the \textit{inviscid damping operator} first introduced in \cite{bedrossian2023stability} and used in \cite{arbon_bedrossian}. Although these works relied on defining $\mathfrak{J}_k$ as a physical-in-$y$ singular integral operator, we simply define $\mathfrak{J}_k$ as the Fourier multiplier
\begin{equation}\label{definition_of_J_k}
        \mathfrak{J}_k(t,\eta) \coloneqq \frac{1}{2}\arctan\left(\frac{\eta}{k} - t \right).
\end{equation}
We note also the hypocoercive structure of \eqref{linear_energy}; a similar norm without the inviscid damping operator was used in \cite{cotizelati2023} to prove enhanced dissipation and Taylor dispersion for the advection of passive scalars. In contrast with \cites{arbon_bedrossian, bedrossian2023stability}, the $k$-by-$k$ energy functional here incorporates cross-terms such as $$\frac{1}{2\sqrt{R}}\mathrm{Re}\left(\frac{\partial_t p}{ |k| p^{1/2}} \left(N_k + c_\tau \mathfrak{J}_k\right) Z_k\bar{Q}_k \right).$$
These cross terms arise naturally in the stability theory of the 2D Boussinesq equations with symmetrized variables and have been used extensively \cites{bedrossian_zelati_bianchini_dolce, bianchini_cotizelati_dolce, bianchini_cotizelati_dolce_2, zhai_zhao}.
Written as Fourier multipliers, we immediately see that $N_k$ and $\mathfrak{J}_k$ are bounded, with
\begin{equation}\label{boundedness}
    | \mathfrak{J}_k| \leq \frac{\pi}{4}, \; \; \; \exp\left(-\frac{2}{2\sqrt{R} - 1}\right) \leq |N_k| \leq 1.
\end{equation}
Additionally, we have the useful facts that
\begin{equation}\label{time_derivs}
    \partial_t \mathfrak{J}_k = -\frac{1}{2}\frac{k^2}{p}, \; \; \; \partial_t \left(\frac{\partial_t p}{|k|p^{1/2}}\right) = 2 \frac{|k|^3}{p^{3/2}}, \; \; \; \left|\frac{\partial_t p}{|k|p^{1/2}}\right| \leq 2.
\end{equation}
Associated to the $k$-by-$k$ energy functional is the $k$-by-$k$ dissipation functional:
\begin{equation}\label{linear_dissipation}
    \begin{split}
        D_k[Z_k, Q_k] &\coloneqq \mu p\left(|Z_k|^2 + |Q_k|^2 \right) +c_\tau |k|^2 p^{-1}\left(|Z_k|^2 + |Q_k|^2\right)\\
        &\quad + c_\alpha \alpha_k \mu p(\eta-kt)^2\left(|Z_k|^2 + |Q_k|^2\right) + c_\tau c_\alpha \alpha_k |k|^2 (\eta-kt)^2 p^{-1} \left(|Z_k|^2 + |Q_k|^2\right)\\
        &\quad  +\frac{1}{\sqrt{R}}|k|^3 p^{-3/2}\left(|Z|_k^2 + |Q_k|^2\right) + \frac{c_\alpha}{\sqrt{R}} \alpha_k |k|^3 (\eta-kt)^2 p^{-3/2}\left(|Z_k|^2 + |Q_k|^2\right)\\
        &\quad + c_\beta \beta_k |k|^2 \left( |Z_k|^2 + |Q_k|^2\right)\\
        &\eqqcolon D_{k, \gamma} + c_\tau D_{k, \tau} + c_\alpha D_{k, \alpha} + c_\tau c_\alpha D_{k, \tau \alpha} + \frac{1}{\sqrt{R}} D_{k,\rho} + \frac{c_\alpha}{\sqrt{R}} D_{k,\rho}+ c_\beta D_{k, \beta}.
    \end{split}
\end{equation}
For the sake of convenience, we define $c_\rho \coloneqq \frac{1}{\sqrt{R}}$. We are now prepared to state our main result at the linear level, which will be a key ingredient in the proof of Theorem \ref{symmetric_theorem}:
\begin{proposition}\label{linear_proposition}
    Fix $\epsilon \in (0,1/2)$ as in \eqref{diffusion_assumption}, and let $R > 1/4$ be given. There exist constants $c_\alpha$, $c_\beta$, and $c_\tau$, and constants $c_0 = c_0(c_\tau, c_\alpha, c_\beta, R, \epsilon)$, $c_1 = c_1(c_\tau, c_\alpha, c_\beta, R, \epsilon)$ which can be chosen independently of $\nu$ and $\kappa$ such that, for any $H^1$ solution $(Z_k, Q_k)$ to \eqref{linear_system}, the following holds for any $\eta \in \R$ and all $k \neq 0$:
    $$\frac{d}{dt} E_k[Z_k, Q_k]\leq - c_1 D_k[Z_k, Q_k]- c_0\lambda_k E_k[Z_k, Q_k].$$
    In particular, this implies that for any $c>0$ sufficiently small
    $$\frac{d}{dt} E_k[Z_k, Q_k] \leq - 8c D_k[Z_k, Q_k] - 8c\lambda_k E_k[Z_k, Q_k],$$
    and the following inviscid damping, enhanced dissipation, and Taylor dispersion estimate holds:
    \begin{equation}
        \begin{split}
        e^{2c\lambda_k t} E_k[Z_k, Q_k] +& \frac{1}{4}c_\tau \int_0^t e^{2 c \lambda_k s} |k|^2 p^{-1} \left( |Z_k|^2 + |Q_k|^2 \right)   ds\\
        &\quad \frac{1}{2R}\int_0^t e^{2 c \lambda_k s} |k|^3 p^{-3/2} \left(|Z_k|^2 + |Q_k|^2 \right)   ds\leq E_k[z_k(0), q_k(0)], \; \; \forall t \geq 0.
        \end{split}
    \end{equation}
\end{proposition}
The proof of Proposition \ref{linear_proposition} will be contained in Section \ref{linear_section}.
\subsection{Outline of the Nonlinear Problem}
The fully nonlinear system \eqref{Symmetrized_moving_reference_equations} naturally re-couples distinct Fourier modes, so we can no longer solely use the function $E_k$. Instead, we use a modification of the norm from \cite{arbon_bedrossian} and we define 
\begin{equation}\label{nonlinear_energy}
    \begin{split}
        &\E[Z, Q] \coloneqq \iint_{\R^2} \frac{\dk^{2J}}{M_k(t)} \langle k, \eta \rangle^{2n} \langle k \rangle^{2m}E_k[Z_k, Q_k] d\eta dk,
    \end{split}
\end{equation}
where $c$ is a sufficiently small constant determined in Proposition \ref{linear_proposition} and (as in \cite{arbon_bedrossian}) $M_k$ is the solution to the ODE:
\begin{equation}\label{correction_ODE}
    \begin{split}
        \dot{M}_k(t) &=  c J^2 \lambda_k \frac{( c \lambda_k t )^2}{\langle c \lambda_k t \rangle^{4}} M_k(t),\\
        M_k(0) &= 1.
    \end{split}
\end{equation}
We note that $M_k$ is uniformly bounded and is uniformly bounded way from zero in both $k$ and $t$, so that $||M_k(t)^{-1} f||_{L^2_k L^2_\eta} \approx ||f||_{L^2_k L^2_\eta}$ for all $f \in L^2(\R \times \R)$. The purpose of $M_k$ is simply to absorb terms from $\frac{d}{dt} \E$ when the time-derivative falls on the decay term $\langle c \lambda_k t \rangle^{2J}$. We note that $\E[Z,Q] \approx ||Z, Q||_{\tilde{H}(t)}^2$, where $\tilde{H}(t)$ is the target norm in Theorem \ref{symmetric_theorem}.
In a similar manner, we define the nonlinear dissipation:
\begin{equation}\label{nonlinear_dissipation}
    \begin{split}
        &\D[Z, Q] \coloneqq \int_{\R^2} \frac{\dk^{2J}}{M_k(t)} \langle k, \eta \rangle^{2n} \langle k \rangle^{2m} 
 D_k[Z_k, Q_k] d\eta dk.
    \end{split}
\end{equation}
In Section \ref{nonlinear_section}, it will be helpful for the sake of readability to distinguish between terms of the nonlinear dissipation \eqref{nonlinear_dissipation} arising from distinct terms of the linear dissipation \eqref{linear_dissipation}, and so we define
\begin{equation}\label{definition_of_diss}
    \D_{*}[Z,Q] \coloneqq \int_\R \frac{\dk^{2J}}{M_k(t)} \langle k, \eta \rangle^{2n} \langle k \rangle^{2m} c_{*} D_{k,*}[Z_k, Q_k] dk
\end{equation}
for $* \in \{\gamma, \tau, \alpha, \tau \alpha, \beta ,\rho, \rho\alpha\}$, letting $c_{\tau \alpha } = c_\tau c_\alpha$ and $c_{\rho \alpha} = c_\rho c_\alpha$. As in \cites{arbon_bedrossian, bedrossian2023stability, zhai_zhao} (among others), the key to proving Theorem \ref{symmetric_theorem} will be a bootstrap argument based on the following lemma:
\begin{lemma}\label{bootstrap_lemma}
    Let $(Z_{in}, Q_{in})$ be initial datum for \eqref{Symmetrized_moving_reference_equations} such that $\E[Z_{in}, Q_{in}] < \infty$. Then there exists a constant $C > 0$ depending only on $m$, $J$, $R$, $\epsilon$ and the choice of $c$ such that
    \begin{equation}\label{bootstrap_equation}
        \frac{d}{dt}\E[Z,Q] \leq - 4 c \ \D[Z,Q] + \mu^{-1/2 - \delta_*} 
 C^{1/2}\E^{1/2}[Z,Q] \D[Z,Q] .
    \end{equation}
\end{lemma}
\begin{remark}
    Assuming Lemma \ref{bootstrap_lemma}, it is easy to prove Theorem \ref{symmetric_theorem}. Indeed, suppose that $\E[Z_{in},Q_{in}] \leq  c^2 \mu^{1 + 2\delta_*} C^{-1}$ and let $$T = \sup\{t \in \R: \E[Z,Q](t) \leq 4 c^2 \mu^{1 + 2\delta_*} C^{-1}\}.$$ 
    By continuity, $T > 0$. Now for any $t \leq T$, $\sup_{s \in [0,t]} \E(s) \leq \E(T)$. Then by \eqref{bootstrap_equation}, $$\frac{d}{dt}\E(t) \leq - 4 c \D(t) + \mu^{-1/2 - \delta_*} 
 C^{1/2} \E(T) \D(t) \leq -2 c \D(t) dt \leq 0,$$
 for all $t \in (0,T)$. Thus $\E(t) + c \int_0^t \D dt$ is monotonically non-increasing on $[0,T)$, and so $T = \infty$, which proves Theorem \ref{symmetric_theorem}.
\end{remark}

\subsection{Proof of Main Theorem}\label{proof_of_main}

We present here the proof of Theorem \ref{main_theorem} assuming Theorem \ref{symmetric_theorem}. Immediately from the definition of the symmetrized coordinates \eqref{def_of_symmetrized_coords}, if $$||\omega_{in}||_{V(0)} + \sqrt{R}||\nabla \theta_{in}||_{V(0)} \leq \delta \mu^{1/2 + \delta_*}$$ for some $\delta > 0$, then 
$$||Z_{in}||_{\tilde{H}(0)} + ||Q_{in}||_{\tilde{H}(0)} \leq \delta \mu^{1/2 + \delta_*},$$
and so we can apply Theorem \ref{symmetric_theorem} to $(Z,Q)$ if $\delta$ is sufficiently small.
We then note that 
$$p^{1/4} \lesssim \langle t \rangle^{1/2} (k^2 + \eta^2)^{1/4} \lesssim \langle t \rangle^{1/2}\langle k, \eta \rangle^{1/2}$$
and 
$$p^{-1/4} \lesssim |k|^{-1/2} \langle k t \rangle^{-1/2} \langle k, \eta \rangle^{1/2}, \; \; \; p^{-1/4} \lesssim |k|^{-1} \langle t \rangle^{-1/2} (k^2 + \eta^2)^{1/4}.$$ We then have the following:
\begin{equation}\label{decay_exp}
    \begin{split}
         || \omega ||_{V(t)} &=  \sum_{0 \leq j \leq 1}|| \langle \partial_x, \partial_y  + t \partial_x \rangle^{n}|\partial_x, \partial_y + t \partial_x|^{-1/2}\langle \partial_x \rangle^{m+1/2}\langle c \lambda(\mu, \partial_x) t \rangle^J \langle \frac{\partial_x}{\mu}\rangle^{-j/3} \partial_y^j \omega||_{L_{x,y}^2}\\
        &= \sum_{0 \leq j \leq 1}|| \langle k, \eta \rangle^{n}|k, \eta|^{-1/2} \langle \partial_x \rangle^{m+1/2} \langle c \lambda_k t \rangle^J \alpha^j |\eta|^j |k|^{1/2} \Omega_k||_{L_{k, \eta}^2}\\
        &= \sum_{0 \leq j \leq 1}|| \langle k, \eta \rangle^{n}|k, \eta|^{-1/2} \langle k \rangle^m \langle c \lambda_k t \rangle^J \alpha^j |\eta|^j p^{1/4} Z_k||_{L_{k, \eta}^2}\\
        &\lesssim \langle t \rangle^{1/2}\sum_{0 \leq j \leq 1}|| \langle k, \eta \rangle^{n} \langle k \rangle^m \langle c \lambda_k t \rangle^J \alpha^j |\eta|^j Z_k||_{L_{k, \eta}^2} = \langle t \rangle^{1/2}||Z||_{\tilde{H}(t)}.
    \end{split}
\end{equation}
Similar reasoning gives
\begin{equation}\label{decay_from_change}
    \begin{split}
        \sqrt{R} ||  \nabla\theta ||_{V(t)} \lesssim \langle t \rangle^{1/2}|| Q ||_{\tilde{H}(t)},& \quad \quad|| \partial_x u^1||_{V(t)} \lesssim \langle t \rangle^{-1/2}||Z||_{\tilde{H}(t)},\\
        \sqrt{R}|| \partial_x \theta||_{V(t)} \lesssim \langle t \rangle^{-1/2}||Q||_{\tilde{H}(t)},& \quad \quad || \partial_x u^2||_{V(t)} \lesssim  \langle t \rangle^{-1/2} || \partial_x p^{-1/2} Z ||_{\tilde{H}(t)},
    \end{split}
\end{equation}
and
\begin{equation}\label{decay_2}
    ||\partial_x |\partial_x, \partial_y + t \partial_x|^{-1} \partial_x u^2||_{V(t)} \lesssim  \langle t \rangle^{-3/2}|| Z ||_{\tilde{H}(t)}.
\end{equation}
Combining \eqref{decay_exp}, \eqref{decay_from_change}, and \eqref{decay_2} with \eqref{symm_estimate} from the conclusion of Theorem \ref{symmetric_theorem}, we obtain \eqref{big_estimate} and hence Theorem \ref{main_theorem}.

\section{Linear Estimates}\label{linear_section}

The proof of Proposition \ref{linear_proposition} relies on the computation  each term of $\frac{d}{dt} E_k$ in the linear case, together with careful selection of the coefficients $c_\alpha, c_\beta,$ and $c_\tau$. We begin with the following lemma concerning the principal hypocoercive terms of the time derivative of \eqref{linear_energy}:
\begin{lemma}\label{hypo_coerc}
    For $(Z_k, Q_k)$ solving \eqref{linear_system} with $k \neq 0$, the following holds:
    \begin{equation}\label{gamma_est}
    \begin{split}    \partial_t\biggl( N_k |Z_k|^2 + N_k |Q_k|^2 + &\frac{1}{2\sqrt{R}}\mathrm{Re}\left(\frac{\partial_t p}{ |k| p^{1/2}} N_k Z_k\bar{Q}_k \right) \biggr) + 2\mu p N_k\left(|Z_k|^2 + |Q_k|^2\right)\\
    &\quad\quad\quad + \frac{1}{2\sqrt{R}} N_k \frac{|k|^3}{p^{3/4}} \left(|Z_k|^2 + |Q_k|^2\right)\leq 0,
        \end{split}
    \end{equation}
    \begin{equation}\label{alpha_est}
    \begin{split}    \alpha_k \partial_t\biggl( N_k |(\eta -kt)Z_k|^2 + &N_k |(\eta -kt)Q_k|^2 + \frac{1}{2\sqrt{R}}\mathrm{Re}\left(\frac{\partial_t p}{ |k| p^{1/2}} N_k (\eta -kt)^2 Z_k\bar{Q}_k \right) \biggr)\\ &+ 2\mu \alpha_k  p N_k (\eta -kt)^2 \left(|Z_k|^2 + |Q_k|^2\right) + \frac{1}{2\sqrt{R}} \alpha_k N_k \frac{|k|^3}{p^{3/4}} (\eta -kt)^2 \left(|Z_k|^2 + |Q_k|^2\right)\\
    &\quad \leq \lambda_k \left(1 + \frac{1}{2\sqrt{R}}\right)N_k\left(|Z_k|^2 + |Q_k|^2\right) + \mu p N_k\left(|Z_k|^2 + |Q_k|^2\right),
        \end{split}
    \end{equation}
    \begin{equation}\label{beta_est}
    \begin{split}    \beta_k \partial_t \biggl( k(\eta -kt) |Z_k|^2 + &k(\eta -kt) |Q_k|^2 - \frac{1}{2\sqrt{R}} \mathrm{Re}\left(\frac{\partial_t p}{|k| p^{1/2}}k(\eta-kt) Z_k \bar{Q}_k\right)\biggr) + \lambda_k\left(|Z_k|^2 + |Q_k|^2\right)\\
    & \leq \beta_k | \mathrm{Re} \frac{1}{\sqrt{R}} k(\eta-kt) \frac{|k|^3}{p^{3/4}} | + 2 \nu \beta_k p |k(\eta-kt)| |Z_k|^2\\
    &\quad + 2 \kappa \beta_k p |k(\eta-kt)| |Q_k|^2 + \beta_k \frac{\partial_t p}{|k| p^{1/2}} \frac{\nu + \kappa}{2\sqrt{R}}|k(\eta-kt)| |\mathrm{Re}(Z_k \bar{Q}_k)|.
        \end{split}
    \end{equation}

\end{lemma}

\begin{proof}
The proofs of \eqref{gamma_est}, \eqref{alpha_est}, and \eqref{beta_est} are straightforward, relying on direct computation, \eqref{boundedness}, \eqref{time_derivs}, the triangle inequality, and Young's product inequality. The chief complexity comes from the number of terms. We begin by proving \eqref{gamma_est}. By definition of $N_k$ \eqref{definition_of_n_k}  and using that $(Z_k,Q_k)$ solve \eqref{linear_system}, we obtain the following:
\begin{equation}\label{lin_gam_z}
    \begin{split}
        \partial_t \biggl(N_k|Z_k|^2 \biggr) &= -N_k\frac{1}{2\sqrt{R} -1} \partial_t\left(\frac{\partial_t p}{|k| p^{1/2}}\right)|Z_k|^2\\
        & \quad + 2N_k \biggl(-\frac{1}{4} \frac{\partial_t p}{p} |Z_k|^2 -\sqrt{R} |k| p^{-1/2} \mathrm{Re}( Q_k \bar{Z_k}) - \nu p |Z_k|^2 \biggr),
    \end{split}
\end{equation}
\begin{equation}\label{lin_gam_q}
    \begin{split}
        \partial_t \biggl(N_k|Q_k|^2 \biggr) &= -N_k\frac{1}{2\sqrt{R} -1} \partial_t\left(\frac{\partial_t p}{|k| p^{1/2}}\right)|Q_k|^2\\
        & \quad + 2N_k\biggl(\frac{1}{4} \frac{\partial_t p}{p}|Q_k|^2 + \sqrt{R} |k| p^{-1/2}\mathrm{Re}(Z_k \bar{Q}_k) - \kappa p |Q_k|^2 \biggr),
    \end{split}
\end{equation}
\begin{equation}\label{lin_gam_mix}
    \begin{split}
        \frac{1}{2\sqrt{R}} \partial_t \mathrm{Re}\left(\frac{\partial_t p}{ |k| p^{1/2}} N_k Z_k\bar{Q}_k \right) &= \frac{1}{2\sqrt{R}} \biggr( \partial_t \left(\frac{\partial_t p}{|k| p^{1/2}} \right) \mathrm{Re}\left(N_k Z_k \bar{Q_k}\right) \\
        &\quad\quad+ \mathrm{Re}\left(-N_k \frac{1}{2\sqrt{R} -1} \frac{\partial_t p}{ |k| p^{1/2}}  \partial_t \left( \frac{\partial_t p}{|k| p^{1/2}} \right)  Z_k\bar{Q}_k \right)\\
        &\quad\quad +\mathrm{Re}\left(\frac{\partial_t p}{ |k| p^{1/2}}N_k\left(-\frac{1}{4}\frac{\partial_t p}{p}Z_k -\sqrt{R} \frac{|k|}{p^{1/2}}Q_k - \nu p Z_k\right)\bar{Q}_k \right)\\
        &\quad\quad +\mathrm{Re}\left(\frac{\partial_t p}{ |k| p^{1/2}}N_k Z_k \left(\frac{1}{4} \frac{\partial_t p}{p} \bar{Q}_k + \sqrt{R}\frac{|k|}{p^{1/2}} \bar{Z_k} - \kappa p \bar{Q}_k\right)\right).
    \end{split}
\end{equation}
Adding \eqref{lin_gam_z}, \eqref{lin_gam_q}, and \eqref{lin_gam_mix} yields
\begin{equation}\label{lin_gam_main}
\begin{split}
        \partial_t \biggl( N_k |Z_k|^2 + N_k &|Q_k|^2 + \frac{1}{2\sqrt{R}}\mathrm{Re}\left(\frac{\partial_t p}{ |k| p^{1/2}}N_k Z_k\bar{Q}_k \right) \biggr)\\
        &= - N_k \frac{1}{2\sqrt{R}-1} \partial_t \biggl(\frac{\partial_t p}{|k| p^{1/2}}\biggr)\biggl(|Z_k|^2 + |Q_k|^2\biggr)\\
        &\quad + \frac{1}{2\sqrt{R}}\partial_t \left(\frac{\partial_t p}{|k|p^{1/2}}\right)N_k\biggl( 1 - \frac{1}{2\sqrt{R} -1}\frac{\partial_t p}{|k| p^{1/2}}\biggr) \mathrm{Re}(Z_k \bar{Q}_k) \\
        &\quad - 2 N_k\left(\nu p |Z_k|^2 + \kappa p |Q_k|^2\right) -\frac{\nu+\kappa}{2\sqrt{R}} \frac{\partial_t p}{|k| p^{1/2}}N_kp\mathrm{Re}(Z_k \bar{Q}_k).
\end{split}
\end{equation}
Now we compute using the triangle inequality, Young's product inequality, and \eqref{boundedness} , \eqref{time_derivs}:
\begin{equation}\label{leftover_gamma_rho}
    \begin{split}
        \frac{1}{2\sqrt{R}}\biggl|\partial_t \left(\frac{\partial_t p}{|k|p^{1/2}}\right)&N_k \biggl(1 - \frac{1}{2\sqrt{R} -1}\frac{\partial_t p}{|k| p^{1/2}} \ \biggr)\mathrm{Re}(Z_k \bar{Q}_k) \biggr | \\
        &\leq \left(\frac{1}{2\sqrt{R}-1} - \frac{1}{4\sqrt{R}}\right)\left|\partial_t \left(\frac{\partial_t p}{|k|p^{1/2}}\right)\right|\left(N_k |Z_k|^2 + N_k|Q_k|^2\right),
        \end{split}
\end{equation}
and
\begin{equation}\label{leftover_gamma_gamma}
    \begin{split}
        \frac{\nu+\kappa}{2\sqrt{R}}\left|\frac{\partial_t p}{|k| p^{1/2}}N_kp\mathrm{Re}(Z_k \bar{Q}_k)\right| \leq \frac{\nu + \kappa}{2\sqrt{R}} N_k p\left(|Z_k|^2 + |Q_k|^2\right).
    \end{split}
\end{equation}
Combining \eqref{lin_gam_main}, \eqref{leftover_gamma_gamma}, \eqref{leftover_gamma_rho}, as well as using the assumption \eqref{diffusion_assumption}, the definition of $\mu$ \eqref{mu_definition}, and \eqref{time_derivs}, we arrive at \eqref{gamma_est} as desired. We now turn our attention to \eqref{alpha_est}. By performing computations similar to those used in \eqref{lin_gam_z}, \eqref{lin_gam_q}, and \eqref{lin_gam_mix} one can show that
\begin{equation}\label{lin_alph_main}
    \begin{split}
        \partial_t \biggl( N_k (\eta-kt)^2|Z_k|^2 + N_k &(\eta-kt)^2 |Q_k|^2 + \frac{1}{2\sqrt{R}}\mathrm{Re}\left(\frac{\partial_t p}{ |k| p^{1/2}}N_k(\eta-kt)^2 Z_k\bar{Q}_k \right) \biggr)\\ &= - N_k \frac{1}{2\sqrt{R}-1} \partial_t \biggl(\frac{\partial_t p}{|k| p^{1/2}}\biggr)(\eta-kt)^2\biggl(|Z_k|^2 + |Q_k|^2\biggr)\\
        &\quad \quad + \frac{1}{2\sqrt{R}}\partial_t \left(\frac{\partial_t p}{|k|p^{1/2}}\right)N_k\biggl( 1 - \frac{1}{2\sqrt{R} -1}\frac{\partial_t p}{|k| p^{1/2}}\biggr)(\eta-kt)^2\mathrm{Re}(Z_k \bar{Q}_k) \\
        &\quad \quad - 2  N_k(\eta-kt)^2\left(\nu p |Z_k|^2 + \kappa p |Q_k|^2\right) -\frac{\nu+\kappa}{2\sqrt{R}} (\eta-kt)^2\frac{\partial_t p}{|k| p^{1/2}}N_kp\mathrm{Re}(Z_k \bar{Q}_k)\\
        &\quad\quad -2 k(\eta-kt) N_k \left(|Z_k|^2 + |Q_k|^2\right) - \frac{1}{\sqrt{R}} k(\eta-kt) \mathrm{Re}\left(\frac{\partial_t p}{|k| p^{1/2}} N_k Z_k Q_k\right).
    \end{split}
\end{equation}
By the same techniques used in obtaining \eqref{leftover_gamma_gamma} and \eqref{leftover_gamma_rho}, together with \eqref{boundedness} and \eqref{time_derivs}, we obtain
\begin{equation}\label{lin_alph_main_2}
    \begin{split}
          \partial_t\biggl( N_k |(\eta -kt)Z_k|^2 + &N_k |(\eta -kt)Q_k|^2 + \frac{1}{2\sqrt{R}}\mathrm{Re}\left(\frac{\partial_t p}{ |k| p^{1/2}} N_k (\eta -kt)^2 Z_k\bar{Q}_k \right) \biggr)\\ &+ 2\mu  p N_k (\eta -kt)^2 \left(|Z_k|^2 + |Q_k|^2\right) + \frac{1}{2\sqrt{R}} N_k \frac{|k|^3}{p^{3/4}} (\eta -kt)^2 \left(|Z_k|^2 + |Q_k|^2\right)\\
    &\quad \leq 2 |k(\eta-kt)| N_k \left(|Z_k|^2 + |Q_k|^2\right) +  \frac{2}{\sqrt{R}} |k(\eta-kt)| |\mathrm{Re}\left( N_k Z_k Q_k\right)|.
    \end{split}
\end{equation}
After multiplying by $\alpha_k$, the second-to-last term on the right-hand side of \eqref{lin_alph_main_2} can be estimated via Cauchy-Schwarz, Young's product inequality, and the definitions of $\alpha_k$ \eqref{definition_of_alpha_k} and $\lambda_k$ \eqref{definition_of_lambda_k}:
\begin{equation}\label{first_terms}
\begin{split}
        2 \alpha_k |k(\eta-kt)| N_k\left(|Z_k|^2 + |Q_k|^2\right) & \leq 2 N_k \left(\sqrt{\lambda_k}|Z_k| \sqrt{\mu}|(\eta-kt)Z_k|\right) + 2 N_k \left(\sqrt{\lambda_k}|Q_k| \sqrt{\mu}|(\eta-kt)Q_k|\right)\\
        &\leq \lambda_k N_k \left(|Z_k|^2+|Q_k|^2\right) + \mu_k N_k p \left(|Z_k|^2+|Q_k|^2\right).
\end{split}
\end{equation}
The final term of \eqref{lin_alph_main_2} is estimated in a similar fashion:
\begin{equation}\label{middle_term}
    \begin{split}
    \alpha_k\frac{2}{\sqrt{R}} |k(\eta-kt)| \mathrm{Re}\left( N_k Z_k Q_k\right)| & \leq \frac{2\sqrt{\lambda_k \mu}}{\sqrt{R}}N_k\left(|Z_k||Q_k|\right)\\
    &\leq \frac{1}{2\sqrt{R}}N_k\left(\lambda + \mu\right)\left(|Z_k|^2 + |Q_k|^2\right).
    \end{split}
\end{equation}
Together, \eqref{lin_alph_main_2}, \eqref{first_terms}, and \eqref{middle_term} give \eqref{alpha_est}. To complete the proof of Lemma \ref{hypo_coerc}, we need to estimate the $\beta$ terms and obtain \eqref{beta_est}. We compute the relevant time derivatives as follows:
\begin{equation}\label{lin_bet_z}
    \begin{split}
        \partial_t k (\eta -kt) |Z_k|^2 &= - k^2 |Z_k|^2 +2 k (\eta -kt)\left(-\frac{1}{4} \frac{\partial_t p}{p} |Z_k|^2 - \sqrt{R} |k| p^{-1/2} \mathrm{Re}(Q_x \bar{Z}_k) - \nu p|Z_k|^2\right),
    \end{split}
\end{equation}
\begin{equation}\label{lin_bet_q}
    \begin{split}
        \partial_t k (\eta -kt) |Z_k|^2 &= - k^2 |Q_k|^2 +2 k (\eta -kt)\left(\frac{1}{4} \frac{\partial_t p}{p} |Q_k|^2 + \sqrt{R} |k| p^{-1/2} \mathrm{Re}(Z_k \bar{Q}_k) - \kappa p|Q_k|^2\right),
    \end{split}
\end{equation}
\begin{equation}\label{lin_bet_mix}
    \begin{split}
        \frac{1}{2\sqrt{R}} \partial_t \mathrm{Re}\left(\frac{\partial_t p}{ |k| p^{1/2}} k(\eta - kt) Z_k\bar{Q}_k \right) &= \frac{1}{2\sqrt{R}} \biggl( \partial_t \biggl(\frac{\partial_t p}{ |k| p^{1/2}}\biggr) k(\eta - kt)\mathrm{Re}( Z_k\bar{Q}_k) - k^2 \mathrm{Re}\left(\frac{\partial_t p}{ |k| p^{1/2}} Z_k\bar{Q}_k \right)\\
        &\quad + \frac{\partial_t p}{ |k| p^{1/2}}k(\eta -kt)\mathrm{Re}\left((-\frac{1}{4} \frac{\partial_t p}{p} Z_k - \sqrt{R} |k| p^{-1/2} Q_k - \nu p Z_k)\bar{Q}_k\right)\\
        &\quad + \frac{\partial_t p}{ |k| p^{1/2}}k(\eta -kt)\mathrm{Re}\left(\bar{Z}(\frac{1}{4} \frac{\partial_t p}{p} Q_k + \sqrt{R} |k| p^{-1/2} Q_k - \kappa p Q_k)\right)\biggr).
    \end{split}
\end{equation}
Summing \eqref{lin_bet_z}, \eqref{lin_bet_q}, and \eqref{lin_bet_mix} we arrive at
\begin{equation}\label{lin_main_bet}
    \begin{split}
        \partial_t\biggl( k(\eta -kt) |Z_k|^2 + k(\eta -kt) |Q_k|^2 - \frac{1}{2\sqrt{R}} &\mathrm{Re}\left(\frac{\partial_t p}{|k| p^{1/2}}k(\eta-kt) Z_k \bar{Q}_k\right)\biggr) + k^2\left(|Z_k|^2 + |Q_k|^2\right)\\
        &=
        \frac{1}{2\sqrt{R}} \partial_t \biggl( \frac{\partial_t p}{|k| p^{1/2}} \biggr)\mathrm{Re}(Z_k \bar{Q}_k) - \nu p|Z_k|^2 - \kappa p |Q_k|^2\\
        &\quad - \frac{\partial_t p}{|k| p^{1/2}} \frac{\nu + \kappa}{2\sqrt{R}} k(\eta-kt) \mathrm{Re}(Z_k \bar{Q}_k).
    \end{split}
\end{equation}
Applying absolute value signs to the right-hand-side of \eqref{lin_main_bet}, multiplying through by $\beta_k$ and noting that $k^2 \beta_k = \lambda_k$, we obtain \eqref{beta_est} as desired.
\end{proof}

We now need to estimate the $\tau$ and $\tau \alpha$ terms. The only changes, as compared with \eqref{gamma_est} and \eqref{alpha_est}, will arise from the time derivative hitting $\mathfrak{J}_k$.
We also avoid some simplifications of the $\tau \alpha$ terms that are present in \eqref{alpha_est}, as these will require more subtle estimates found in the final proof of Proposition \ref{linear_proposition}.
\begin{lemma}\label{j_lemma}
    The following two estimates hold: 
    \begin{equation}\label{tau_est}
        \begin{split}
        \partial_t\biggl(\mathfrak{J}_k |Z_k|^2 + \mathfrak{J}_k|Q_k|^2 + \mathrm{Re}\biggl(\frac{1}{2 \sqrt{R}} \frac{\partial_t p}{|k| p^{1/2}}\mathfrak{J}_k Z_k \bar{Q}_k\biggr)\biggr) &+ \left(\frac{1}{2} - \frac{1}{4\sqrt{R}}\right)k^2 p^{-1}\left(|Z_k|^2 + |Q_k|^2\right)\\
        &\leq 16 \pi \frac{R}{\epsilon}  \mu p \left(|Z_k|^2 + |Q_k|^2\right)\\
        & \quad + \frac{\pi}{8\sqrt{R}}|k|^{3/2}p^{-3/4}\left(|Z_k|^2 + |Q_k|^2\right),
        \end{split}
    \end{equation}\label{alpha_tau_est}
        \begin{equation}
        \begin{split}
        \alpha_k\partial_t\biggl(\mathfrak{J}_k (\eta-kt)^2|Z_k|^2 + \mathfrak{J}_k(\eta-kt)^2|Q_k|^2 + &\mathrm{Re}\biggl(\frac{1}{2 \sqrt{R}} \frac{\partial_t p}{|k| p^{1/2}}\mathfrak{J}_k (\eta-kt)^2 Z_k \bar{Q}_k\biggr)\biggr)\\
        &+ \alpha_k\left(\frac{1}{2} - \frac{1}{4\sqrt{R}}\right)k^2 p (\eta-kt)^2\left(|Z_k|^2 + |Q_k|^2\right)\\
        &\quad \leq \alpha_k 16 \pi \frac{R}{\epsilon}  \mu p (\eta-kt)^2\left(|Z_k|^2 + |Q_k|^2\right)\\
        &\quad\quad  + \alpha_k\frac{\pi}{8\sqrt{R}}|k|^{3/2}p^{-3/4}(\eta-kt)^2\left(|Z_k|^2 + |Q_k|^2\right)\\
        & \quad\quad  + \frac{\pi}{4} D_{k,\beta}^{1/2} D_{k,\gamma}^{1/2}.
        \end{split}
    \end{equation}

\end{lemma}

\begin{proof}
    We begin with the proof of \eqref{tau_est}. Computing the time derivatives as in the proof of \eqref{gamma_est} and \eqref{time_derivs}, we find
\begin{equation}\label{lin_tau_main}
    \begin{split}
         \partial_t\biggl(\mathfrak{J}_k |Z_k|^2 + \mathfrak{J}_k|Q_k|^2 + &\mathrm{Re}\biggl(\frac{1}{2 \sqrt{R}} \frac{\partial_t p}{|k| p^{1/2}} \mathfrak{J}_k Z_k \bar{Q}_k\biggr)\biggr)  + \frac{1}{2}k^2 p^{-1}\left(|Z_k|^2 + |Q_k|^2\right) \\
        &\quad = -2 \nu p \mathfrak{J}_k |Z_k|^2 - 2\kappa p \mathfrak{J}_k |Q_k|^2\\
        &\quad \quad -\frac{1}{2} k^2 p^{-1} \mathrm{Re}\biggl(\frac{1}{2 \sqrt{R}} \frac{\partial_t p}{|k| p^{1/2}} Z_k \bar{Q}_k\biggr)  + \mathrm{Re}\biggl(\frac{1}{2 \sqrt{R}} \partial_t\biggl(\frac{\partial_t p}{|k| p^{1/2}}\biggr) \mathfrak{J}_k Z_k \bar{Q}_k\biggr)\\
        &\quad\quad -\frac{(\nu+\kappa)}{2\sqrt{R}} \frac{\partial_t p}{|k| p^{1/2}} p\mathfrak{J}_k \mathrm{Re}(Z_k \bar{Q_k}).
    \end{split}
\end{equation}
By Young's inequality, Cauchy-Schwarz, the assumption \eqref{diffusion_assumption}, the definition of $\mu$, and \eqref{boundedness}, \eqref{time_derivs}, we see
\begin{equation}\label{jk_diff_est}
    \begin{split}
        2 \nu p |\mathfrak{J}_k| |Z_k|^2 + 2\kappa |\mathfrak{J}_k| |Q_k|^2 + &\frac{(\nu+\kappa)}{2\sqrt{R}} \frac{\partial_t p}{|k| p^{1/2}} p|\mathfrak{J}_k| |\mathrm{Re}(Z_k \bar{Q_k})|\\
        &\leq \frac{\pi}{2} p \left( \left(\nu + \frac{\nu+ \kappa}{4\sqrt{R}}\right)   + \left(\kappa + \frac{\nu+ \kappa}{4\sqrt{R}}\right)|Q_k|^2 \right)\\
        &\leq 16 \pi \frac{R}{\epsilon} \mu p \left(|Z_k|^2 + |Q_k|^2\right).
    \end{split}
\end{equation}
Using \eqref{jk_diff_est}, we apply absolute values to the right-hand-side of \eqref{lin_tau_main}, use Young's inequality and \eqref{time_derivs} and obtain
\begin{equation}
\begin{split}
    \partial_t\biggl(\mathfrak{J}_k |Z_k|^2 + \mathfrak{J}_k|Q_k|^2 + \mathrm{Re}\biggl(\frac{1}{2 \sqrt{R}} \frac{\partial_t p}{|k| p^{1/2}}\mathfrak{J}_k Z_k \bar{Q}_k\biggr)\biggr) &+ \left(\frac{1}{2} - \frac{1}{4\sqrt{R}}\right)k^2 p^{-1}\left(|Z_k|^2 + |Q_k|^2\right)\\
        &\leq 16 \pi \frac{R}{\epsilon}  \mu p \left(|Z_k|^2 + |Q_k|^2\right)\\
        & \quad + \frac{\pi}{8\sqrt{R}}|k|^{3/2}p^{-3/4}\left(|Z_k|^2 + |Q_k|^2\right),
\end{split}
\end{equation}
giving \eqref{tau_est}. The proof of \eqref{alpha_tau_est} is similar to the proof of \eqref{tau_est}, with similar modifications to those found in the proof of \eqref{alpha_est}.
\end{proof}

We now complete the proof of Proposition \ref{linear_proposition}.
\begin{proof}
    By Lemma \ref{hypo_coerc}, Lemma \ref{j_lemma}, and the definition of $N_k$ \eqref{definition_of_n_k}, we have:
    \begin{equation}\label{first_hypo}
        \begin{split}
            \partial_t E_k + &2 D_{k,\gamma} + \frac{1}{2R} D_{k,\rho} + 2 c_\alpha D_{k, \alpha} + \frac{1}{2\sqrt{R}} c_\alpha D_{k,\rho \alpha} + \frac{c_\tau}{2}\left(1 - \frac{1}{2\sqrt{R}}\right) D_{k,\tau}\\
            &+ \frac{c_\tau c_\alpha}{2}\left(1 - \frac{1}{2\sqrt{R}}\right)D_{k, \tau \alpha}  + c_\beta D_{k,\beta} \leq c_\alpha\left(1 + \frac{1}{2\sqrt{R}}\right)e^{\frac{2}{2\sqrt{R}-1}}D_{k,\beta} +  c_\alpha\left(1 + \frac{1}{2R}\right)e^{\frac{2}{2\sqrt{R}-1}}D_{k,\gamma}\\
            &\quad + 16 \pi c_\tau  \frac{R}{\epsilon} D_{k,\gamma} + c_\tau \frac{\pi}{8\sqrt{R}} D_{k, \rho} +16 \pi c_\tau c_\alpha  \frac{R}{\epsilon}  D_{k,\alpha} + c_\tau c_\alpha \frac{\pi}{8\sqrt{R}} D_{k, \rho \alpha} + c_\tau c_\alpha \frac{\pi}{4}D_{k, \beta}^{1/2} D_{k,\gamma}^{1/2}\\
            &\quad +c_\beta\beta_k | \mathrm{Re} \frac{1}{\sqrt{R}} k(\eta-kt) \frac{|k|^3}{p^{3/4}} Z_k \bar{Q}_k | + 2 \nu c_\beta \beta_k p |k(\eta-kt)| |Z_k|^2\\
    &\quad + 2 \kappa c_\beta \beta_k p |k(\eta-kt)| |Q_k|^2 + c_\beta \beta_k \frac{\partial_t p}{|k| p^{1/2}} \frac{\nu + \kappa}{2\sqrt{R}} |k(\eta-kt)||\mathrm{Re}(Z_k \bar{Q}_k)|.
        \end{split}
    \end{equation}
    We impose the following smallness conditions on $c_\alpha, c_\beta$, and $c_\tau$:
\begin{equation}\label{initial_smallness}
        c_\tau \leq \frac{1}{32 \pi}\min\{\frac{1}{8}, \frac{\epsilon}{R} \}, \quad c_\alpha \left(1+\frac{1}{2\sqrt{R}}\right)e^{\frac{2}{2\sqrt{R}-1}} < \min\{\frac{1}{25} c_\beta, \frac{1}{2 \pi}\}.
    \end{equation}
    We note that by \eqref{initial_smallness}, Young's product inequality, and $R > 1/4$, we have
    \begin{equation}\label{first_youngs}
         c_\tau c_\alpha \frac{\pi}{4} D_{k, \beta}^{1/2} D_{k,\gamma}^{1/2} \leq \frac{1}{10} c_\beta D_{k,\beta} + \frac{5}{2} \frac{c_\tau^2 c_\alpha^2}{c_\beta}D_{k,\gamma} \leq \frac{1}{10} c_\beta D_{k,\beta} + \frac{5}{32}D_{k,\gamma}.
    \end{equation}
    Under the conditions in \eqref{initial_smallness}, the expression \eqref{first_hypo} implies (using \eqref{first_youngs}) the simpler expression
        \begin{equation}\label{second_hypo}
        \begin{split}
            \partial_t E_k + &\frac{7}{10} D_{k,\gamma} + \frac{1}{2\sqrt{R}} \frac{3}{4} D_{k,\rho} +  c_\alpha D_{k, \alpha} + \frac{1}{2\sqrt{R}} \frac{3}{4} c_\alpha D_{k,\rho \alpha} + \frac{c_\tau}{2}\left(1 - \frac{1}{2\sqrt{R}}\right) D_{k,\tau}\\
            & + \frac{c_\tau c_\alpha}{2}\left(1 - \frac{1}{2\sqrt{R}}\right)D_{k, \tau \alpha}  + \frac{2}{5}  c_\beta D_{k,\beta} \leq \frac{2}{5}c_\beta D_{k,\beta} + c_\beta\beta_k | \mathrm{Re} \frac{1}{\sqrt{R}} k(\eta-kt) \frac{|k|^3}{p^{3/4}} Z_k \bar{Q}_k|\\
            &\quad + 2 \nu c_\beta \beta_k p |k(\eta-kt)| |Z_k|^2 + 2 \kappa c_\beta \beta_k p |k(\eta-kt)| |Q_k|^2 +  c_\beta \beta_k \frac{\partial_t p}{|k| p^{1/2}} \frac{\nu + \kappa}{2\sqrt{R}} |k(\eta-kt)||\mathrm{Re}(Z_k \bar{Q}_k)|.
        \end{split}
    \end{equation}
    We now impose the additional criterion
\begin{equation}\label{second_smallness}
    \frac{c_\beta^2}{c_\alpha} < \frac{1}{16}.
\end{equation}
    Noting that $\beta_k |k| \leq \alpha_k$, we estimate using \eqref{second_smallness}
    \begin{equation}\label{second_youngs}
        \begin{split}
            c_\beta\beta_k | \mathrm{Re} \frac{1}{\sqrt{R}} k(\eta-kt) \frac{|k|^3}{p^{3/4}} Z_k \bar{Q}_k| & \leq c_\beta \beta_k \frac{1}{2 \sqrt{R}} \frac{|k|^3}{p^{3/4}}|\mathrm{Re}\left( \frac{\sqrt{c_\alpha c_\beta}}{\sqrt{c_\alpha c_\beta}}k Z_k (\eta-kt) \bar{Q}_k + \frac{\sqrt{c_\alpha c_\beta}}{\sqrt{c_\alpha c_\beta}}(\eta - kt) Z_k k \bar{Q}_k  \right) |\\
            &\leq  \beta_k \frac{1}{2 \sqrt{R}} \frac{|k|^3}{p^{3/4}}\left( \frac{c_\beta^2}{c_\alpha}\frac{k^2}{2}( |Z_k|^2 + |Q_k|^2) + c_\alpha \frac{(\eta-kt)^2}{2}(|Z_k|^2 + |Q_k|^2)\right)\\
            &\leq \frac{1}{4} \frac{1}{2\sqrt{R}} D_{k,\rho} + \frac{1}{4} \frac{1}{2 \sqrt{R}} c_\alpha D_{k, \rho \alpha}.
        \end{split}
    \end{equation}
    Now we treat the final terms using Young's convolution inequality, \eqref{time_derivs}, the assumption \eqref{diffusion_assumption}, and the definition of $\mu$ \eqref{mu_definition}:
    \begin{equation}\label{third_youngs}
        \begin{split}
        2 c_\beta \beta_k |k||\eta-kt| p &\biggl(   \nu |Z_k|^2 + \kappa |Q_k|^2 + \frac{\nu + \kappa}{2 \sqrt{R}} \mathrm{Re}(Z_k \bar{Q}_k)\biggr)\\
        &\leq 2 c_\beta |\eta-kt| p \mu^{-1} \biggl(   \nu \alpha_k^{1/2} |Z_k|^2 + \kappa  \alpha_k^{1/2} |Q_k|^2 + \frac{\nu + \kappa}{2 \sqrt{R}} \alpha_k^{1/2}|Z_k||Q_k|\biggr)\\
        &\leq 2 c_\beta p \biggl( \frac{\nu}{2} \biggl(\frac{\epsilon}{32 R c_\beta} |Z_k|^2 + \alpha_k\frac{32c_\beta R^2}{\epsilon}(\eta-kt)^2 |Z_k|^2 \biggr) \\
        &\quad \quad \quad \quad + \frac{\kappa}{2} \biggl(\frac{\epsilon}{32 R c_\beta} |Q_k|^2 + \alpha_k \frac{32c_\beta R^2}{\epsilon}(\eta-kt)^2 |Q_k|^2 \biggr)\biggr)\\
        &\quad \quad \quad \quad + \frac{\nu + \kappa}{8 \sqrt{R}}\biggl(\frac{\epsilon}{32 R c_\beta}(|Z_k|^2 + |Q_k|^2) + \frac{32 R c_\beta}{\epsilon}(\eta-kt)^2(|Z_k|^2 + |Q_k|^2)\biggr)\\
        &\leq \frac{\epsilon}{32 R c_\beta} p \biggl(\biggl( \nu + \frac{\nu + \kappa}{4 \sqrt{R}}\biggr)|Z_k|^2 +  \biggl( \kappa + \frac{\nu + \kappa}{4 \sqrt{R}}\biggr)|Q_k|^2\biggr)\\
        &\quad \quad \quad + \frac{32 c_\beta^2 R^2}{\epsilon} \alpha_k (\eta-kt)^2 \biggl(\biggl( \nu + \frac{\nu + \kappa}{4 \sqrt{R}}\biggr)|Z_k|^2 +  \biggl( \kappa + \frac{\nu + \kappa}{4 \sqrt{R}}\biggr)|Q_k|^2\biggr)\\
        &\leq \frac{1}{2} D_{k, \gamma} + \frac{512 c_\beta^2 R^2 }{\epsilon^2}.
        \end{split}
    \end{equation}
 Our final smallness assumption will be therefore be
 \begin{equation}\label{third_smallness}
        \frac{c_\beta^2}{c_\alpha} < \frac{\epsilon^2}{1024 R^2}.
    \end{equation}
    Applying \eqref{second_youngs}, \eqref{third_youngs}, and \eqref{third_smallness} to \eqref{second_hypo}, we arrive at
    \begin{equation}\label{third_hypo}
        \begin{split}
            \partial_t E_k + &\frac{1}{5} D_{k,\gamma} + \frac{1}{2\sqrt{R}} \frac{1}{2} D_{k,\rho} +  \frac{1}{2} c_\alpha D_{k, \alpha} + \frac{1}{2\sqrt{R}} \frac{1}{2} c_\alpha D_{k,\rho \alpha} + \frac{c_\tau}{2} \left(1 - \frac{1}{2\sqrt{R}}\right)D_{k,\tau} + \frac{c_\tau c_\alpha}{2}\left(1 - \frac{1}{2\sqrt{R}}\right)D_{k, \tau \alpha}\\
            &\quad + \frac{2}{5}  c_\beta D_{k,\beta} \leq \frac{2}{5}  c_\beta D_{k,\beta}.
        \end{split}
    \end{equation}
    Next, we observe that $\beta_k |k|^2 = \lambda_k$ and $\alpha_k \lambda_k \leq \nu$. Additionally, \eqref{third_smallness} implies that $c_\beta^2 \leq \frac{c_\alpha}{4} + \frac{(1-c_\tau)}{4}$. Thus 
    $$\lambda_k E_k \approx_R \lambda_k(1 + c_\alpha \alpha(\eta-kt)^2)(|Z_k|^2 + |Q_k|^2) \lesssim c_\beta D_{k,\beta} + D_{k,\gamma}.$$
    We therefore obtain
        \begin{equation}\label{fourth_hypo}
        \begin{split}
            \partial_t E_k + &\frac{1}{10} D_{k,\gamma} + \frac{1}{2\sqrt{R}} \frac{1}{2} D_{k,\rho} +  \frac{1}{2} c_\alpha D_{k, \alpha} + \frac{1}{2\sqrt{R}} \frac{1}{2} c_\alpha D_{k,\rho \alpha} + \frac{c_\tau}{2} \left(1 - \frac{1}{2\sqrt{R}}\right)D_{k,\tau} + \frac{c_\tau c_\alpha}{2}\left(1 - \frac{1}{2\sqrt{R}}\right)D_{k, \tau \alpha}\\
            &\quad + \frac{1}{5}  c_\beta D_{k,\beta} + \lambda_k E_k\leq \frac{2}{5}  c_\beta D_{k,\beta}.
        \end{split}
    \end{equation}
    yielding Proposition \ref{linear_proposition} as desired.
\end{proof}
\begin{remark}
Gathering the smallness assumptions \eqref{initial_smallness}, \eqref{second_smallness}, and \eqref{third_smallness}, we state all assumptions together as
\begin{equation}\label{all_assumptions}
    c_\tau \leq \frac{1}{32 \pi }\min\{\frac{1}{8}, \frac{\epsilon}{R} \}, \quad c_\alpha \left(1+\frac{1}{2\sqrt{R}}\right)e^{\frac{2}{2\sqrt{R}-1}} < \min\{\frac{1}{25} c_\beta, \frac{1}{2 \pi}\}, \quad \frac{c_\beta^2}{c_\alpha} < \min\{\frac{\epsilon^2}{1024 R^2}, \frac{1}{16}\}. 
\end{equation}
We note that $R > 1/4$ and $\epsilon \in (0,1/2)$. Hence the following is an explicit choice of constants satisfying \eqref{all_assumptions}:
\begin{equation}
\begin{split}
     &c_\tau = \frac{1}{32 \pi }\min\{\frac{1}{8}, \frac{\epsilon}{R} \}, \quad c_\alpha = \frac{1}{50,000} \frac{\epsilon^2}{1024 R^2}\left(1+\frac{1}{2\sqrt{R}}\right)^{-2}e^{-\frac{4}{2\sqrt{R}-1}},\\ &c_\beta = \frac{1}{100} \frac{\epsilon^2}{1024 R^2}\left(1+\frac{1}{2\sqrt{R}}\right)^{-1/2}e^{-\frac{1}{2\sqrt{R}-1}}.
\end{split}
\end{equation}
\end{remark}

\subsection{Linear Estimates in Original Variables}\label{linear_results}
Just as Theorem \eqref{symmetric_theorem} implies estimates for the original system \eqref{vorticity_perturb_eqn} in Theorem \ref{main_theorem}, the pointwise-in-frequency estimates in Proposition \ref{linear_proposition} imply certain corollaries for the untransformed linear system, which we write as
\begin{equation}\label{2D_linearized_vorticity}
\begin{cases}
    \partial_t \omega + y \partial_x \omega = \nu \Delta \omega - R \partial_x \theta,\\
    \partial_t \theta + \partial_x \theta = \kappa \Delta \theta + \partial_x \psi,\\
    \Delta \psi = \omega.
\end{cases}
\end{equation}
Our first lemma is an immediate consequence of Proposition \ref{linear_proposition}, and follows from undoing the change $Z \to \Omega$, $Q \to \Theta$, and noticing that multiplication by $\langle k \rangle^{1/2}$ commutes with the linearized system.
\begin{lemma}\label{inviscid_estimate_pointwise_2D}
    Let $\nu = \kappa = 0$. Then the solution to \eqref{2D_linearized_vorticity} in the moving frame satisfies \begin{equation}\label{inviscid_pointwise_estimate}
       |p^{-1/4} \Omega_k(t)|^2 + |p^{1/4} \Theta_k (t)|^2 \approx |(k^2 + \eta^2)^{-1/4} \Omega_k(0)|^2 + |(k^2 + \eta^2)^{1/4} \Theta_k(0)|^2.
    \end{equation}
    Now suppose $\nu ,\kappa$ satisfy \eqref{diffusion_assumption}. Then there exist constants $c_1 = c_1(R,\epsilon) > 0$, $c_2 = c_2(R,\epsilon) > 0$  such that the solution to \eqref{2D_linearized_vorticity} in the moving frame satisfies
\begin{equation}\label{viscous_pointwise}
    \begin{split}
            |p^{-1/4} &\left(1 + c_2 \alpha_k i(\eta -kt) \right)\Omega_k(t)|^2 + |p^{1/4} \left(1+ c_2 \alpha_k i(\eta -kt) \right)\Theta_k (t)|^2\\ &\quad \lesssim \exp\left(-c_1\lambda_k t\right)|(k^2 + \eta^2)^{-1/4}\left(1 + c_2 \alpha_k i\eta\right) \Omega_k(0)|^2 + |(k^2 + \eta^2)^{1/2} \left(1 + c_2\alpha_k i\eta\right) \Theta_k(0)|^2.
    \end{split}
    \end{equation}
\end{lemma}
The inviscid pointwise estimate \eqref{inviscid_pointwise_estimate} is known to hold on the domain $\mathbb{T} \times \R$ (see \cite{bianchini_cotizelati_dolce}). In fact, one can use the techiniques of \cite{bianchini_cotizelati_dolce} to arrive at \eqref{inviscid_pointwise_estimate} directly. Once one has \eqref{inviscid_pointwise_estimate}, the following corollary can be obtained by undoing the change to the moving reference, as done in Section \ref{proof_of_main} (see also \cite{bianchini_cotizelati_dolce} where this is done more explicitly in the $\mathbb{T} \times \R$ case).
\begin{corollary}[Linear Inviscid Damping and Vorticity Growth]\label{inviscid_damping_and_vorticity_growth}
    Suppose that $\nu = \kappa = 0$, and that $\omega, \theta$ solve \eqref{2D_linearized_vorticity}. Then for $d = 0, 1, 2$ we have the inviscid damping estimates on $u$ and $\theta$:
    \begin{equation}
        \begin{split}
            &||u^1(t)||_{\dot{H}^{1}_x L^2_y} +||\langle t \rangle^{1/2} \theta(t)||_{\dot{H}^{1}_x L^2_y} \lesssim \langle t \rangle^{-1/2} \left(||\omega_{in}||_{L^2_{x,y}}+ ||\nabla \theta_{in}||_{L^2_{x,y}}\right),\\
            &||u^2(t)||_{\dot{H}_x^{1 + d/2}L^2_y} \lesssim \langle t \rangle^{-(1+d)/2}\left( |||\nabla|^{d/2}\omega_{in}||_{L^2_{x,y}} + |||\nabla|^{1+d/2}\theta_{in}||_{L^2_{x,y}}\right).
        \end{split}
    \end{equation}
    Meanwhile, we have the following inviscid upper and lower bounds on the growth of the vorticity and $\nabla \theta$:
    \begin{equation}
        \begin{split}
            \langle t \rangle^{1/2}\left(|| \partial_x u_{in} ||_{L^2_{x,y}}+ ||\partial_x \theta_{in}||_{L^2_{x,y}}\right)& \lesssim||\omega(t)||_{L^2_{x,y}} + ||\nabla \theta(t)||_{L^2_{x,y}}\lesssim \langle t \rangle^{1/2} \left(||\omega_{in}||_{L^2_{x,y}} + ||\nabla \theta_{in}||_{L^2_{x,y}}\right).
        \end{split}
    \end{equation}
\end{corollary}
For the viscous case, we obtain the following corrolary, which combines the inviscid damping, enhanced dissipation, and Taylor dispersion, expressed in the hypocoercive norm.

\begin{corollary}[Linear Enhanced Dissipation, Taylor Dispersion, Inviscid Damping, and Transient Vorticity Growth]\label{enhanced_dissipation_linear}
    Suppose that $\nu, \kappa$ satisfy \eqref{diffusion_assumption} and $(\omega, \theta)$ satisfy \eqref{2D_linearized_vorticity}. Then for $d = 0, 1, 2$, there exist constants $c_1, c_2 > 0$ independent of $\nu$ and $\kappa$, but depending on $R$ and $\epsilon$ such that
    \begin{equation}
        \begin{split}
            &\sum_{j=0}^1|| e^{c_1 \lambda_k t} c_2^j \langle \frac{\partial_x}{\mu} \rangle^{-j/3} \partial_y^j u^1(t)||_{\dot{H}^{1}_x L^2_y} +||  e^{c_1 \lambda_k t} c_2^j \langle \frac{\partial_x}{\mu} \rangle^{-j/3} \partial_y^j \theta(t)||_{\dot{H}^{1}_x L^2_y} \lesssim \langle t \rangle^{-1/2}\left(||\omega_{in}||_{L^2_{x,y}}+ ||\nabla \theta_{in}||_{L^2_{x,y}}\right),\\
            &\sum_{j=0}^1|| e^{c_1 \lambda_k t} c_2^j \langle \frac{\partial_x}{\mu} \rangle^{-j/3} \partial_y^j u^2(t)||_{\dot{H}_x^{1 + d/2}L^2_y} \lesssim  \langle t  \rangle^{(1+d)/2}\left(|||\nabla|^{d/2}\omega_{in}||_{L^2_{x,y}} + |||\nabla|^{1+d/2}\theta_{in}||_{L^2_{x,y}}\right),\\
            &\sum_{j=0}^1|| e^{c_1 \lambda_k t} c_2^j \langle \frac{\partial_x}{\mu} \rangle^{-j/3} \partial_y^j \omega(t)||_{L^2_{x,y}} + || e^{c_1 \lambda_k t} c_2^j \langle \frac{\partial_x}{\mu} \rangle^{-j/3} \partial_y^j \nabla \theta(t)||_{L^2_{x,y}} \lesssim \langle t \rangle^{1/2} \left(||\omega_{in}||_{L^2_{x,y}} + ||\nabla \theta_{in}||_{L^2_{x,y}}\right).
        \end{split}
    \end{equation}
\end{corollary}
Lastly, we observe that in the linear case, all of these estimates are based off of Lemma \ref{inviscid_estimate_pointwise_2D}. Multiplying both sides of either \eqref{inviscid_pointwise_estimate} or \eqref{viscous_pointwise} by a Fourier multiplier would enable one to write different estimates. In particular, on the Fourier side one could write a variety of $L^p$ estimates.

\section{Non-linear Computations}\label{nonlinear_section}

To complete the proof of Theorem \ref{symmetric_theorem}, it suffices to prove Lemma \ref{bootstrap_lemma}. Due to our work in Section \ref{linear_section} on the linearized estimates, our remaining work will focus on estimating the terms in the energy arising from the nonlinearity in \eqref{Symmetrized_moving_reference_equations}. To be precise, we define, in a similar manner as \cites{arbon_bedrossian, bedrossian2023stability}:
\begin{equation}\label{lin_and_nonlin}
    \begin{split}
        & \mathbb{L}_k^{(Z)} \coloneqq \frac{1}{4}\frac{\partial_t p}{p}Z_k - \nu p Z_k - \sqrt{R} |k|p^{-1/2} Q_k,\\
        &\mathbb{L}_k^{(Q)} \coloneqq \frac{1}{4}\frac{\partial_t p}{p}Q_k-\kappa p Q_k + \sqrt{R} |k| p^{-1/2} Z_k,\\
        &\mathbb{NL}_k^{(Z)} \coloneqq - \langle k \rangle^{1/2}p^{-1/4}\left(U \cdot \nabla_t \Omega\right)_k,\\
        &\mathbb{NL}_k^{(Q)} \coloneqq - \sqrt{R} i \sgn(k) \langle k \rangle^{1/2} p^{1/4} \left(U \cdot \nabla_t\Theta\right)_k.
    \end{split}
\end{equation}

Using the definition of $\E$ \eqref{nonlinear_energy}, the definition of $M_k(t)$ \eqref{correction_ODE}, and the definitions above \eqref{lin_and_nonlin}, we compute
\begin{equation}\label{nonlinear_time_deriv}
    \begin{split}
        \frac{d}{dt} \E[Z, Q] &= \iint_{\R^2} \langle k, \eta\rangle^{2n} \langle k \rangle^{2m} \frac{d}{dt} \left( \frac{\dk^{2J}}{M_k(t)} E_k[Z_k, Q_k]\right) dk d\eta\\
        &= \mathcal{L}_a + \mathcal{L}_b + \mathcal{NL} + \iint_{\R^2} \langle k, \eta \rangle^{2m}\left( J c \lambda_k \frac{2 c \lambda_k t}{\langle c \lambda_k t \rangle^2} \frac{\dk^{2J}}{M_k(t)} -\frac{\dot{M}_k(t)}{M_k(t)} \frac{\dk^{2J}}{M_k(t)}\right) E_k dk d\eta,
    \end{split}
\end{equation}
where
\begin{equation}
    \begin{split}
        \mathcal{L}_a &= \iint_{\R^2} \langle k, \eta\rangle^{2n} \langle k \rangle^{2m} \frac{\dk^{2J}}{M_k(t)} \biggl( \partial_t\left(N_k + c_\tau \mathfrak{J}_k\right) |Z_k|^2 + \partial_t\left(N_k + c_\tau \mathfrak{J}_k\right) |Q_k|^2 + \frac{1}{2\sqrt{R}}\mathrm{Re}\left(\partial_t \biggl(\frac{\partial_t p}{ |k| p^{1/2}}\left(N_k + \mathfrak{J}_k\right)\biggr)  Z_k\bar{Q}_k \right)\\
 & \quad \quad \quad + \partial_t\biggl(\left(N_k + c_\tau \mathfrak{J}_k\right) (\eta - kt)^2 \biggr) \biggl(|Z_k|^2 + |Q_k|^2\biggr)+ \frac{c_\alpha \alpha_k}{2\sqrt{R}}\partial_t\biggl(\frac{\partial_t p}{ |k| p^{1/2}} \left(N_k + \mathfrak{J}_k\right)  (\eta - kt)^2\biggr)\mathrm{Re}\left( Z_k\bar{Q}_k \right)\biggr) \\
 &\quad \quad \quad+ c_\beta \beta_k \biggl( k\partial_t(\eta -kt) |Z_k|^2 +k\partial_t(\eta -kt) |Q_k|^2 - \frac{1}{2\sqrt{R}} \mathrm{Re}\left(\partial_t\biggl(\frac{\partial_t p}{|k| p^{1/2}}k(\eta-kt)\biggr) Z_k \bar{Q}_k\right)\biggr) \biggr) dk d\eta,
    \end{split}
\end{equation}
\begin{equation}\label{def_of_linear_time_deriv}
    \begin{split}
        \mathcal{L}_b &\coloneqq \iint_{\R^2} \frac{\dk^{2J}}{M_k(t)} \langle k, \eta\rangle^{2n} \langle k \rangle^{2m} \biggl( (N_k + c_\tau\mathfrak{J}_k)\biggl(2\mathrm{Re}(\mathbb{L}_k^{(Z)} \bar{Z}_k) + 2\mathrm{Re}(\mathbb{L}_k^{(Q)} \bar{Q}_k) \frac{1}{2\sqrt{R}}\frac{\partial_t p}{|k|p^{1/2}}\biggl( \mathrm{Re}(\mathbb{L}_k^{(Z)} \bar{Q}_k) + \mathrm{Re}(\bar{Z}_k \mathbb{L}_k^{(Q)}\biggr)\biggr)\\
        &\quad \quad \quad + c_\alpha \alpha_k(N_k + c_\tau \mathfrak{J}_k)(\eta-kt)^2\biggl(2\mathrm{Re}(\mathbb{L}_k^{(Z)} \bar{Z}_k) + 2\mathrm{Re}(\mathbb{L}_k^{(Q)} \bar{Q}_k) + \frac{1}{2\sqrt{R}}\frac{\partial_t p}{|k|p^{1/2}}\biggl( \mathrm{Re}(\mathbb{L}_k^{(Z)} \bar{Q}_k) + \mathrm{Re}(\bar{Z}_k \mathbb{L}_k^{(Q)}\biggr) \biggr)\\
        &\quad \quad \quad + c_\beta \beta_k k(\eta-kt) \biggl( 2\mathrm{Re}(\mathbb{L}_k \bar{Z}_k) + 2\mathrm{Re}(\mathbb{L}_k^{(Q)} \bar{Q}_k) - \frac{1}{2\sqrt{R}}\frac{\partial_t p}{|k|p^{1/2}}\biggl( \mathrm{Re}(\mathbb{L}_k^{(Z)} \bar{Q}_k) + \mathrm{Re}(\bar{Z}_k \mathbb{L}_k^{(Q)}\biggr) \biggr) dk d\eta,
    \end{split}
\end{equation}
and $\mathcal{NL}$ is defined analogously to \eqref{def_of_linear_time_deriv} with $\mathbb{NL}_k^{(Z)}$ and $\mathbb{NL}_k^{(Q)}$ replacing $\mathbb{L}_k^{(Z)}$ and $\mathbb{L}_k^{(Q)}$, respectively. By Proposition \ref{linear_proposition}, we have
\begin{equation}\label{non_linear_set_up_with_linear}
    \mathcal{L}_a + \mathcal{L}_b \leq - 8 c \D - 8 c \int_{\R^2} \lambda_k \frac{\dk^{2J}}{M_k(t)} \langle k, \eta\rangle^{2n} \langle k \rangle^{2m}E_k[Z_k, Q_k] dk d\eta.
\end{equation}
Using Young's product inequality, as shown in \cite{arbon_bedrossian}, one can show that
\begin{equation}\label{non_linear_set_up_with_youngs}
\begin{split}
        \iint_{\R^2} \langle k, \eta\rangle^{2n} \langle k \rangle^{2m} &\left( J c \lambda_k \frac{2 c \lambda_k t}{\langle c \lambda_k t \rangle^2} \frac{\dk^{2J}}{M_k(t)} -\frac{\dot{M}_k(t)}{M_k(t)} \frac{\dk^{2J}}{M_k(t)}\right)  E_k[Z_k,Q_k] dk d\eta\\
        &\leq \iint_{\R^2} c \lambda_k \langle k, \eta\rangle^{2n} \langle k \rangle^{2m} \frac{\dk^{2J}}{M_k(t)}  E_k[Z_k,Q_k] dk d\eta.
\end{split}
\end{equation}
Taking \eqref{non_linear_set_up_with_linear} and \eqref{non_linear_set_up_with_youngs} applied to \eqref{nonlinear_time_deriv}, we arrive at
\begin{equation}\label{nonlin_deriv}
    \frac{d}{dt}\mathcal{E} \leq - 4c \D + \mathcal{NL}.
\end{equation}
We have therefore reduced the proof of Lemma \ref{bootstrap_lemma} to showing
\begin{equation}
    \mathcal{NL} \lesssim \mu^{-1/2-\delta_*} \E^{1/2} \D.
\end{equation}
We begin by decomposing $\mathcal{NL}$ into
\begin{equation}
    \mathcal{NL} = T_\gamma + T_\alpha + T_\beta,
\end{equation}
where we define
\begin{equation}\label{T_main}
    \begin{split}
        T_\gamma & \coloneqq \iint_{\R^2} \frac{\dk^{2J}}{M_k(t)} \langle k, \eta\rangle^{2n} \langle k \rangle^{2m} (N_k + c_\tau\mathfrak{J}_k)\biggl(2\mathrm{Re}(\mathbb{NL}_k^{(Z)} \bar{Z}_k) + 2\mathrm{Re}(\mathbb{NL}_k^{(Q)} \bar{Q}_k)\\
        &\quad \quad \quad  \quad \quad + \frac{1}{2\sqrt{R}}\frac{\partial_t p}{|k|p^{1/2}}\biggl( \mathrm{Re}(\mathbb{NL}_k^{(Z)} \bar{Q}_k) + \mathrm{Re}(\bar{Z}_k \mathbb{NL}_k^{(Q)})\biggr) \biggr) dk d \eta \\
        &\eqqcolon T_{\gamma, Z} + T_{\gamma, Q} + T_{\gamma, m_1} + T_{\gamma, m_2},\\
                T_\alpha &\coloneqq c_\alpha \iint_{\R^2} \frac{\dk^{2J}}{M_k(t)} \langle k, \eta\rangle^{2n} \langle k \rangle^{2m} (N_k + c_\tau\mathfrak{J}_k)(\eta-kt)^2\biggl(2\mathrm{Re}(\mathbb{NL}_k^{(Z)} \bar{Z}_k) + 2\mathrm{Re}(\mathbb{NL}_k^{(Q)} \bar{Q}_k)\\
        &\quad \quad \quad  \quad \quad + \frac{1}{2\sqrt{R}}\frac{\partial_t p}{|k|p^{1/2}}\biggl( \mathrm{Re}(\mathbb{NL}_k^{(Z)} \bar{Q}_k) + \mathrm{Re}(\bar{Z}_k \mathbb{NL}_k^{(Q)})\biggr) \biggr) dk d \eta\\
        &\eqqcolon T_{\alpha, Z} + T_{\alpha, Q} + T_{\alpha, m_1} + T_{\alpha, m_2},\\
        T_\beta &\coloneqq c_\beta \iint_{\R^2} \frac{\dk^{2J}}{M_k(t)} \langle k, \eta\rangle^{2n} \langle k \rangle^{2m} \beta_k k (\eta-kt) \biggl(2\mathrm{Re}(\mathbb{NL}_k^{(Z)} \bar{Z}_k) + 2\mathrm{Re}(\mathbb{NL}_k^{(Q)} \bar{Q}_k)\\
        &\quad \quad \quad  \quad \quad - \frac{1}{2\sqrt{R}}\frac{\partial_t p}{|k|p^{1/2}}\biggl( \mathrm{Re}(\mathbb{NL}_k^{(Z)} \bar{Q}_k) + \mathrm{Re}(\bar{Z}_k \mathbb{NL}_k^{(Q)})\biggr) \biggr) dk d \eta \\
 &\eqqcolon T_{\beta, Z} + T_{\beta, Q} + T_{\beta, m_1} + T_{\beta, m_2},
    \end{split}
\end{equation}
So it suffices to show following lemma:
\begin{lemma}\label{key_lemma}
    For $T_\gamma$, $T_\alpha$, and $T_\beta$ as defined in \eqref{T_main}, we have
\begin{equation}\label{bootstrap_bound}
        |T_\gamma| + |T_\alpha| + |T_\beta| \lesssim \mu^{-1/2 - \delta_*} \D \E^{1/2},
    \end{equation}
    where the implicit constant is independent of $\nu, \kappa$, and $\mu$, but is allowed to depend on $\delta_*$, $\epsilon$, $m$, $J$, and $R$.
\end{lemma}
The remainder of this paper will be dedicated to the proof of Lemma \ref{key_lemma}. Before giving the details, we introduce some general notation which we will employ in the following sections.\\

Each of the $T_*$, $* \in \{\gamma, \alpha, \beta\}$, contains expressions $\mathbb{NL}_k^{(Z)}$ and $\mathbb{NL}_k^{(Q)}$, which on the Fourier side are properly expressed as convolutions in $k$ and $\eta$. To keep track of the convolutions, we introduce the notation $\xi \coloneqq (k, \eta)$. Furthermore, the Fourier multiplier $p = k^2 + (\eta-kt)^2$ will be found in various parts of the convolution. We therefore define
\begin{equation}
    p_\xi \coloneqq k^2 + (\eta-kt)^2, \; \; p_{\xi'} \coloneqq (k')^2 + (\eta' -k't)^2, \; \; p_{\xi - \xi'} \coloneqq (k-k')^2 + ((\eta-\eta') - (k-k')t)^2.
\end{equation}
Furthermore, we make a slight adjustment to the notation for $Z$ and $Q$. Until now, we have written $Z_k$ and $Q_k$ for the Fourier transforms of $Z$ and $Q$, respectively. We will now write $Z_\xi$ and $Q_\xi$, since we will be dealing with terms arising from convolutions in both $k$ and $\eta$. With this notation, we write out explicitly
\begin{equation}\label{conv_NL_Z}
    \begin{split}
        \mathbb{NL}_k^{(Z)} & = -\sk^{1/2} p_\xi^{-1/4} (U \cdot \nabla_t \Omega)_k\\
        &= -\sk^{1/2} p_\xi^{-1/4} \iint_{\R^2} |k-k'|^{-1/2} p_{\xi -\xi'}^{-3/4} i (k-k') Z_{\xi-\xi'} |k'|^{-1/2} p_{\xi'}^{1/4} i(\eta'-k't) Z_{\xi'}d\eta' dk'\\
        & \quad + \sk^{1/2} p_\xi^{-1/4} \iint_{\R^2} |k-k'|^{-1/2} p_{\xi -\xi'}^{-3/4} i ((\eta-\eta') - (k-k')t) Z_{\xi-\xi'} |k'|^{-1/2} p_{\xi'}^{1/4} i k' Z_{\xi'}d\eta' dk',
    \end{split}
\end{equation}
and
\begin{equation}\label{conv_NL_Q}
    \begin{split}
        \mathbb{NL}_k^{(Q)} & = -\sqrt{R} i \sgn(k)\sk^{1/2} p_\xi^{1/4} (U \cdot \nabla_t \Theta)_k\\
        &=  - \sgn(k)\sk^{1/2} p_\xi^{1/4}  \iint_{\R^2} |k-k'|^{-1/2} p_{\xi -\xi'}^{-3/4} i (k-k') Z_{\xi-\xi'} k'^{-1} |k'|^{1/2} p_{\xi'}^{-1/4} i(\eta'-k't) Q_{\xi'}d\eta' dk'\\
        & \quad + \sgn(k)\sk^{1/2} p_\xi^{1/4} \iint_{\R^2} |k-k'|^{-1/2} p_{\xi -\xi'}^{-3/4} i ((\eta-\eta') - (k-k')t) Z_{\xi-\xi'} k'^{-1} |k'|^{1/2} p_{\xi'}^{-1/4} i k' Q_{\xi'}d\eta' dk'.
    \end{split}
\end{equation}
Using \eqref{conv_NL_Z} and \eqref{conv_NL_Q}, we will generally break $T_{*,Z}$ into $T_{*, Z}^x$ and $T_{*, Z}^y$, and we will break $T_{*, Q}$ into $T_{*, Q}^x$ and $T_{*, Q}^y$. The terms $T_{*, Z}^x$ and $T_{*, Q}^x$ will contain the factors $i(k-k')Z_{\xi-\xi'}$ (meaning the $x$-derivatives fall on $Z_{\xi-\xi'}$), while  $T_{*, Z}^x$ and $T_{*, Q}^x$ will contain the factors $i((\eta-\eta') - (k-k')t)Z_{\xi-\xi'}$ (meaning the $y$-derivatives fall on $Z_{\xi-\xi'}$). For example, we have
\begin{equation}
    \begin{split}
    T_{\gamma, Z}^x &= - \iint_{\R^2} \frac{\dk^{2J}}{M_k(t)} \langle \xi \rangle^{2n}\langle k \rangle^{2m}  (N_k + c_\tau\mathfrak{J}_k)\mathrm{Re}\biggl(\iint_{\R^2} |k|^{1/2} p_\xi^{-1/4} \bar{Z}_\xi\\
    &\quad \quad \quad \quad \quad \quad |k-k'|^{-1/2} p_{\xi -\xi'}^{-3/4} i (k-k') Z_{\xi-\xi'} |k'|^{-1/2} p_{\xi'}^{1/4} i(\eta'-k't) Z_{\xi'}d\eta' dk'\biggr) d\eta dk
    \end{split}
\end{equation}
and
\begin{equation}
    \begin{split}
    T_{\gamma, Z}^y &= \iint_{\R^2} \frac{\dk^{2J}}{M_k(t)} \langle \xi \rangle^{2n}\langle k \rangle^{2m}  (N_k + c_\tau\mathfrak{J}_k)\mathrm{Re}\biggl(\iint_{\R^2} |k|^{1/2} p_\xi^{1/4} \bar{Z}\xi\\
    &\quad \quad \quad \quad \quad \quad |k-k'|^{-1/2} p_{\xi -\xi'}^{-3/4} i ((\eta-\eta') - (k-k')t) Z_{\xi-\xi'} |k'|^{-1/2} p_{\xi'}^{1/4} i k' Z_{\xi'}d\eta' dk'\biggr) d\eta dk,
    \end{split}
\end{equation}
so that $T_{\gamma, Z} = T_{\gamma, Z}^x + T_{\gamma, Z}^y$. We will call $T_{*, Z}^x$ and  $T_{*, Z}^y$ the $x$-\textit{derivatives} and the $y$-\textit{derivatives}, respectively, with similar naming for the breakdown of the $T_{*,Q}$ terms.
Additionally, we will not explicitly compute the bounds on the terms $T_{*, m_1}$ and $T_{*, m_2}$ of mixed-type. The arguments for these follow identically from the arguments used on $T_{*, Z}$ and $T_{*, Q}$. For additional information, see Section \ref{mixed_gamma_terms}.

Throughout our proof of Lemma \ref{key_lemma} we will continually make decompositions of the $T_*$ in frequency space. The most common of these will separate ``low-high" terms from ``high-low" terms. The ``low-high", or $LH$ terms, will occur when $|k-k'| < |k'|/2$, while the ``high-low" or $HL$ terms correspond to the case $|k-k'| \geq |k'|/2$. We will employ the following shorthand for the corresponding characteristic functions:
\begin{equation}
    1_{LH} \coloneqq 1_{|k-k'| < |k'|/2}, \;\;\; 1_{HL} \coloneqq 1_{|k-k'| \geq |k'|/2}.
\end{equation}
Lastly, we introduce additional notation for further restrictions on the domain of integration. We use the pairs $(A,B)$ as a subscript. A term containing $(A,B)$, for $A,B \in \{H,L,\cdot\}$ indicates that the $k$ frequencies are at frequency-type $A$ and the $k'$ frequencies are of type $B$. The symbol $H$ refers to frequencies which have magnitude $\geq \mu$, and the symbol $L$ refers to frequencies with magnitude $< \mu$. The symbol $\cdot$ is used to represent no restrictions on a frequency.
To see how this is used in practice, alongside other minor notation clarifications, see Section \ref{gamma_terms_for_Z}. We are now ready to begin the proof of Lemma \ref{key_lemma}.
\subsection{Gamma Terms for Z}\label{gamma_terms_for_Z}

We recall that
\begin{equation}
\begin{split}
T_{\gamma,Z} &=  - \iint_{\R^2} \frac{\dk^{2J}}{M_k(t)} \langle \xi \rangle^{2n}\langle k \rangle^{2m}  (N_k + c_\tau\mathfrak{J}_k)\mathrm{Re}\biggl(\iint_{\R^2} \langle k \rangle^{1/2} p_\xi^{-1/4} \bar{Z}_\xi\\
    &\quad \quad \quad \quad \quad \quad \langle k-k' \rangle^{-1/2} p_{\xi -\xi'}^{-3/4} i (k-k') Z_{\xi-\xi'} \langle k' \rangle^{-1/2} p_{\xi'}^{1/4} i(\eta'-k't) Z_{\xi'}d\eta' dk'\biggr) d\eta dk\\
    &\quad +\iint_{\R^2} \frac{\dk^{2J}}{M_k(t)} \langle \xi \rangle^{2n}\langle k \rangle^{2m}  (N_k + c_\tau\mathfrak{J}_k)\mathrm{Re}\biggl(\iint_{\R^2} \langle k \rangle^{1/2} p_\xi^{1/4} \bar{Z}_\xi\\
    &\quad \quad \quad \quad \quad \quad \langle k-k' \rangle^{-1/2} p_{\xi -\xi'}^{-3/4} i ((\eta-\eta') - (k-k')t) Z_{\xi-\xi'} \langle k'\rangle^{-1/2} p_{\xi'}^{1/4} i k' Z_{\xi'}d\eta' dk'\biggr) d\eta dk\\
    &\eqqcolon T_{\gamma, Z}^x + T_{\gamma, Z}^y.
\end{split}
\end{equation}
We will bound $T_{\gamma,Z}^x$ in Section \ref{gamma_x_deriv_section} and we will bound $T_{\gamma,Z}^y$ in Section \ref{gamma_y_deriv_section}.
\subsubsection{x-derivatives}\label{gamma_x_deriv_section}
By considering the Fourier side in $y$ and using boundedness of $N_k$, $M_k$, and $\mathfrak{J}_k$ as Fourier multipliers, we have by the triangle inequality,
\begin{equation}\label{split_gamma_z_x_terms}
    \begin{split}
        |T_{\gamma, Z}^x| &\lesssim \iiiint_{\R^4}  \dk^{2J} \langle \xi \rangle^{n} \langle k \rangle^{2m} \langle k \rangle^{1/2} |Z_\xi| \vkm^n \skm^{-1/2} |k-k'| p_{\xi-\xi'}^{-3/4} |Z_{\xi -\xi'}| \skp^{-1/2}\\
        & \quad \quad \quad \quad \quad \quad \langle \xi' \rangle^{n}|\eta' -k't| Z_{\xi'} d\eta' d\eta dk' dk\\
        &\quad + \iiiint_{\R^4} \dk^{2J}\langle \xi \rangle^{n}   \langle k \rangle^{2m} \sk^{1/2} p_\xi^{-1/4} |Z_\xi| \langle \xi - \xi'\rangle^n \skm^{-1/2} |k-k'|^{1/2} p_{\xi-\xi'}^{-1/2} |Z_{\xi -\xi'}| \\
        & \quad \quad \quad \quad \quad \quad \quad \vkp^n \skp^{-1/2} |\eta' -k't| Z_{\xi'} d\eta' d\eta dk' dk\\
        &\eqqcolon T_{\gamma, Z_1}^x + T_{\gamma, Z_2}^x.
    \end{split}
\end{equation}
We begin by addressing $T_{\gamma, Z_2}^x$, as it is simpler than $T_{\gamma, Z_1}^x$. We define the decomposition into low-high and high-low terms as
\begin{equation}
    \begin{split}
        T_{\gamma,Z_2}^x &= \iiiint_{\R^4} \left(1_{LH} + 1_{HL}\right) \dk^{2J} \langle \xi \rangle^{n} \langle k \rangle^{2m} \langle k \rangle^{1/2} p_\xi^{-1/4} |Z_\xi|\\
        &\quad \quad \quad \quad \quad \langle \xi - \xi' \rangle^n \langle k-k' \rangle^{-1/2} |k-k'| p_{\xi-\xi'}^{-1/2} |Z_{\xi-\xi'}|\vkp^n \langle k' \rangle^{-1/2}|(\eta' -k't) Z_{\xi'}| d\eta' d\eta dk' dk\\
        &\coloneqq T_{\gamma,Z_2, LH}^x + T_{\gamma,Z_2, HL}^x.
    \end{split}
\end{equation}
In general, we will use the $LH$ subscript to denote the low-high terms and $HL$ to denote the high-low terms throughout Section \ref{nonlinear_section}. For $T_{\gamma, Z_2, LH}^x$ we use the fact that we are in the $LH$ regime, and so $|k-k'|^{1/2} \lesssim |k'|^{1/2}$. Then by Young's convolution inequality, we find
\begin{equation}\label{gamma_z2_LH}
    \begin{split}
        T_{\gamma, Z_2, LH}^x &\lesssim \iiiint_{\R^4} 1_{LH}\dk^{J} \langle \xi\rangle^{n} \sk^m \min(|k|^{1/4}, 1) p_\xi^{-1/4} |Z_\xi| \vkm^{n} p_{\xi -\xi'}^{-1/2} |Z_{\xi-\xi'}|\\
        &\quad \quad \quad \quad (|k-k'|^{-1/4} 1_{|k| \leq 1} + 1_{|k| \geq 1})\dkp^J \vkp^n \skp^m |(\eta' - k't) Z_{\xi'}| d\eta' d\eta dk' dk\\
        & \lesssim || \dk^J \langle \xi \rangle^n \langle k \rangle^m \min(|k|^{1/4}, 1) p_\xi^{-1/4} Z_\xi ||_{L^2_k L^1_\eta}  \\
        & \quad \quad || \vk^n \sk^{-1/2} |k| \max(|k|^{-1/4}, 1) p_\xi^{-1/2} Z_\xi||_{L^1_k L^2_\eta}|| \dk^J \vk^n \sk^m (\eta-kt) Z_\xi ||_{L^2_\xi}.
    \end{split}
\end{equation}
    The first factor in \eqref{gamma_z2_LH} is controlled by interpolating in $\eta$ as done in Lemma \ref{new_energy_lemmas}. For the second factor, we interpolate in $k$ as
\begin{equation}\label{interp_in_k_example}
    \begin{split}
         || \vk^n \sk^{-1/2} |k| &\max(|k|^{-1/4}, 1) p_\xi^{-1/2} Z_\xi||_{L^1_k L^2_\eta}\\ &\lesssim \left(\int_\R \max(|k|^{-1/2},1) \langle k \rangle^{-1 - 2m} dk \right)^{1/2} || \vk^n \sk^{m} |k| p_\xi^{-1/2} Z_\xi||_{L^2_\xi}\lesssim \D_\tau^{1/2},
    \end{split}
\end{equation}
where we have crucially used $m > 0$. To handle the final factor, we simply recall the definition of $D_\gamma$ in \eqref{definition_of_diss}. Altogether, we find
\begin{equation}    
    \begin{split}
        T_{\gamma, Z_2, LH}^x \lesssim (\E^{1/2} + \mu^{-\delta_*} \D_\gamma^{\delta_*} \E^{1/2 - \delta_*}) \D_\tau^{1/2} \mu^{-1/2} \D_\gamma^{1/2},
    \end{split}
\end{equation}
which suffices for the purposes of Lemma \ref{key_lemma} since $\D_\tau \lesssim \E$. To bound $T_{\gamma, Z_2, HL}^x$, we begin by using that $|k| \lesssim |k-k|$ in the $HL$ regime. This yields
\begin{equation}\label{gamma_z_2_x_HL}
    \begin{split}
        T_{\gamma, Z_2, HL}^x &\lesssim \iiiint_{\R^4}  1_{HL}\dk^{J} \langle \xi\rangle^{n} \sk^m  p_\xi^{-1/4} |Z_\xi| \dkm^J \langle \xi -\xi'\rangle^n \skm^m \\
        &\quad\quad\quad\quad  |k-k'| p_{\xi-\xi'}^{-1/2} |Z_{\xi-\xi'}| \vk^n
        \langle k ' \rangle^{-1/2} |(\eta' - k't) Z_{\xi'}| d\eta' d\eta dk' dk.
    \end{split}
\end{equation}
Now we additionally split along the sets 
\begin{equation}\label{set_def_E}
    \mathcal{O} \coloneqq \{ (\xi, \xi') \in \R^4 : |k| \leq |k'|\}
\end{equation} and $\mathcal{O}^c$. Then by exploiting $|k| \leq |k'|$ on $\mathcal{O}$ and $|k'| < |k|$ on $\mathcal{O}^c$, we have by interpolation, $m > 0$, Young's inequality, and Lemma \ref{new_energy_lemmas},
\begin{equation}
    \begin{split}
        T_{\gamma, Z_2, HL}^x &\lesssim \iiiint_{\R^4}  1_{HL} \biggl(1_{\mathcal{O}} \dkp^J\skp^{m+1/2} \sk^{-1/2} + 1_{\mathcal{O}^c}\dk^J\sk^{m} (1_{|k| < 1} |k|^{1/4} |k'|^{-1/4} + 1_{|k| \geq 1}) \biggr) \\
        &\quad\quad\quad\quad  \langle \xi\rangle^{n}   p_\xi^{-1/4} |Z_\xi| \dkm^J \langle \xi -\xi'\rangle^n \skm^m |k-k'| p_{\xi-\xi'}^{-1/2} |Z_{\xi-\xi'}|\\
        &\quad\quad\quad\quad \vkp^n
        \langle k ' \rangle^{-1/2} |(\eta' - k't) Z_{\xi'}| d\eta' d\eta dk' dk\\
        &\lesssim || \langle \xi \rangle^n \sk^{-1/2} p_\xi^{-1/4} Z_\xi||_{L^1_\xi} || \dk^J \vk^n \sk^m |k| p_\xi^{-1/2} Z_\xi||_{L^2_\xi} || \dk^J \vk^n \sk^m  |\eta - kt|  Z_\xi||_{L^2_\xi}\\
        & \quad + ||\dk^J \langle \xi \rangle^n \sk^m \min(|k|^{1/4}, 1) p_\xi^{-1/4}||_{L^2_k L^1_\eta} || \dk^J \vk^n \sk^m |k| p_\xi^{-1/2} Z_\xi||_{L^2_\xi} \\
        &\quad \quad \quad || \vk^n \sk^{-1/2} \max(|k|^{-1/4}, 1) |\eta - kt|  Z_\xi||_{L^1_k L^2_\eta}\\
        &\lesssim (\E^{1/2} + \mu^{-\delta_*} \D_\gamma^{\delta_*} \E^{1/2 - \delta_*}) \D_\tau^{1/2} \mu^{-1/2} \D_\gamma^{1/2}.
    \end{split}
\end{equation}
We now turn our attention to  $T_{\gamma, Z_1}^x$ term, and split as
\begin{equation}\label{splitting_of_T_gamma_Z1_x}
    \begin{split}
        T_{\gamma, Z_1}^x &=  \iiiint_{\R^4}\left(1_{LH} + 1_{HL}\right) \dk^{2J} \langle \xi \rangle^{n} \langle k \rangle^{2m} \langle k \rangle^{1/2}  |Z_\xi| \vkm^n\\
        &\quad \quad \quad \quad \langle k-k' \rangle^{-1/2}|k-k'| p_{\xi-\xi'}^{-3/4} |Z_{\xi-\xi'}| \vkp^n \langle k'\rangle^{-1/2}|(\eta'-k't)Z_{\xi'}| d\eta' d\eta dk' dk\\
        &\eqqcolon T_{\gamma, Z_1, LH}^x + T_{\gamma, Z_1, HL}^x.
    \end{split}
\end{equation}
To estimate $T_{\gamma, Z_1, HL}^x$, we compute by interpolation, $m > 0$, Lemma \ref{main_D_tau_lemma}, and the fact that $|k|,|k'| \lesssim |k-k'|$  on the support of the integrand:
\begin{equation}
    \begin{split}
        T_{\gamma, Z_1, HL}^x &\lesssim \iiiint_{\R^4} 1_{HL} \dk^{J} \langle \xi \rangle^{n} \langle k \rangle^{m} |Z_\xi| \dkm^J \vkm^n \skm^m |k-k'| \min(|k-k'|^{1/4},1) p_{\xi-\xi'}^{-3/4}\\
        & \quad \quad \quad \quad \quad  \quad  |Z_{\xi-\xi'}| \vkp^n \langle k'\rangle^{-1/2} (1 + 1_{|k-k'| \leq 1}|k'|^{-1/4})|(\eta'-k't)Z_{\xi'}| d\eta' d\eta dk' dk\\
        &\lesssim || \dk^J \vk^n \sk^{m} Z_\xi||_{L_\xi^2}|| \dk^J \sk^m \min(|k|,|k|^{5/4}) p^{-3/4} Z_\xi ||_{L_k^2 L_\eta^1}\\& \quad \quad || \vk^n \sk^{-1/2}\max(|k|^{-1/4},1) |\eta - kt| Z_\xi ||_{L^1_k L^2_\eta}\\
        &\lesssim \E^{1/2} (\D_\tau^{1/2} + \mu^{-\delta^*} \D_\gamma^{\delta_*} \D_\tau^{1/2 - \delta_*}) \mu^{-1/2} \D_\gamma^{1/2}.
    \end{split}
\end{equation}
 We now address the LH term.  Using that $|k| \approx |k'|$ in the $LH$ case, together with Young's inequality and Lemma \ref{main_D_tau_lemma}, we have
\begin{equation}
    \begin{split}
        T_{\gamma, Z_1, LH}^x &\lesssim  \iiiint_{\R^4} 1_{LH} \dk^{J}\langle \xi \rangle^{n} \sk^m |Z_\xi|  \vkm^n \langle k-k'\rangle^{-1/2} |k-k'| p_{\xi-\xi'}^{-3/4} |Z_{\xi-\xi'}|\\
        &\quad \quad \quad \quad \dkp^J \vkp^n \skp^m |(\eta' -k't) Z_{\xi'}| d\eta' d\eta dk' dk\\
        &\lesssim ||\dk^{J}\langle \xi \rangle^{n} \sk^m  Z_\xi||_{L^2_\xi} || \vk^n \sk^{-1/2} |k| p_\xi^{-3/4} Z_{\xi}||_{L^1_\xi}\\
        &\quad \quad ||\dk^{J}\langle \xi \rangle^{n} \sk^m  (\eta -kt) Z_\xi||_{L^2_\xi}\\
        &\lesssim \E^{1/2}( \D_\tau^{1/2} + \mu^{-\delta_*} \D_\gamma^{ \delta_*} \D_\tau^{1/2 - \delta_*}) \mu^{-1/2} \D_\gamma^{1/2}.
    \end{split}
\end{equation}
This completes the $T_{\gamma, Z}^x$ terms.

\subsubsection{y-derivatives}\label{gamma_y_deriv_section}
We now turn our attention to $T_{\gamma,Z}^y$. By boundedness of $N_k$ and $\mathfrak{J}_k$, we have 
\begin{equation}
    \begin{split}
        |T_{\gamma, Z}^y| &\lesssim \iiiint_{\R^4} \dk^{2J} \langle \xi \rangle^{n} \sk^{2m} \langle k \rangle^{1/2} |Z_\xi| \vkm^n \langle k-k'\rangle^{-1/2} p_{\xi-\xi'}^{-3/4}\\
        &\quad \quad \quad |(\eta - \eta') + (k-k')t| |Z_{\xi-\xi'}| \vkp^n \langle k'\rangle^{-1/2} |k'| |Z_{\xi'}| d\eta' d\eta dk' dk\\
        &\quad + \iiiint_{\R^4} \dk^{2J} \langle \xi \rangle^{n} \sk^{2m} \langle k \rangle^{1/2} p_{\xi}^{-1/4}|Z_\xi| \vkm^n \langle k-k' \rangle^{-1/2} p_{\xi-\xi'}^{-1/2}\\
        &\quad \quad \quad \quad |(\eta - \eta') + (k-k')t| |Z_{\xi-\xi'}| \vkp^n \langle k' \rangle^{-1/2}|k'| |Z_{\xi'}| d\eta' d\eta dk' dk\\
        &\eqqcolon T_{\gamma, Z_1}^y + T_{\gamma, Z_2}^y,
    \end{split}
\end{equation}
where we have performed a splitting similar to that of $T_{\gamma, Z}^x$ in \eqref{split_gamma_z_x_terms}. Let us begin with $T_{\gamma, Z_1}^y$, which we divide in frequency space according to the decomposition
\begin{equation}\label{T_gamma_Z1_splitting}
    \begin{split}
       T_{\gamma, Z_1}^y &=  \iiiint_{\R^4} \biggl(1_{LH}(1_{|k| \geq \mu} + 1_{|k| < \mu}) + 1_{HL}(1_{|k'| \geq \mu} + 1_{|k'| < \mu}1_{|k| \geq \mu} + 1_{|k'| < \mu}1_{|k|< \mu})\biggr)\\
       &\quad \quad \quad \quad \dk^{2J}\langle \xi \rangle^{n} \sk^{2m} \langle k \rangle^{1/2} |Z_\xi| \vkm^n \langle k-k' \rangle^{-1/2} p_{\xi-\xi'}^{-3/4} |(\eta - \eta' + (k-k')t) Z_{\xi-\xi'}|\\
       &\quad \quad \quad \quad \vkp^n \langle k' \rangle^{-1/2} |k'| |Z_{\xi'}|  d\eta' d\eta dk' dk\\
       &\eqqcolon T_{\gamma, Z_1, LH, (H,\cdot)}^y + T_{\gamma, Z_1, LH, (L,\cdot)}^y + T_{\gamma, Z_1, HL, (\cdot,H)}^y + T_{\gamma, Z_1, HL, (H, L)}^y + T_{\gamma, Z_1, HL, (L, L)}^y
    \end{split}
\end{equation}
We note the large number of terms in \eqref{T_gamma_Z1_splitting} stands in contrast with the $T_{\gamma, Z}^x$ terms. In general the $y$-derivatives will require more subcases than the $x$-derivatives. Starting with $T_{\gamma, Z_1, LH, (H,\cdot)}^y$, we find by interpolation, Young's inequality, $|k| \approx |k'|$ on the support of the integrand, boundedness of the Riesz Transform, and a variation on Lemma \ref{new_energy_lemmas},
\begin{equation}
    \begin{split}
        T_{\gamma, Z_1, LH,(H,\cdot)}^y &\lesssim \int_{|k| \geq \mu}\iiiint_{\R^3} 1_{LH}\dk^{J} \langle \xi \rangle^{n} \sk^m |k|^{2/3 + 2\delta_*/3} |Z_\xi| \vkm^n \langle k-k' \rangle^{-1/2}  \\
        &\quad \quad \quad \quad p_{\xi-\xi'}^{-1/4} |Z_{\xi-\xi'}| 
        \dkp^{J} \langle \xi' \rangle^n \skp^m |k'|^{1/3 - 2\delta_*/3} |Z_{\xi'}|  d\eta' d\eta dk' dk\\
        &\lesssim \mu^{-1/12 + \delta_*/6} \D_\beta^{1/4 - \delta_*/2} \mu^{-1/4 - \delta_*/2} \D_\gamma^{1/4 + \delta_*/2} || \vk^n \sk^{-1/2} p_\xi^{-1/4} Z_\xi||_{L^1_\xi}\\
        &\quad \quad \mu^{-1/6 + \delta_*/3}\D_\beta^{1/2 - \delta_*} \E^{\delta_*}\\
        &\lesssim  \mu^{-1/2} \D_\beta^{1/4 - \delta_*/2} \D_\gamma^{1/4 + \delta_*/2} \mu^{-\delta_*} \D_\gamma^{\delta_*} \E^{1/2 - \delta_*} \D_\beta^{1/2 - \delta_*} \E^{\delta_*}\\
        &\lesssim \mu^{-1/2 - \delta_*} \D \E^{1/2}.
    \end{split}
\end{equation}
Now for $T_{\gamma, Z_1, LH, (L,\cdot)}^y$, we apply the triangle inequality to write $1 = p_\xi^{-1/4}p_{\xi}^{1/4} \lesssim p_\xi^{-1/4}(p_{\xi-\xi'}^{1/4} + p_{\xi'}^{1/4})$. This then implies
\begin{equation}\label{decomp_of_Z1_LH_at_low_freq}
\begin{split}
    T_{\gamma, Z_1, LH, (L,\cdot)}^y &\lesssim \int_{|k| \leq \mu }\iiiint_{\R^3} 1_{LH}\dk^{2J} \langle \xi \rangle^{n} \sk^m p_{\xi}^{-1/4} |Z_\xi| \vkm^n \langle k-k'\rangle^{-1/2}  p_{\xi-\xi'}^{-1/2}\\
    &\quad \quad \quad \quad \quad |(\eta - \eta' + (k-k')t) Z_{\xi-\xi'}| \vkp^n \skp^m  |k'| |Z_{\xi'}| d\eta' d\eta dk' dk\\
    &+ \int_{|k| \leq \mu }\iiiint_{\R^3} 1_{LH}\dk^{2J} \langle \xi \rangle^{n} \sk^m p_{\xi}^{-1/4} |Z_\xi| \vkm^n \langle k-k'\rangle^{-1/2}  p_{\xi-\xi'}^{-3/4}\\
    &\quad \quad \quad \quad \quad |(\eta - \eta' + (k-k')t) Z_{\xi-\xi'}| \vkp^n \skp^m p_{\xi'}^{1/4} |k'| |Z_{\xi'}| d\eta' d\eta dk' dk\\
    &\eqqcolon (I) + (I').
\end{split}
\end{equation}
For $(I)$, we simply have $(I) \lesssim T_{\gamma, Z_2, LH}^y$, which is controlled in \eqref{gamma_Z2_LH_y}. To address $(I')$, we note that in the $LH$ regime, $|k| \approx |k'|$, implying that $|k|, |k'|, |k-k'| \lesssim \mu$ and so we utilize Young's inequality, boundedness of the Riesz Transform, and Lemma \ref{new_energy_lemmas}. Our usage of Lemma \ref{new_energy_lemmas} is slightly modified to gain powers of $\mu$ via H\"older's inequality in $k$.
\begin{equation}
    \begin{split}
        (I') &\lesssim \int_{|k| \leq \mu }\iiiint_{\R^3} 1_{LH}\dk^{J} \langle \xi \rangle^{n} \sk^m |k| p_{\xi}^{-1/2} |Z_\xi| \vkm^n \langle k-k'\rangle^{-1/2}  p_{\xi-\xi'}^{-3/4}\\
    &\quad \quad \quad \quad \quad |(\eta - \eta' + (k-k')t) Z_{\xi-\xi'}| \dkp^{J}\vkp^n \skp^m p_{\xi'}^{1/2}|Z_{\xi'}| d\eta' d\eta dk' dk\\
    &\lesssim  || \dk^{J} \langle \xi \rangle^{n} \sk^m |k| p_{\xi}^{-1/2} Z_\xi ||_{L^2_\xi} || \vk^n\sk^{-1/2} p_\xi^{-1/4} Z_\xi||_{L^1_{|k| < \mu}L^1_\eta} || \dk^{J} \vk^n \sk^m  p_\xi^{1/2} Z_\xi||_{L^2_\xi}\\
    &\lesssim \D_\tau^{1/2} (\mu^{1/4}\E^{1/2} + \mu^{1/2-\delta_*} \D_\gamma^{\delta_*} \E^{1/2 - \delta_*}) \mu^{-1/2} \D_\gamma^{1/2}.
    \end{split}
\end{equation}
We turn our attention to the $HL$ terms for $Z_1$. In the $T_{\gamma, Z_1, HL, (\cdot, H)}^y$ case, we use that $|k'| \lesssim |k-k'|$ to move additional frequencies to the $k-k'$ factor. Then we use Young's inequality to place the $\xi'$ term in $L^1_{|k| \geq \mu} L^2_\eta$. Hence by interpolation, $m > 0$ and Lemma \ref{D_taualpha_and_alpha},
\begin{equation}\label{T_gamma_Z1_HL_y_high_prime}
    \begin{split}
        T_{\gamma, Z_1, HL, (\cdot, H)}^y & \lesssim \int_{|k'| \geq \mu} \iiint_{\R^3} 1_{HL }\dk^{J}\langle \xi \rangle^{n} \sk^m  |Z_\xi| \dkm^{J} \vkm^n \skm^m |k-k'|^{2/3}\\    
        &\quad \quad \quad \quad  \min(|k-k'|^{1/4}, 1) p_{\xi-\xi'}^{-3/4}|(\eta - \eta') - (k-k')t| | Z_{\xi-\xi'}| \\
       &\quad \quad \quad \quad  \vkp^n \langle k' \rangle^{-1/2} \max(|k'|^{-1/4},1) |k'|^{1/3} |Z_{\xi'}|  d\eta' d\eta dk dk'\\
       &\lesssim ||\dk^{J}\langle \xi \rangle^{n} \sk^m  Z_\xi||_{L^2_\xi}\\
       &\quad \quad ||\dk^{J}\vk^n \sk^m \min(|k|^{1/4}, 1) |k|^{2/3} |\eta - kt| p_\xi^{-3/4} Z_\xi||_{L^2_{|k| \geq \mu} L^1_\eta}\\
       &\quad \quad ||\langle k'\rangle^{-1/2} \max(|k|^{-1/4},1) |k|^{1/3} Z_\xi||_{L^1_{|k| \geq \mu}L^2_\eta}\\
       &\lesssim \E^{1/2}(\mu^{-1/3} \D_{\tau \alpha}^{1/2} + \mu^{-1/3 - \delta_*} \D_{\tau\alpha}^{1/2 - \delta_*} \D_\alpha^{\delta_*}) \mu^{-1/6} \D_\beta^{1/2}.
    \end{split}
\end{equation}
For $T_{\gamma, Z_1, HL, (H, L)}^y$, we use the fact that this implies $|k'| \leq |k|$ and $|k-k'| \gtrsim \mu$. This allows us to proceed in a similar fashion to \eqref{T_gamma_Z1_HL_y_high_prime}, applying Young's inequality, interpolation with $m > 0$, and Lemma \ref{D_taualpha_and_alpha},
\begin{equation}
    \begin{split}
        T_{\gamma, Z_1, HL, (H, L)}^y & \lesssim \int_{|k| \geq \mu} \int_{|k'| < \mu}\iint_{\R^2} 1_{HL } \dk^{J}\langle \xi \rangle^{n} \sk^m  |k|^{1/3} |Z_\xi| \dkm^{J} \vk^n\\
        &\quad \quad \quad \quad \skm^m |k-k'|^{2/3} \min(|k-k'|^{1/4}, 1) p_{\xi-\xi'}^{-3/4} |(\eta - \eta') - (k-k')t|  | Z_{\xi-\xi'}|\\
       &\quad \quad \quad \quad  \vkp^n \langle k' \rangle^{-1/2} \max(|k'|^{-1/4},1)  |Z_{\xi'}|  d\eta' d\eta dk' dk\\
       &\lesssim \mu^{-1/6} \D_\beta^{1/2}(\mu^{-1/3} \D_{\tau \alpha}^{1/2} + \mu^{-1/3 - \delta_*} \D_{\tau\alpha}^{1/2 - \delta_*} \D_\alpha^{\delta_*}) \E^{1/2}.
    \end{split}
\end{equation}
Lastly, when $|k| \lesssim \mu$ and $|k'| \lesssim \mu$, all $k$-based frequencies are $\lesssim \mu$. We then apply integration by parts to the $\partial_Y - t\partial_X$ derivative on the $k-k'$ factor. One could view this as an application of the triangle inequality. Then we have by Young's inequality, H\"older's inequality in $k$, and Lemma \ref{main_D_tau_lemma},
\begin{equation}
    \begin{split}
    T_{\gamma, Z_1, HL, (L, L)}^y & \lesssim \int_{|k| < \mu}\int_{|k'| < \mu} \iint_{\R^2} 1_{HL } \dk^{J}\langle \xi \rangle^{n} \sk^m |\eta -kt| |Z_\xi| \dkm^{J}  \vkm^n  \\
       &\quad \quad \quad \quad  \skm^m p_{\xi-\xi'}^{-3/4} |k-k'| | Z_{\xi-\xi'}|\vkp^n \langle k' \rangle^{-1/2} |Z_{\xi'}|  d\eta' d\eta dk' dk\\
       &+ \int_{|k| < \mu}\int_{|k'| < \mu} \iint_{\R^2} 1_{HL } \dk^{J}\langle \xi \rangle^{n} \sk^m |Z_\xi| \dkm^{J} \vkm^n\\
       &\quad \quad \quad \quad \skm^m  p_{\xi-\xi'}^{-3/4} |k-k'| | Z_{\xi-\xi'}|\vkp^n \langle k' \rangle^{-1/2}|\eta' -k't| |Z_{\xi'}|  d\eta' d\eta dk' dk\\
       &\lesssim \mu^{-1/2} \D_\gamma^{1/2} || \dk^J \vk^n \sk^m |k| p_\xi^{-3/4}||_{L^1_{|k| \lesssim \mu} L^1_\eta} \E^{1/2}\\
       & \quad + \E^{1/2} || \dk^J \vk^n \sk^m |k| p_\xi^{-3/4}||_{L^1_{|k| \lesssim \mu} L^1_\eta}\mu^{-1/2} \D_\gamma^{1/2}\\
       &\lesssim \mu^{-1/2}\D_{\gamma}^{1/2} (\D_\tau^{1/2} + \mu^{-\delta_*} \D_\gamma^{\delta_*} \D_\tau^{1/2 - \delta_*}) \E^{1/2}.
       \end{split}
\end{equation}
This completes the estimate for $T_{\gamma, Z_1}^y$ answering Lemma \ref{key_lemma}. For $T_{\gamma, Z_2}^y$, we use the frequency decomposition
\begin{equation}
    \begin{split}
       T_{\gamma, Z_2}^y &=  \iiiint_{\R^4} \biggl(1_{LH} + 1_{HL}\biggr)\dk^{2J}\langle \xi \rangle^{n} \sk^{2m} \langle k \rangle^{1/2} p_\xi^{-1/4} |Z_\xi|\\
       &\quad \quad \quad \quad \vkm^n \langle k-k' \rangle^{-1/2} p_{\xi-\xi'}^{-1/2} |(\eta - \eta' + (k-k')t) Z_{\xi-\xi'}| \langle k' \rangle^{-1/2} |k'| |Z_{\xi'}|  d\eta' d\eta dk' dk\\
       &\eqqcolon T_{\gamma, Z_2, LH}^y  + T_{\gamma, Z_2, HL}^y.
    \end{split}
\end{equation}
Dealing with $T_{\gamma, Z_2, LH}^y$ first, we apply $1 = p_\xi^{-1/4}p_\xi^{1/4} \lesssim p_\xi^{-1/4}(p_{\xi-\xi'}^{1/4} + p_{\xi'}^{1/4})$. We then use boundedness of the Riesz Transform, Young's inequality, interpolation, Lemma \ref{main_D_gamma_lemma}, $m > 0$, and $|k| \approx |k'|$:
\begin{equation}\label{gamma_Z2_LH_y}
    \begin{split}
        T_{\gamma, Z_2, LH}^y & \lesssim \iiint_{\R^4} 1_{LH} \dk^J \dkp^J \langle \xi \rangle^{n} \sk^{m} \skp^m |k| p_{\xi}^{-1/2}|Z_\xi| \langle k-k' \rangle^{-1/2} \vkm^n \vkp^n  \\
        &\quad\quad \quad \quad \biggl( p_{\xi-\xi'}^{1/4}|Z_{\xi-\xi'}||Z_{\xi'}| + (|k-k'|^{-1/4} |k'|^{1/4} 1_{|k'| \leq 1} + 1_{|k'| > 1})|Z_{\xi-\xi'}|p_{\xi'}^{1/4} |Z_{\xi'}|\biggr)d\eta' d\eta dk'\\
        &\lesssim \D_\tau^{1/2} \biggl( || \vk^n \langle k \rangle^{-1/2} \max(|k|^{-1/4},1) p_\xi^{1/4} Z_\xi||_{L^1_k L^2_\eta} || \dk^J \vk^n \sk^m \min(1, |k|^{1/4})  Z_\xi||_{L^2_k L^1_\eta}\\
        &\quad \quad \quad \quad \quad +
        || \vk^n \langle k \rangle^{-1/2} Z_\xi||_{L^1_k L^1_\eta} ||\dk^J \vk^n\sk^m p_\xi^{1/4} Z_\xi||_{L^2_k L^2_\eta} \biggr)\\
        &\lesssim \D_\tau^{1/2} \biggl(\mu^{-1/4}\D_\gamma^{1/4} \E^{1/4} (\mu^{-1/4 - \delta_*} \D_\gamma^{1/4 + \delta_*} \E^{1/4 - \delta_*} + \mu^{-1/4} \D_\gamma^{1/4} \E^{1/4})\\
        &\quad \quad \quad \quad \quad + (\mu^{-1/4 - \delta_*} \D_\gamma^{1/4 + \delta_*} \E^{1/4 - \delta_*} + \mu^{-1/4} \D_\gamma^{1/4} \E^{1/4})\mu^{-1/4}\D_\gamma^{1/4} \E^{1/4} \biggr)\\
&\lesssim \mu^{-1/2 - \delta_*} \D \E^{1/2}.
    \end{split}
\end{equation}
which suffices since $\D_\tau \lesssim \E$. Lastly, we consider $T_{\gamma, Z_2, HL}$. Here, we begin by noting that in the $HL$ regime
\begin{equation}
    1 \lesssim p_{\xi'}^{-1}\left(|k'| |k-k'| + |\eta' - k't|\left(|\eta-kt| + |\eta - \eta' + (k-k' )t| \right)\right).
\end{equation}
This gives us
\begin{equation}\label{decomposed_HL_gamma_y}
\begin{split}
         T_{\gamma, Z_2, HL}^y & \lesssim \iiiint_{\R^4 }1_{HL}\dk^{J}\langle \xi \rangle^{n} \sk^m  p_\xi^{-1/4} |Z_\xi| \dkm^{J} \vkm^n \skm^m\\
       &\quad \quad \quad \quad |k-k'| |Z_{\xi-\xi'}| \vkp^n \langle k' \rangle^{-1/2} |k'|^{2} p_{\xi'}^{-1} |Z_{\xi'}|  d\eta' d\eta dk' dk\\
       &\quad + \iiiint_{\R^4 }1_{HL}\dk^{J}\langle \xi \rangle^{n} \sk^m  p_\xi^{-1/4} |Z_\xi| \dkm^{J}\vkm^n \skm^m\\
       &\quad \quad \quad \quad  |\eta - \eta' + (k-k' )t| |Z_{\xi-\xi'}|  \vkp^n \langle k' \rangle^{-1/2}|\eta' - k't| |k'| p_{\xi'}^{-1} |Z_{\xi'}|  d\eta' d\eta dk' dk\\
       &\quad + \iiiint_{\R^4 }1_{HL}\dk^{J}\langle \xi \rangle^{n} \sk^m |\eta-kt|  p_\xi^{-1/4} |Z_\xi| \dkm^{J} \vkm^n \skm^m\\
       &\quad \quad \quad \quad |Z_{\xi-\xi'}| \vkp^n \langle k' \rangle^{-1/2} |k'|  |\eta' - k't|p_{\xi'}^{-1} |Z_{\xi'}|  d\eta' d\eta dk' dk\\
       &\eqqcolon T_{\gamma, Z_2, HL}^{y,a} + T_{\gamma, Z_2, HL}^{y,b} + T_{\gamma, Z_2, HL}^{y,c}.
\end{split}
\end{equation}
For $T_{\gamma, Z_2, HL}^{y,a}$  we split along $|k| \geq 1$ and $|k| < 1$, then apply Young's inequality, interpolation with $m > 0$ and Lemma \ref{new_energy_lemmas} to find
\begin{equation}\label{gamma_Z2_HL}
    \begin{split}
        T_{\gamma, Z_2, HL}^{y,a} & \lesssim ||\dk^{J} \sk^n \vk^m  p_\xi^{-1/4} Z_\xi ||_{L^2_{|k| \geq 1}L^1_\eta} || \dk^J \vk^n \sk^m |k| Z_\xi||_{L^2_\xi} ||\vk^n \sk^{-1/2} |k| p_\xi^{-1/2} Z_\xi||_{L^1_k L^2_\eta}\\
        &\quad +
         ||\dk^{J} \sk^n \vk^m p_\xi^{-1/4} Z_\xi ||_{L^1_{|k| < 1}L^1_\eta} || \dk^J \vk^n \sk^m |k| Z_\xi||_{L^2_\xi} ||\vk^n \sk^{-1/2} |k| p_\xi^{-1/2} Z_\xi||_{L^2_\xi}\\
        &\lesssim (\E^{1/2} + \mu^{-\delta_*} \D_\gamma^{\delta_*} \E^{1/2 - \delta_*}) \mu^{-1/2} \D_\gamma^{1/2} \D_\tau^{1/2}.
    \end{split}
\end{equation}
For the second term of \eqref{decomposed_HL_gamma_y}, we again split along $|k| \geq 1$ and $|k| < 1$, together with Young's inequality, interpolation, and Lemma \ref{new_energy_lemmas}:
\begin{equation}
\begin{split}
    T_{\gamma, Z_2, HL}^{y,b} & \lesssim ||\dk^J \langle \xi\rangle^n \sk^m   p_\xi^{-1/4} Z_{\xi}  ||_{L^2_{|k| \geq 1}L^1_\eta}  ||\dk^J \langle \xi\rangle^n \sk^m  |\eta - kt| Z_\xi||_{L^2_\xi}\\
    &\quad \quad \quad || \vk^n \sk^{-1/2} |k| p_\xi^{-1/2} Z_{\xi}||_{L^1_k L^2_{\eta}}\\
    &\quad + ||\dk^J \langle \xi\rangle^n \sk^m   p_\xi^{-1/4} Z_{\xi}  ||_{L^1_{|k| < 1}L^1_\eta}  ||\dk^J \langle \xi\rangle^n \sk^m  |\eta - kt| Z_\xi||_{L^2_\xi}\\
    &\quad \quad \quad \quad  || \vk^n \sk^{-1/2} |k| p_\xi^{-1/2} Z_{\xi}||_{L^2_\xi}\\
    &\lesssim (\E^{1/2} + \mu^{-\delta_*} \D_\gamma^{\delta_*} \E^{1/2 - \delta_*}) \mu^{-1/2} \D_\gamma^{1/2} \D_\tau^{1/2}.
\end{split}
\end{equation}
Lastly, we consider $T_{\gamma, Z_2, HL}^{y,c}$. Here, we split along $|k-k'| \geq 1$ and $|k-k'|<1$. We then use boundedness of the Riesz Transform, Young's inequality, interpolation, and Lemma \ref{main_D_gamma_lemma}:
\begin{equation}
    \begin{split}
        T_{\gamma, Z_2, HL}^{y,c} & \lesssim ||\dk^J \langle \xi\rangle^n \sk^m  |\eta - kt|^{1/2} Z_{\xi}  ||_{L^2_\xi}  || \dk^J \vk^n \sk^m  Z_\xi||_{L^2_{|k| \geq 1}L^1_\eta}\\
        &\quad ||\vk^n \sk^{-1/2} |k| p_\xi^{-1/2} Z_\xi||_{L^1_k L^2_\eta}\\
        &\quad +||\dk^J \langle \xi\rangle^n \sk^m  |\eta - kt|^{1/2} Z_{\xi}||_{L^2_\xi}|| \dk^J \vk^n \sk^m  Z_\xi||_{L^1_{|k| < 1}L^1_\eta}\\
        & \quad \quad ||\vk^n\sk^{-1/2} |k| p_\xi^{-1/2} Z_\xi||_{L^2_\xi}\\
        &\lesssim \mu^{-1/4} \D_\gamma^{1/4} \E^{1/4} \mu^{-1/4}\D_\gamma^{1/4}(\E^{1/4} + \mu^{-\delta_*} \D_\gamma^{\delta_*} \E^{1/4 - \delta_*}) \D_\tau^{1/2}.
    \end{split}
\end{equation}
which concludes the necessary bounds on $T_{\gamma, Z}$.
\subsection{Gamma Terms for Q}
Using the decomposition \eqref{conv_NL_Q}, we have explicitly
\begin{equation}
    \begin{split}
        T_{\gamma,Q} &=  -\iint_{\R^2} \frac{\dk^{2J}}{M_k(t)} \langle \xi \rangle^{2n}\langle k \rangle^{2m}  (N_k + c_\tau\mathfrak{J}_k)\mathrm{Re}\biggl(\iint_{\R^2} \sgn(k) \langle k \rangle^{1/2} p_\xi^{1/4} \bar{Q}_\xi\\
    &\quad \quad \quad \quad \quad \quad \langle k-k' \rangle^{-1/2} p_{\xi -\xi'}^{-3/4} i (k-k') Z_{\xi-\xi'} \sgn(k') \langle k' \rangle^{-1/2} p_{\xi'}^{-1/4}  i(\eta'-k't) Q_{\xi'}d\eta' dk'\biggr) d\eta dk\\
    &\quad +\iint_{\R^2} \frac{\dk^{2J}}{M_k(t)}  \langle \xi \rangle^{2n}\langle k \rangle^{2m}  (N_k + c_\tau\mathfrak{J}_k)\mathrm{Re}\biggl(\iint_{\R^2} \sgn(k) \langle k \rangle^{1/2} p_\xi^{1/4} \bar{Q}_\xi\\
    &\quad \quad \quad \quad \quad \quad \langle k-k' \rangle^{-1/2} p_{\xi -\xi'}^{-3/4} i ((\eta-\eta') - (k-k')t) Z_{\xi-\xi'} \sgn(k') \langle k' \rangle^{-1/2} p_{\xi'}^{-1/4} i k' Q_{\xi'}d\eta' dk'\biggr) d\eta dk\\
    &\eqqcolon T_{\gamma, Q}^x + T_{\gamma, Q}^y.
    \end{split}
\end{equation}
We will treat $T_{\gamma, Q}^x$ in Section \ref{x_derivs_Q} and $T_{\gamma, Q}^y$ in Section \ref{y_dervis_Q}.

\subsubsection{x-derivatives}\label{x_derivs_Q}
Employing a similar decomposition as in \eqref{split_gamma_z_x_terms}, we have by the triangle inequality:
\begin{equation}
    \begin{split}
        |T_{\gamma, Q}^x| & \lesssim \iiiint_{\R^4} \langle \xi \rangle^{n} \sk^{2m}  \langle k \rangle^{1/2} |Q_\xi| \vkm^n \langle k-k' \rangle^{-1/2} |k-k'| p_{\xi-\xi'}^{-3/4}\\
        &\quad \quad \quad \quad |Z_{\xi-\xi'}|  \vkp^n \skp^{-1/2} | \eta' -k't| |Q_{\xi'}| d\eta' d\eta dk' dk\\
        &\quad + \iiiint_{\R^4} \dk^{2J} \langle \xi \rangle^{n} \sk^{2m} \sk^{1/2} |Q_\xi| \vkm^n \skm^{-1/2}\\
        &\quad \quad \quad \quad \quad |k-k'| p_{\xi-\xi'}^{-1/2} |Z_{\xi-\xi'}| \vkp^n p_{\xi'}^{-1/4} \skp^{-1/2} |\eta'-k't||Q_{\xi'}| d\eta' d\eta dk' dk\\
        &\eqqcolon T_{\gamma, Q_1}^x + T_{\gamma, Q_2}^x.
    \end{split}
\end{equation}
We notice that $T_{\gamma, Q_1}^x$ is schematically identical to $T_{\gamma, Z_1}^x$, and hence the exact same arguments yield
\begin{equation}
    T_{\gamma, Q_1}^x \lesssim \mu^{-1/2 -\delta_*} \D \E^{1/2},
\end{equation}
as desired. We will focus our attention on $T_{\gamma, Q_2}^x$. We begin with introducing the decomposition
\begin{equation}
    \begin{split}
         T_{\gamma, Q_2}^x& \lesssim \iiiint_{\R^4} \left(1_{LH}+ 1_{HL}\right) \dk^{2J} \langle \xi \rangle^{n} \sk^{2m} \sk^{1/2} |Q_\xi| \vkm^n \skm^{-1/2}\\
         & \quad \quad \quad \quad |k-k'| p_{\xi-\xi'}^{-1/2} |Z_{\xi-\xi'}| \vkp^n p_{\xi'}^{-1/4} \skp^{-1/2} |\eta'-k't||Q_{\xi'}| d\eta' d\eta dk' dk\\
         &\eqqcolon T_{\gamma, Q_2, LH}^x  +  T_{\gamma, Q_2, HL}^x.
    \end{split}
\end{equation}
We start with the $LH$ term. Since $|k-k'| \lesssim |k|$ and $|k'| \approx |k|$ in the LH regime:
\begin{equation}\label{gamma_Q2_LH_1}
    \begin{split}
        T_{\gamma, Q_2, LH}^x &\lesssim \iiiint_{\R^4} 1_{LH}\dk^{J} \langle \xi \rangle^{n} \sk^m \max(|k|^{1/4}, 1)  |Q_\xi| \langle \xi - \xi'\rangle^n \skm^{-1/2}\\
        &\quad \quad \quad \quad \max(|k-k'|^{-1/4},1) |k-k'| p_{\xi-\xi'}^{-1/2} |Z_{\xi-\xi'}|  \dkp^{J}\\
         &\quad \quad \quad \quad \vkp^n\skp^m p_{\xi'}^{-1/4}|\eta'-k't||Q_{\xi'}| d\eta' d\eta dk' dk.
    \end{split}
\end{equation}
We now utilize Young's inequality twice, boundedness of the Riesz Transform (which implies $\D_\tau \lesssim \E^{1/2})$, and Lemma \ref{main_D_gamma_lemma} to obtain 
\begin{equation}\label{gamma_Q2_LH_2}
    \begin{split}
        T_{\gamma, Q_2, LH}^x &\lesssim ||\dk^{J} \langle \xi \rangle^{n} \sk^m \max(|k|^{1/4}, 1)  Q_\xi ||_{L^2_k L^1_\eta}|| \vk^n \sk^{-1/2} \max(|k|^{-1/4},1) |k| p_\xi^{-1/2} Z_\xi||_{L^1_k L^2_\eta}\\
        &\quad\quad \quad ||\dk^{J} \vk^n \sk^m |\eta- kt|^{1/2} Q_\xi ||_{L^2_\xi}\\
        &\lesssim \mu^{-1/4}\left(\D_\gamma^{1/4} \E^{1/4} + \mu^{-\delta_*} D_{\gamma}^{1/4+\delta_*}\E^{1/2-\delta_*}\right) \D_\tau^{1/2} \E^{1/4} \mu^{-1/4} \D_\gamma^{1/4} \E^{1/4}.
    \end{split}
\end{equation}
The $HL$ term follows similar principles. We begin by utilizing that $|k|,|k'| \lesssim |k-k'|$ to write
\begin{equation}
    \begin{split}
        T_{\gamma, Q_2, HL}^x &\lesssim \iiiint_{\R^4} 1_{HL}\dk^{J} \langle \xi \rangle^{n}  |Q_\xi| \dkm^{J} \langle \xi - \xi'\rangle^n \skm^m |k-k'| \\
         &\quad \quad \quad \quad   p_{\xi-\xi'}^{-1/2} |Z_{\xi-\xi'}| \vk^n|\left(1_{|k'| \geq 1} + 1_{|k'| < 1}|k-k'|^{-1/4}|k'|^{1/4}\right)\\
         &\quad \quad \quad \quad \skp^{-1/2} p_{\xi'}^{-1/4} |\eta'-k't||Q_{\xi'}| d\eta' d\eta dk' dk.
    \end{split}
\end{equation}
We are now in a form similar to \eqref{gamma_Q2_LH_1}, and so we use Young's inequality, interpolation and boundedness of the Riesz Transform. However, we use  Lemma \ref{main_D_gamma_lemma} to find
\begin{equation}
    \begin{split}
        T_{\gamma, Q_2, HL}^x &\lesssim ||\dk^{J} \langle \xi \rangle^{n} \sk^m  Q_\xi ||_{L^2_k L^1_\eta}||\dk^{J} \langle \xi \rangle^{n} \sk^m |k| p_\xi^{-1/2} Z_\xi||_{L^2_k L^2_\eta}\\
        &\quad \quad \quad \quad || \vk^n \sk^{-1/2} \max(|k|^{-1/4},1)  |\eta- kt|^{1/2} Q_\xi ||_{L^1_k L^2_\eta}\\
        &\lesssim (\mu^{-1/4 - \delta_*} \D_\gamma^{1/4 + \delta_*} \E^{1/4 - \delta_*} + \mu^{-1/8} \D_\gamma^{1/8} \E^{3/8}) \D_\tau^{1/2} || \vk^n\sk^m  \min(|k|^{1/4},1) |\eta- kt|^{1/2} Q_\xi ||_{L^2_\xi}.
    \end{split}
\end{equation}
We then use that
\begin{equation}
    || \vk^n\sk^m  \min(|k|^{1/4},1) |\eta- kt|^{1/2} Q_\xi ||_{L^2_\xi} \lesssim \mu^{-3/8} \D_\gamma^{3/8} \E^{1/8}
\end{equation}
and
\begin{equation}
    || \vk^n\sk^m  \min(|k|^{1/4},1) |\eta- kt|^{1/2} Q_\xi ||_{L^2_\xi} \lesssim \mu^{-1/4} \D_\gamma^{1/4} \E^{1/4}
\end{equation}
alongside $\D_\tau \lesssim \E$ to find
\begin{equation}
    T_{\gamma, Q_2, HL}^x \lesssim \mu^{-1/2 - \delta_*} \D \E^{1/2}.
\end{equation}
This completes the estimates on $T_{\gamma, Q}^x$.

\subsubsection{y-derivatives}\label{y_dervis_Q}
For the $y$ derivatives, we decompose on the Fourier side as
\begin{equation}
    \begin{split}
        |T_{\gamma, Q}^y|
        &\lesssim \iiiint_{\R^4} \dk^{2J}\langle \xi \rangle^{n} \sk^{2m} \langle k \rangle^{1/2} |Q_\xi| \vkm^n \langle k-k' \rangle^{-1/2} p_{\xi -\xi'}^{-3/4}\\
        &\quad \quad \quad \quad |\eta - \eta' - (k-k')t| |Z_{\xi -\xi'}| \vkp^n \langle k' \rangle^{-1/2} |k'| |Q_{\xi'}| d\eta' d\eta dk' dk\\
        &\quad + \iiiint_{\R^4} \dk^{2J}\langle \xi \rangle^{n} \sk^{2m} \langle k \rangle^{1/2} |Q_\xi|  \vkm^n \langle k-k' \rangle^{-1/2} p_{\xi -\xi'}^{-1/2}\\
        &\quad \quad \quad \quad \quad |\eta - \eta' - (k-k')t| |Z_{\xi -\xi'}| \vkp^n p_{\xi'}^{-1/4}\langle k' \rangle^{-1/2} |k'||Q_{\xi'}| d\eta' d\eta dk' dk.\\
        &\eqqcolon T_{\gamma, Q_1}^y + T_{\gamma, Q_2}^y.
    \end{split}
\end{equation}
Similar to the $T_{\gamma, Q}^x$ terms, the term $T_{\gamma, Q_1}^y$ can be handled by the same techniques used in estimating $T_{\gamma, Z_1}^y$. For $T_{\gamma, Q_2}^y$, we note that if we split into $LH$ and $HL$ cases in the usual manner, we find
\begin{equation}
    \begin{split}
         T_{\gamma, Q_2, LH}^y &\lesssim \iiiint_{\R^4} 1_{LH}\dk^{J}\langle \xi \rangle^{n} \sk^m \langle k \rangle^{1/2} |Q_\xi| \vkm^n \langle k-k'\rangle^{-1/2}p_{\xi -\xi'}^{-1/2} |\eta - \eta' - (k-k')t|\\
         &\quad \quad \quad \quad |Z_{\xi -\xi'}| \dkp^{J} \vkp^n \skp^m p_{\xi'}^{-1/4}\langle k' \rangle^{-1/2} |k'| |Q_{\xi'}| d\eta' d\eta dk' dk,
    \end{split}
\end{equation}
which can be controlled in the same way as $T_{\gamma, Z_2, LH}^y$ in \eqref{gamma_Z2_LH_y}, only with the role of the $\xi$ and $\xi'$ factors reversed. For $T_{\gamma, Q_2, HL}^y$, we note that
\begin{equation}\label{gamma_Q2_HL}
    \begin{split}
        T_{\gamma, Q_2, HL}^y & \lesssim \iiiint_{\R^4} 1_{HL}\dk^{J}\langle \xi \rangle^{n} \sk^m |Q_\xi| \dkm^{J}\langle \xi - \xi'\rangle^n \sk^m p_{\xi -\xi'}^{-1/2}\\ &\quad \quad \quad \quad|\eta - \eta' - (k-k')t| |Z_{\xi -\xi'}| \vkp^n p_{\xi'}^{-1/4} \langle k' \rangle^{-1/2} |k'| |Q_{\xi'}| d\eta' d\eta dk' dk.
    \end{split}
\end{equation}
Then by using the decomposition
\begin{equation}
\begin{split}
    1 \lesssim p_{\xi - \xi'}^{-1}p_{\xi - \xi'}\lesssim p_{\xi-\xi'}^{-1}\biggl( |k-k'|^2 + |(\eta- \eta') - (k-k')t|^{3/2}(|\eta' - k't|^{1/2} + |\eta - kt|^{1/2})\biggr),
\end{split}
\end{equation}
we see that through Young;s inequality, we may bound \eqref{gamma_Q2_HL} by
\begin{equation}
    \begin{split}
        T_{\gamma, Q_2, HL}^y & \lesssim || \dk^J \vk^n \sk^m Q_\xi||_{L^2_\xi} || \dk^J \vk^n \sk^m |k| p_\xi^{-1/2} Z_\xi||_{L^2_k L^1_\eta}|| \vk^n\sk^{-1/2}|k| p_\xi^{-1/4} Q_\xi||_{L^1_k L^2_\eta}\\
        &\quad + || \dk^J \vk^n \sk^m Q_\xi||_{L^2_\xi} || \dk^J \vk^n \sk^m |k|^{1/2} \min(|k|^{1/4}, 1) p_\xi^{-1/4} Z_\xi||_{L^2_k L^1_\eta}\\
        &\quad \quad || \vk^n\sk^{-1/2} \max(|k|^{-1/4},1)|k|^{1/2} |\eta-kt|^{1/2} p_\xi^{-1/4} Q_\xi||_{L^1_k L^2_\eta}\\
        & \quad + || \dk^J \vk^n \sk^m |\eta - kt|^{1/2} Q_\xi||_{L^2_\xi} || \dk^J \vk^n \sk^m |k|^{1/2} \min(|k|^{1/4}, 1) p_\xi^{-1/4} Z_\xi||_{L^2_k L^1_\eta}\\
        &\quad \quad || \vk^n\sk^{-1/2} \max(|k|^{-1/4},1)|k|^{1/2} p_\xi^{-1/4} Q_\xi||_{L^1_k L^2_\eta}.
    \end{split}
\end{equation}
Then we employ Lemma \ref{main_D_tau_lemma}, alongside interpolation in $k$ with $m > 0$ to find 
\begin{equation}
    \begin{split}
        T_{\gamma, Q_2, HL}^y & \lesssim \E^{1/2}(\D_\tau^{1/2} + \mu^{-1/4 - \delta_*}\D_\gamma^{1/4+\delta_*} \D_\tau^{1/4})\D_\tau^{1/4} \mu^{-1/4} \D_\gamma^{1/4}\\
        & \quad +  \E^{1/2}(\D_\tau^{1/4} \mu^{-1/4} \D_\gamma^{1/4} + \mu^{-1/4 - \delta_*}\D_\gamma^{1/4+\delta_*} \D_\tau^{1/4})\D_\tau^{1/4} \mu^{-1/4} \D_\gamma^{1/4}\\
        &\quad + \E^{1/4} \mu^{-1/4} \D_\gamma^{1/4} (\D_\tau^{1/4} \mu^{-1/4} \D_\gamma^{1/4} + \mu^{-1/4 - \delta_*}\D_\gamma^{1/4+\delta_*} \D_\tau^{1/4}) \D_\tau^{1/4} \E^{1/4}\\
        &\lesssim \mu^{-1/2 - \delta_*} \D \E^{1/2}.
    \end{split}
\end{equation}
This concludes the discussion of the $T_{\gamma, Q}$ terms.

\subsection{Mixed Gamma Terms}\label{mixed_gamma_terms}
We now turn to the terms of mixed type
\begin{equation}
    \begin{split}
        T_{\gamma, m_1} + T_{\gamma, m_2} &= \iint_{\R^2} \frac{\dk^{2J}}{M_k(t)} \langle k, \eta \rangle^{2n} \sk^{2m} \frac{1}{2\sqrt{R}}\frac{\partial_t p}{|k|p^{1/2}}\biggl( \mathrm{Re}(\mathbb{NL}_k^{(Z)} \bar{Q}_k) + \mathrm{Re}(\bar{Z}_\xi \mathbb{NL}_k^{(Q)})\biggr) \biggr) dk d \eta.
    \end{split}
\end{equation}
Importantly, we will not explicitly prove the bound \eqref{bootstrap_bound} for $T_{\gamma, m_1}$ or for $T_{\gamma, m_2}$. Indeed, after expanding out the nonlinear terms, we see by the triangle inequality and boundedness of $\frac{\partial_t p}{|k|p^{1/2}}$, that $T_{\gamma, m_1}$ can be bounded using schematically identical arguments to those used to bound $T_{\gamma, Z}$, and the same is true for $T_{\gamma, m_2}$ and $T_{\gamma, Q}$. The only true differences are hidden in the implicit dependence of all constants on $R$.

\subsection{Alpha Terms for Z}

We split into $x$-derivative terms and $y$-derivative terms as follows:
\begin{equation}
\begin{split}
T_{\alpha, Z} &=  - c_\alpha\iint_{\R^2} \frac{\dk^{2J}}{M_k(t)} \langle \xi \rangle^{2n} \sk^{2m} A(k)^2 (N_k + c_\tau\mathfrak{J}_k)(\eta -kt)^2\mathrm{Re}\biggl(\iint_{\R^2} \langle k \rangle^{1/2} p_\xi^{-1/4} \bar{Z}_\xi\\
    &\quad \quad \quad \quad \quad \quad \langle k-k' \rangle^{-1/2} p_{\xi -\xi'}^{-3/4} i (k-k') Z_{\xi-\xi'} \langle k' \rangle^{-1/2} p_{\xi'}^{1/4} i(\eta'-k't) Z_{\xi'}d\eta' dk'\biggr) d\eta dk\\
    &\quad + c_\alpha\iint_{\R^2} \frac{\dk^{2J}}{M_k(t)} \langle \xi \rangle^{2n} \sk^{2m}   A(k)^2(N_k + c_\tau\mathfrak{J}_k)(\eta -kt)^2\mathrm{Re}\biggl(\iint_{\R^2} \langle k \rangle^{1/2} p_\xi^{1/4} \bar{Z}_\xi\\
    &\quad \quad \quad \quad \quad \quad \langle k-k' \rangle^{-1/2} p_{\xi -\xi'}^{-3/4} i ((\eta-\eta') - (k-k')t) Z_{\xi-\xi'} \langle k' \rangle^{-1/2} p_{\xi'}^{1/4} i k' Z_{\xi'}d\eta' dk'\biggr) d\eta dk\\
    &\eqqcolon T_{\alpha, Z}^x + T_{\alpha, Z}^y.
\end{split}
\end{equation}

\subsubsection{x-derivatives}

By boundedness of $N_k$, $\mathfrak{J}_k$, and $M_k$, as well as the triangle inequality, we find
\begin{equation}
    \begin{split}
        |T_{\alpha, Z}^x| &\lesssim \iiiint_{\R^4} \dk^{2J} \langle \xi \rangle^{n} \sk^{2m} A(k)^2  \langle k \rangle^{1/2} |\eta - kt|^2 |Z_\xi|\\
        &\quad \quad \quad \quad \vkm^n \skm^{-1/2} |k-k'| p_{\xi-\xi'}^{-3/4} |Z_{\xi - \xi'}| \vkp^n \langle k' \rangle^{-1/2}|\eta' -k't| |Z_{\xi'}| d\eta' d\eta dk' dk\\
        &\quad +\iiiint_{\R^4} \dk^{2J} \langle \xi \rangle^{n} \sk^{2m} A(k)^2  \langle k \rangle^{1/2} |\eta - kt|^2 p_\xi^{-1/4} |Z_\xi| \vkm^n\\
        &\quad \quad \quad \quad \quad \skm^{-1/2} |k-k'| p_{\xi-\xi'}^{-1/2} |Z_{\xi - \xi'}| \vkp^n \langle k' \rangle^{-1/2}|\eta' -k't| |Z_{\xi'}| d\eta' d\eta dk' dk\\
        &\eqqcolon T_{\alpha, Z_1}^x + T_{\alpha, Z_2}^x.
    \end{split}
\end{equation}
We start by treating $T_{\alpha, Z_1}^x$, decomposing it as
\begin{equation}
    \begin{split}
        T_{\alpha, Z_1}^x &= \iiiint_{\R^4}\left(1_{LH} + 1_{HL}(1_{|k| \geq \mu} + 1_{|k| < \mu}1_{|k'| \geq \mu} + 1_{|k|< \mu}1_{|k'|< \mu})\right) \dk^{2J} \langle \xi \rangle^{n} \sk^{2m} A(k)^2 \langle k \rangle^{1/2} \\
        &\quad \quad \quad \quad  |\eta - kt|^2 |Z_\xi| \vkm^n\skm^{-1/2} |k-k'| p_{\xi-\xi'}^{-3/4} |Z_{\xi - \xi'}| \vkp^n \skp^{-1/2} |\eta' -k't| |Z_{\xi'}| d\eta' d\eta dk' dk\\
        &\eqqcolon T_{\alpha,Z_1, LH}^x + T_{\alpha,Z_1, HL, (H, \cdot)}^x + T_{\alpha,Z_1, HL, (L, H)}^x + T_{\alpha,Z_1, HL, (L, L)}^x.
    \end{split}
\end{equation}
To treat the $LH$ case, we begin by noting that since $|k| \approx |k'|$, we have $A(k) \approx A(k')$. Then by Young's inequality and Lemma \ref{main_D_tau_lemma}, we have
\begin{equation}
    \begin{split}
        T_{\alpha, Z_1, LH}^x &\lesssim \iiiint_{\R^4}1_{LH}\dk^{J} \langle \xi \rangle^{n} \sk^m A(k)  |\eta - kt|^2 |Z_\xi|  \vkm^n \skm^{-1/2} |k-k'| \\
        &\quad \quad \quad \quad p_{\xi-\xi'}^{-3/4} |Z_{\xi - \xi'}| \dkp^J \vkp^n \sk^m A(k') |\eta' -k't| |Z_{\xi'}| d\eta' d\eta dk dk'\\
        &\lesssim \mu^{-1/2} \D_{\alpha}^{1/2} ||\vk^n \sk^{-1/2} |k| p_{\xi}^{-3/4} Z_\xi||_{L^1_\xi} \E^{1/2}\\
        &\lesssim \mu^{-1/2} \D_{\alpha}^{1/2} \left(\D_\tau^{1/2} + \mu^{-\delta_*} \D_\gamma^{\delta_*} \D_\tau^{1/2 - \delta_*}\right) \E^{1/2}.
    \end{split}
\end{equation}
Examining the $HL$ cases, we begin with the $|k| \geq \mu$ case. A simple calculation shows
\begin{equation}\label{HL_alpha_key_equation}
    \langle k \rangle^{m+1/2} A(k)^2 \lesssim (1_{1 \geq |k| \geq \mu}|k|^{-1/2} + 1_{|k| > 1}) A(k) \langle k-k' \rangle^{m} |k-k'|^{1/2} A(k').
\end{equation}
Treating $T_{\alpha, Z_1, HL, (H, \cdot)}^x$ first, we use \eqref{HL_alpha_key_equation}.  We also distinguish between $|k| \geq 1$ and $|k| < 1$, applying Young's inequality to each case. We then use interpolation and Lemma \ref{main_D_tau_lemma}:
\begin{equation}\label{alpha_z1_x_HL}
    \begin{split}
        T_{\alpha, Z_1, HL, (H, \cdot)}^x &\lesssim \int_{|k| \geq \mu}\iiint_{\R^3}1_{HL}\dk^{J} \langle \xi \rangle^{n} \sk^m A(k) \max(|k|^{-1/2},1) |\eta - kt|^2 |Z_\xi| \dkm^J  \vkm^n  \\
        &\quad \quad \quad \quad \skm^m |k-k'|  p_{\xi-\xi'}^{-3/4} |Z_{\xi - \xi'}| \vkp^n A(k') \langle k' \rangle^{-1/2} |\eta' -k't| |Z_{\xi'}| d\eta' d\eta dk dk'\\
        &\lesssim ||\langle c \lambda_k t \rangle^J \langle \xi \rangle^n \sk^m A(k) |\eta-kt|^2 Z_\xi||_{L^2_{|k| \geq 1}L^2_\eta} || \langle c \lambda_k t \rangle^J \vk^n \sk^m |k| p_\xi^{-3/4} Z_\xi||_{L^2_{|k| \geq 1} L^1_\eta}\\
        &\quad  || \vk^n \sk^{-1/2} A(k) |\eta -kt| Z_\xi||_{L^1_k L^2_\eta}\\
        &\quad + ||\langle c \lambda_k t \rangle^J \langle \xi \rangle^n \sk^m |k|^{-1/2} A(k) |\eta-kt|^2 Z_\xi||_{L^1_{\mu < |k| < 1}L^2_\eta} \\
        &\quad \quad  || \langle c \lambda_k t \rangle^J \vk^n \sk^{m} |k| \min(|k|^{1/2},1) p_\xi^{-3/4} Z_\xi||_{L^2_{|k| \geq \mu} L^1_\eta} || \vk^n A(k) |\eta -kt| Z_\xi||_{L^2_\xi}\\
        &\lesssim \mu^{-1/2}\D_\alpha^{1/2}\left(\D_\tau^{1/2} + \mu^{-\delta_*} \D_\gamma^{\delta_*} \D_\tau^{1/2-\delta_*}\right) \E^{1/2}.\\
        &\quad + \mu^{-1/2} \ln(1/\mu)^{1/2}\D_\alpha^{1/2}\left(\D_\tau^{1/2} + \mu^{-\delta_*/2} \D_\gamma^{\delta_*/2} \D_\tau^{1/2-\delta_*/2}\right) \E^{1/2}\\
        &\lesssim \mu^{-1/2 - \delta_*} \D \E^{1/2}.
    \end{split}
\end{equation}
For $|k| < \mu$, $|k'| \geq \mu$ in $T_{\alpha, Z_1, HL, (L, H)}$, we use $1 \lesssim \mu^{-1/2} |k-k'|^{1/2} \mu^{1/3} |k'|^{-1/3}$. Then by Young's inequality, $\D_\tau \lesssim \E$, interpolation, and Lemma \ref{main_D_tau_lemma}
\begin{equation}\label{alpha_z1_HL_L_x}
\begin{split}
    T_{\alpha, Z_1, HL, (L, H)} & \lesssim \int_{|k| < \mu} \int_{|k'| \geq \mu} \iint_{\R^2}1_{HL}\dk^{J} \langle \xi \rangle^{n} \sk^m A(k) |\eta - kt|^2 |Z_\xi| \dkm^J \vkm^n  \\
        &\quad \quad \quad \quad \skm^{m-1/2} \mu^{-1/2} |k-k'|^{3/2} p_{\xi-\xi'}^{-3/4} |Z_{\xi - \xi'}| \vkp^n\\
        &\quad \quad \quad \quad A(k') \langle k' \rangle^{-1/2} |\eta' -k't| |Z_{\xi'}| d\eta' d\eta dk dk'\\
        &\lesssim ||\langle c \lambda_k t \rangle^J \langle \xi \rangle^n \sk^m A(k) |\eta-kt|^2 Z_\xi||_{L^1_{|k| < \mu} L^2_\eta} \\
        &\quad \quad  || \langle c \lambda_k t \rangle^J \vk^n \sk^m |k| \min(|k|^{1/2},1) p_\xi^{-3/4} Z_\xi||_{L^2_k L^1_\eta} || \vk^n  A(k) |\eta -kt| Z_\xi||_{L^2_\xi}\\
        &\lesssim \mu^{1/2-1/2}\D_\alpha^{1/2}(\D_\tau^{1/2} + \mu^{-\delta_*} \D_\gamma^{\delta_*} \D_{\tau}^{1/2 - \delta_*}) \E^{1/2}.
\end{split}
\end{equation}
Lastly, we examine $|k|, |k'| < \mu$. Note that this also implies $|k-k'| \lesssim \mu$. We have by interpolation, Young's inequality, and a variation on Lemma \ref{main_D_tau_lemma},
\begin{equation}
\begin{split}
    T_{\alpha, Z_1, HL, (L, L)}^x & \lesssim \int_{|k| < \mu} \int_{|k| < \mu} \iint_{\R^2}1_{HL}\dk^{J} \langle \xi \rangle^{n} \sk^m A(k) |\eta - kt|^2 |Z_\xi| \dkm^J  \vkm^n \\
        &\quad \quad \quad \quad \skm^{m} |k-k'| p_{\xi-\xi'}^{-3/4} |Z_{\xi - \xi'}| \vkp^n \langle k' \rangle^{-1/2} |\eta' -k't| |Z_{\xi'}| d\eta' d\eta dk dk'\\
        &\lesssim ||\langle c \lambda_k t \rangle^J \langle \xi \rangle^n \sk^m A(k) |\eta-kt|^2 Z_\xi||_{L^2_\xi} || \langle c \lambda_k t \rangle^J \vk^n \sk^m |k| p_\xi^{-3/4} Z_\xi||_{L^1_{|k| \lesssim \mu} L^1_\eta}\\
        &\quad \quad \quad \quad \quad || \vk^n   \langle k \rangle^{-1/2} |\eta -kt| Z_\xi||_{L^2_{|k| < \mu} L^2_\eta}\\
        &\lesssim \mu^{-1/2}\D_\alpha^{1/2}\mu^{1/4}(\D_\tau^{1/2} + \mu^{-\delta_*} \D_\gamma^{\delta_*} \D_{\tau}^{1/2 - \delta_*}) \E^{1/2}.
\end{split}
\end{equation}

Having completed the estimates for $T_{\alpha, Z_1}^x$, we turn our attention to $T_{\alpha, Z_2}^x$. We divide $T_{\alpha, Z_2}^x$ into $LH$ and $HL$ terms in the usual manner:
\begin{equation}
    T_{\alpha, Z_2}^x = T_{\alpha, Z_2, LH}^x + T_{\alpha, Z_2, HL, (H, \cdot)}^x + T_{\alpha, Z_2, HL, (L, \cdot)}^x.
\end{equation}
For the $LH$ term, we use $|k| \approx |k'|$ on the domain of integration, together with interpolation and boundedness of the Riesz Transform, and compute via Young's inequality, $A(k) \lesssim 1$, and Lemma \ref{half_energy_half_alpha_lemma},
\begin{equation}
    \begin{split}
        T_{\alpha,  Z_2, \cdot, LH}^x &\lesssim \iiiint_{\R^4} 1_{LH}\dk^{J} \langle \xi \rangle^{n} \sk^m A(k)   |\eta - kt|^2 p_\xi^{-1/4} |Z_\xi| \vkm^n  \\
        &\quad \quad \quad \quad \skm^{-1/2} |k-k'| p_{\xi-\xi'}^{-1/2} |Z_{\xi - \xi'}| \dkp^J \vkp^n \skp^m A(k') |\eta' -k't| |Z_{\xi'}| d\eta' d\eta dk' dk\\
        &\lesssim || \dk^J \langle \xi \rangle^n \sk^m A(k)^{3/4} |\eta -kt|^{3/2} Z_\xi||_{L^2_\xi} || \vk^n \sk^{-1/2} |k| p_\xi^{-1/2} Z_\xi||_{L^1_k L^2_\eta}\\
        &\quad \quad ||\dk^J \vk^n\sk^m  A(k) |\eta -kt| Z_\xi||_{L^2_k L^1_\eta}\\
        &\lesssim \mu^{-1/4} \D_\alpha^{1/4} \E^{1/4} \D_\tau^{1/2} \left(\E^{1/4-\delta_*} \mu^{-1/4 -\delta_*} \D_\alpha^{1/4 +\delta_*} + \E^{1/4} \mu^{-1/4} \D_\gamma^{1/4}\right),
 \end{split}
\end{equation}
which is sufficient since $\D_\tau \lesssim \E$. We treat the $HL$ case in a similar fashion, Starting with $T_{\alpha,  Z_2, \cdot, HL, (H, \cdot)}^x$, we use \eqref{HL_alpha_key_equation}, interpolation, and Lemma \ref{half_energy_half_alpha_lemma} to compute
\begin{equation}
    \begin{split}
        T_{\alpha,  Z_2, \cdot, HL, (H, \cdot)}^x &\lesssim \int_{|k| \geq \mu}\iiint_{\R^3} 1_{HL}\dk^{J} \langle \xi \rangle^{n} \sk^m (1_{|k| < 1} |k|^{-1/2} + 1_{|k| \geq 1})A(k)  |\eta - kt|^{3/2} |Z_\xi| \dkm^J  \\
        &\quad \quad \quad \quad \vkm^n \skm^m |k-k'| p_{\xi-\xi'}^{-1/2} |Z_{\xi - \xi'}| \vkp^n A(k') \langle k' \rangle^{-1/2} |\eta' -k't| |Z_{\xi'}| d\eta' d\eta dk' dk\\
        &\lesssim || \dk^J \langle \xi \rangle^n \sk^m A(k) |\eta -kt|^{3/2} Z_\xi||_{L^2_{|k| \geq 1}L^2_\eta} || \dk^J \langle \xi \rangle^n \sk^m |k| p_\xi^{-1/2} Z_\xi||_{L^2_\xi}\\
        &\quad \quad || \vk^n A(k) \langle k \rangle^{-1/2} |\eta -kt| Z_\xi||_{L^1_k L^1_\eta}\\
        &\quad + || \dk^J \langle \xi \rangle^n \sk^m A(k) |k|^{-1/2} |\eta -kt|^{3/2} Z_\xi||_{L^1_{\mu \leq |k| < 1}L^2_\eta} \\
        &\quad \quad || \dk^J \langle \xi \rangle^n \sk^m |k| p_\xi^{-1/2} Z_\xi||_{L^2_\xi} || \vk^n A(k) |\eta -kt| Z_\xi||_{L^2_k L^1_\eta}\\
        &\lesssim \mu^{-1/4} \D_\alpha^{1/4} \E^{1/4} \D_\tau^{1/2} \left(\E^{1/4-\delta_*} \mu^{-1/4 -\delta_*} \D_\alpha^{1/4 +\delta_*} + \E^{1/4} \mu^{-1/4} \D_\gamma^{1/4}\right)\\
        &\quad + \mu^{-1/4} \ln(1/\nu)^{1/2}\D_\alpha^{1/4} \E^{1/4} \D_\tau^{1/2} \left(\E^{1/4-\delta_*} \mu^{-1/4 -\delta_*} \D_\alpha^{1/4 +\delta_*} + \E^{1/4} \mu^{-1/4} \D_\gamma^{1/4}\right).
 \end{split}
\end{equation}
For the low frequency case of $T_{\alpha,  Z_2, \cdot, HL, (L, \cdot)}^x$, when we apply Young's inequality, we crucially place the $k$-factors in $L^1_{|k| < \mu}$ in order to gain powers of $\mu$ by H\"older's inequality. Applying Lemma \ref{main_D_tau_lemma} and $\D_\tau \lesssim \E$, we find
\begin{equation}
    \begin{split}
        T_{\alpha,  Z_2, \cdot, HL, (L, \cdot)}^x &\lesssim \int_{|k| < \mu} \iiint_{\R^3} 1_{HL}\dk^{J} \langle \xi \rangle^{n} \sk^m  |\eta - kt|^{3/2} |Z_\xi| \dkm^J \vkm^n \skm^m \\
        &\quad \quad \quad \quad |k-k'|^{1/2} \min(|k-k'|^{1/2},1) p_{\xi - \xi'}^{-1/2}|Z_{\xi - \xi'}| \vkp^n \langle k' \rangle^{-1/2}  |\eta' -k't| |Z_{\xi'}| d\eta' d\eta dk' dk\\
        &\lesssim || \dk^J \langle \xi \rangle^n \sk^m A(k) |\eta -kt|^{3/2} Z_\xi||_{L^1_{|k| < \mu}L^2_\eta} \\
        &\quad \quad || \dk^J \langle \xi \rangle^n \sk^m |k|^{1/2} \min(|k|^{1/2},1) p_\xi^{-1/2} Z_\xi||_{L^2_k L^1_\eta} || \vk^n \langle k \rangle^{-1/2} |\eta -kt| Z_\xi||_{L^2_\xi}\\
        &\lesssim \mu^{1/2}\E^{1/4} \mu^{-1/4}\D_\alpha^{1/4}(\E^{1/4}\D_\tau^{1/4} + \mu^{-\delta_*} \D_\gamma^{-\delta_*} \E^{1/4} \D_{\tau}^{1/4 - \delta_*}) \mu^{-1/2} \D_\gamma^{1/2}\\
        &\lesssim \mu^{-1/2-\delta_*} \D \E^{1/2}.
 \end{split}
\end{equation}
This completes the $T_{\alpha, Z}^x$ terms.
\subsubsection{y-derivatives}
Performing the standard decomposition and using boundedness of $N_k$ and $\mathfrak{J}_k$, we see
\begin{equation}
\begin{split}
    |T_{\alpha,Z}^y| &\lesssim \iiiint_{\R^4} \dk^{2J} \langle \xi \rangle^{n} \sk^{2m} A(k)^2 |\eta - kt|^2  \langle k \rangle^{1/2} |Z_\xi| \vkm^n \langle k-k' \rangle^{-1/2} p_{\xi-\xi'}^{-3/4}\\
    & \quad \quad \quad \quad |(\eta - \eta' + (k-k')t) Z_{\xi-\xi'}| \vkp^n \langle k' \rangle^{-1/2} |k'||Z_{\xi'}| d\eta' d\eta dk' dk\\
        &\quad + \iiiint_{\R^4} \dk^{2J} \langle \xi \rangle^{n} \sk^{2m} A(k)^2 |\eta - kt|^2 \langle k \rangle^{1/2} p_{\xi}^{-1/4}|Z_\xi| \vkm^n \langle k-k' \rangle^{-1/2} p_{\xi-\xi'}^{-1/2} \\
        &\quad \quad \quad \quad \quad |(\eta - \eta' + (k-k')t) Z_{\xi-\xi'}| \vkp^n \langle k' \rangle^{-1/2} |k'| |Z_{\xi'}| d\eta' d\eta dk' dk\\
        &\coloneqq T_{\alpha, Z_1}^y + T_{\alpha, Z_2}^y.
\end{split} 
\end{equation}
We start with $T_{\alpha, Z_1}^y$ which we split according to
$$T_{\alpha, Z_1}^y = T_{\alpha, Z_1,LH}^y +  T_{\alpha, Z_1, (H \cdot)}^y + T_{\alpha, Z_1, (L, \cdot)}^y$$
in the usual manner. Next, we observe that in the $LH$ regime, $A(k) \lesssim A(k-k')^{\delta_*}A(k')^{1 - \delta_*}$. Then by Young's inequality, interpolation, boundedness of the Reisz Transform, and Lemma \ref{alpha_energy_control}:
\begin{equation}\label{T_alpha_LH_y_Z1}
    \begin{split}
        T_{\alpha,Z_1, LH}^y &\lesssim \iiiint_{\R^4} 1_{LH}\dk^{J} \langle \xi \rangle^{n} \sk^m A(k) |\eta - kt|^2  |Z_\xi|  \vkm^n \langle k-k' \rangle^{-1/2} p_{\xi-\xi'}^{-3/4}\\
        &\quad \quad \quad \quad A(k-k')^{\delta_*} |(\eta - \eta' + (k-k')t) Z_{\xi-\xi'}| \\
        &\quad \quad \quad \quad \dkp^J \langle \xi' \rangle^n \skp^m A(k')^{1-\delta_*} |k'| |Z_{\xi'}| d\eta' d\eta dk' dk\\
        &\lesssim ||\dk^J \vk^n \sk^m A(k) |\eta -kt|^2 Z_\xi||_{L^2_\xi}\\
        &\quad \quad || \sk^{-1/2} A(k)^{\delta_*} p_{\xi}^{-1/4} Z_\xi||_{L^1_\xi} || \dk^J \vk^n \sk^m A(k)^{1-\delta_*}|k| Z_\xi||_{L^2_\xi}\\
        &\lesssim \mu^{-1/2} \D_\alpha^{1/2}(\E^{1/2} +  \mu^{-\delta_*/2} \E^{1/2}) \left(\mu^{-\delta_*/2}\D_\gamma^{1/4(1+\delta_*)} \D_\beta^{1/4(1-\delta_*)} + \mu^{1/2} \D_\beta^{1/2}\right). 
    \end{split}
\end{equation}
For the $HL$ case at high-in-$k$ frequencies, we apply a variation of \eqref{HL_alpha_key_equation}. Then interpolation, boundedness of the Riesz Transform, and Young's inequality allow us to write
\begin{equation}\label{alpha_Z1_HL}
    \begin{split}
        T_{\alpha,Z_1, HL,(H,\cdot)}^y &\lesssim \int_{|k| \geq \mu}\iiint_{\R^3} 1_{HL}\dk^{J} \langle \xi \rangle^{n} \sk^m A(k)(|k|^{-1/2}1_{\mu \leq |k| < 1} + 1_{|k| \geq 1}) |\eta - kt|^2  |Z_\xi|  \\
        &\quad \quad \quad \quad \dkm^J \vkm^n \sk^m A(k-k')^{1/2}|k-k'|  |(\eta - \eta') - (k-k')t|^{1/2} \\
        &\quad \quad \quad \quad p_{\xi-\xi'}^{-1/2} |Z_{\xi-\xi'}| \vkp^n \skp^{-1/2}  A(k')^{1/2} |Z_{\xi'}| d\eta' d\eta dk' dk\\
        &\lesssim \mu^{-1/2} \D_\alpha^{1/2} \D_\tau^{1/4} \D_{\tau \alpha}^{1/4} || \vk^n \langle k \rangle^{-1/2}  A(k)^{1/2} Z_{\xi} ||_{L^1_\xi}\\
        &\quad +  \mu^{-1/2} \ln(1/\nu)^{1/2} \D_\alpha^{1/2} \D_\tau^{1/4} \D_{\tau \alpha}^{1/4} || \vk^n \langle k \rangle^{-1/2}  A(k)^{1/2} Z_{\xi} ||_{L^2_k L^1_\eta}.
    \end{split}
\end{equation}
To handle the final factors of \eqref{alpha_Z1_HL}, we perform a simple interpolation (back in stationary coordinates) to find
\begin{equation}\label{simple_interp1}
\begin{split}
    || \vk^n \langle k \rangle^{-1/2}  A(k)^{1/2} Z_{\xi} ||_{L^1_\xi} &\lesssim || \vk^n  A(k)^{1/2} \langle \eta - kt \rangle^{1+ 2\delta_*} Z_{\xi} ||_{L^2_\xi}\\
    &\lesssim \E^{1/2} + \E^{1/2 - \delta_*} \mu^{\delta_*} \D_\gamma^{\delta_*},
\end{split}
\end{equation}
and
\begin{equation}\label{simple_interp2}
\begin{split}
    || \vk^n \langle k \rangle^{-1/2}  A(k)^{1/2} Z_{\xi} ||_{L^2_k L^1_\xi} &\lesssim || \vk^n  A(k)^{1/2} \langle \eta - kt \rangle^{1+ \delta_*} Z_{\xi} ||_{L^2_\xi}\\
    &\lesssim \E^{1/2} + \E^{1/2 - \delta_*/2} \mu^{\delta_*/2} \D_\gamma^{\delta_*/2}.
\end{split}
\end{equation}
Combining \eqref{alpha_Z1_HL}, \eqref{simple_interp1}, and \eqref{simple_interp2}, alongside $\D_\tau^{1/2} \lesssim \E^{1/2}$ gives
\begin{equation}
    T_{\alpha,Z_1, HL,(H,\cdot)}^y \lesssim \mu^{-1/2 - \delta_*} \D \E^{1/2},
\end{equation}
as desired. For the low frequency $HL$ case we proceed more simply. Here we have by interpolation, Young's inequality, and Lemma \ref{main_D_gamma_lemma},
\begin{equation}
    \begin{split}
        T_{\alpha,Z_1, HL,(L,\cdot)}^y &\lesssim \int_{|k| < \mu}\iiint_{\R^3} 1_{HL}\dk^{J} \langle \xi \rangle^{n} \sk^m  |\eta - kt|^2  |Z_\xi| \dkm^J \vkm^n \sk^m  \\
        &\quad \quad \quad \quad |k-k'| |(\eta - \eta') - (k-k')t|^{1/2} p_{\xi-\xi'}^{-1/2}|Z_{\xi-\xi'}|  \vkp^n \skp^{-1/2}  |Z_{\xi'}| d\eta' d\eta dk' dk\\
        &\lesssim \D_\alpha^{1/2} \D_\tau^{1/4} \D_{\tau \alpha}^{1/4} || \vk^n \langle k \rangle^{-1/2}  Z_{\xi} ||_{L^2_k L^1_\eta}\\
        &\lesssim \D_\alpha^{1/2} \D_\tau^{1/4} \D_{\tau \alpha}^{1/4}(\mu^{-1/4}\D_\gamma^{1/4} \E^{1/4} + \mu^{-1/4 - \delta_*} \D_\gamma^{1/4 + \delta_*} \E^{1/4+ \delta_*}).\\
        &\lesssim \mu^{-1/2 - \delta_*} \D \E^{1/2}.
    \end{split}
\end{equation}
We now turn our attention to $T_{\alpha, Z_2}^y$, which we split as
\begin{equation}
    T_{\alpha, Z_2}^y = T_{\alpha, Z_2, LH}^y + T_{\alpha, Z_2, HL, (H, \cdot)}^y+T_{\alpha, Z_2, HL, (L, \cdot)}^y.
\end{equation}
In the $LH$ case, we use a similar interpolation strategy as in \eqref{T_alpha_LH_y_Z1}. Namely, we use $A(k) \lesssim A(k-k')^{\delta_*} A(k')^{1-\delta_*}$. Then we use $|k| \approx |k'|$, interpolation, boundedness of the Riesz Transform, and Young's inequality, and Lemma \ref{main_D_gamma_lemma}:
\begin{equation}
    \begin{split}
        T_{\alpha, Z_2, LH}^y &\lesssim  \iiiint_{\R^4} 1_{LH}\dk^{J} \langle \xi \rangle^{n} \sk^m A(k) |\eta - kt|^{3/2} |Z_\xi| \vkm^n \langle k-k' \rangle^{-1/2} \\
        &\quad \quad \quad \quad \quad A(k-k')^{\delta_*} | Z_{\xi-\xi'}| \dkp^J \vkp^n \skp^m A(k')^{1 - \delta_*} |k'| |Z_{\xi'}| d\eta' d\eta dk' dk\\
        &\lesssim \mu^{-1/4} \D_\alpha^{1/4} \E^{1/4} ||\vk^n A(k)^{\delta_*} \langle k \rangle^{-1/2}  Z_{\xi} ||_{L^1_\xi} \left(\mu^{-\delta_*/2}\D_\gamma^{1/4(1+\delta_*)} \D_\beta^{1/4(1-\delta_*)} + \mu^{1/2} \D_\beta^{1/2}\right) \\
        &\lesssim \mu^{-1/4} \D_\alpha^{1/4} \E^{1/4} \mu^{-1/4}(1+ \mu^{ - \delta_*/2})\D_\gamma^{1/4} \E^{1/4}\left(\mu^{-\delta_*/2}\D_\gamma^{1/4(1+\delta_*)} \D_\beta^{1/4(1-\delta_*)} + \mu^{1/2} \D_\beta^{1/2}\right).
    \end{split}
\end{equation}
Our next estimate, on $T_{\alpha, Z_2, HL, (H, \cdot) }^y$ uses a variant of \eqref{HL_alpha_key_equation}, $|k| \lesssim |k-k'|$, interpolation, Young's inequality, boundedness of the Riesz Transform, and (a variant of) Lemma \ref{main_D_gamma_lemma} to compute
\begin{equation}\label{T_alpha_Z2_HL_H_y}
    \begin{split}
         T_{\alpha, Z_2, HL, (H, \cdot) }^y &\lesssim  \int_{|k| \geq \mu}\iiint_{\R^3} 1_{HL}\dk^{J} \langle \xi \rangle^{n} \sk^m A(k) (1_{\mu \leq |k| < 1}|k|^{-1/2} + 1_{|k| \geq 1})|\eta - kt|^{3/2} |Z_\xi|\\
         &\quad \quad \quad \quad \dkm^J  \langle \xi - \xi'\rangle^n \skm^m |\eta - \eta' + (k-k')t| p_{\xi-\xi'}^{-1/2}  A(k-k') |k-k'|\\
        &\quad \quad \quad \quad  \vkp^n | Z_{\xi-\xi'}|  \skp^{-1/2}|Z_{\xi'}| d\eta' d\eta dk' dk\\
        &\lesssim || \dk^J \vk^n \sk^m A(k) |\eta -kt|^{3/2} Z_\xi||_{L^2_{|k| \geq 1}L^2_\eta} \\
        &\quad || \dk^J \vk^n \sk^m A(k) |\eta-kt| |k| p_\xi^{-1/2}||_{L^2_\xi} || \vk^n \sk^{-1/2} Z_\xi||_{L^1_\xi}\\
        &\quad + || \dk^J \vk^n \sk^m |k|^{-1/2} A(k) |\eta -kt|^{3/2} Z_\xi||_{L^1_{\mu \leq |k| < 1}L^2_\eta} \\
        &\quad \quad || \dk^J \vk^n \sk^m A(k) |\eta-kt| |k| p_\xi^{-1/2}||_{L^2_\xi} || \vk^n \sk^{-1/2} Z_\xi||_{L^2_k L^1_\eta}\\
        &\lesssim \mu^{-1/4} \D_\alpha^{1/4} \E^{1/4} \D_{\tau \alpha}^{1/2} (\E^{1/4} \mu^{-1/4} \D_\gamma^{1/4} + \E^{1/4-\delta_*} \mu^{-1/4-\delta_*} \D_\gamma^{1/4+\delta_*})\\
        &\quad + || \dk^J \vk^n \sk^m A(k) |\eta -kt|^{3/2} Z_\xi||_{L^1_{\mu \leq |k| < 1}L^2_\eta} \\
        &\quad \quad \D_{\tau \alpha}^{1/2} (\E^{1/2}  + \E^{1/4-\delta_*/2} \mu^{-1/4-\delta_*/2} \D_\gamma^{1/4+\delta_*/2}).
    \end{split}
\end{equation}
The estimate \eqref{T_alpha_Z2_HL_H_y} is completed upon noting that $\D_{\tau \alpha} \lesssim \E$ and
\begin{equation}\label{alpha_mid_freq_interp}
    \begin{split}
        || \dk^J \vk^n \sk^m &|k|^{-1/2} A(k) |\eta -kt|^{3/2} Z_\xi||_{L^1_{\mu \leq |k| < 1}L^2_\eta} \lesssim\\
        &\ln(1/\nu)^{1/2}  \mu^{-1/4}\D_{\alpha}^{1/4} \min(\E^{1/4}, \mu^{-1/4}\D_\gamma^{1/4}).
    \end{split}
\end{equation}
Our next case is $T_{\alpha, Z_2, HL, (L, \cdot) }^y$. Here, we start by noting that $|k| < \mu$ implies $\langle k \rangle^{m-1/2} \lesssim 1$. Further using $|k'| \lesssim |k-k'|$, alongside Young's inequality, we use interpolation, boundedness of the Riesz Transform, and Young's inequality to obtain
\begin{equation}\label{low_alpha1}
    \begin{split}
        T_{\alpha, Z_2, HL, (L, \cdot) }^y &\lesssim  \int_{|k| < \mu}\iiint_{\R^3} 1_{HL}\dk^{J} \langle \xi \rangle^{n}  |\eta - kt|^{3/2} |Z_\xi| \dkm^J  \langle \xi - \xi'\rangle^n\\
         &\quad \quad \quad \quad  \skm^{-1/2} |k-k'| |(\eta-\eta') - (k-k')t|p_{\xi-\xi'}^{-1/2} | Z_{\xi-\xi'}|\\
         &\quad \quad \quad \quad \vkp^n \skp^{-1/2} |Z_{\xi'}| d\eta' d\eta dk' dk\\
        &\lesssim || \dk^J \vk^n |\eta -kt|^{3/2} Z_\xi||_{L^1_{|k| < \mu}L^2_\eta} \\
        &\quad \quad || \dk^J \vk^n \sk^{-1/2} |k| |\eta -kt| p_{\xi}^{-1/2} Z_\xi||_{L^2_\xi} || \vk^n \sk^{-1/2} Z_\xi||_{L^2_k L^1_\eta}.
    \end{split}
\end{equation}
Next, we observe that
\begin{equation}\label{low_alpha2}
    || \dk^J \vk^n A(k) |\eta -kt|^{3/2} Z_\xi||_{L^1_{|k| < \mu}L^2_\eta} \lesssim \D_{\alpha}^{1/4} \D_{\gamma}^{1/4}.
\end{equation}
Additionally, we note that
\begin{equation}\label{control_en_with_so}
    \langle k \rangle^{-1/2} \lesssim A(k)1_{|k| < \mu} + \mu^{-1/3}A(k)1_{|k| \geq \mu}.
\end{equation}
Thus we have
\begin{equation}\label{low_alpha3}
    || \dk^J \vk^n \sk^{-1/2} |\eta -kt| |k| p_{\xi}^{-1/2} Z_\xi||_{L^2_\xi} \lesssim (1+\mu^{-1/3}) D_{\tau \alpha}^{1/2}.
\end{equation}
By interpolation in $\eta$, usage of stationary coordinates, and \eqref{control_en_with_so} again,
\begin{equation}\label{low_alpha4}
    || \vk^n \sk^{-1/2} Z_\xi||_{L^2_k L^1_\eta} \lesssim (1+ \mu^{-(1/2 + \delta_*)/3})\E^{1/2}.
\end{equation}
Applying \eqref{low_alpha2}, \eqref{low_alpha3}, and \eqref{low_alpha4} to \eqref{low_alpha1} yields
\begin{equation}
     T_{\alpha, Z_2, HL, (L, \cdot) }^y \lesssim \D_\alpha^{1/4} \D_\gamma^{1/4} (1+ \mu^{-1/3}) \D_{\tau \alpha}^{1/2}(1+ \mu^{-(1/2 + \delta_*)/3} )\E^{1/2} \lesssim \mu^{-1/2 - \delta_*} \D \E^{1/2}.
\end{equation}
This completes the estimates on $T_{\alpha, Z}^y$.

\subsection{Alpha Terms for Q}

We write explicitly
\begin{equation}
    \begin{split}
        T_{\alpha,Q} &=  -c_\alpha\iint_{\R^2} \frac{\dk^{2J}}{M_k(t)} \langle \xi \rangle^{2n} \sk^{2m} A(k)^2  (N_k + c_\tau\mathfrak{J}_k)(\eta -kt)^2\mathrm{Re}\biggl(\iint_{\R^2} \langle k \rangle^{1/2} p_\xi^{1/4} \bar{Q}_\xi\\
    &\quad \quad \quad \quad \quad \quad \skm^{-1/2} p_{\xi -\xi'}^{-3/4} i (k-k') Z_{\xi-\xi'} \sgn(k') \skp^{-1/2}p_{\xi'}^{-1/4}  i(\eta'-k't) Q_{\xi'}d\eta' dk'\biggr) d\eta dk\\
    &\quad +c_\alpha\iint_{\R^2} \frac{\dk^{2J}}{M_k(t)} \langle \xi \rangle^{2n} \sk^{2m} A(k)^2 (N_k + c_\tau\mathfrak{J}_k)(\eta -kt)^2\mathrm{Re}\biggl(\iint_{\R^2} \sgn(k) \langle k \rangle^{1/2} p_\xi^{1/4}  \bar{Q}_\xi\\
    &\quad \quad \quad \quad \quad \quad \skm^{-1/2} p_{\xi -\xi'}^{-3/4} i ((\eta-\eta') - (k-k')t) Z_{\xi-\xi'} \sgn(k') \skp^{-1/2}p_{\xi'}^{-1/4}ik' Q_{\xi'}d\eta' dk'\biggr) d\eta dk\\
    &\eqqcolon T_{\alpha, Q}^x + T_{\alpha, Q}^y.
    \end{split}
\end{equation}

\subsubsection{x-derivatives}
Performing the initial splitting on the Fourier side based on the triangle inequality and boundedness of multipliers, we have
\begin{equation}
    \begin{split}
        |T_{\alpha, Q}^x| &\lesssim \iiiint_{\R^4} \dk^{2J} \langle \xi \rangle^{n} \sk^{2m} A(k)^2  \sk^{1/2} |\eta - kt|^2 |Q_\xi| \vkm^n \skm^{-1/2}|k-k'| \\
        &\quad \quad \quad \quad p_{\xi-\xi'}^{-3/4} |Z_{\xi - \xi'}| \vkp^n \skp^{-1/2}|\eta' -k't| |Q_{\xi'}| d\eta' d\eta dk' dk\\
        &\quad +\iiiint_{\R^4} \dk^{2J} \langle \xi \rangle^{n} \sk^{2m} A(k)^2  |k|^{1/2} |\eta - kt|^2  |Q_\xi| \vkm^n \skm^{-1/2} |k-k'| \\
        & \quad \quad \quad \quad \quad p_{\xi-\xi'}^{-1/2} |Z_{\xi - \xi'}| p_{\xi'}^{-1/4} \vkp^n \skp^{-1/2}|\eta' -k't| |Q_{\xi'}| d\eta' d\eta dk' dk\\
        &\eqqcolon T_{\alpha, Q_1}^x + T_{\alpha, Q_2}^x.
    \end{split}
\end{equation}
Similar to the $\gamma$ terms, $T_{\alpha, Q_1}^x$ can be controlled with the same techniques used to control $T_{\alpha, Z_1}^x$. We then split
\begin{equation}
    T_{\alpha, Q_2}^x = T_{\alpha, Q_2, LH}^x + T_{\alpha, Q_2, HL, (H, \cdot)}^x + T_{\alpha, Q_2, HL, (L, \cdot)}^x.
\end{equation}
We start with $T_{\alpha, Q_2, LH}^x$. First, since we are in the $LH$ case, $A(k) \approx A(k')$. Then by boundedness of the Riesz Transform, Lemma \ref{new_energy_lemmas}, Young's inequality, and interpolation
\begin{equation}
    \begin{split}
        T_{\alpha, Q_2, LH}^x &\lesssim \iiiint_{\R^4} 1_{LH}\dk^{2J} \langle \xi \rangle^{n} \sk^{m} A(k)   |\eta - kt|^2  |Q_\xi| \vkm^n \skm^{-1/2} |k-k'| p_{\xi-\xi'}^{-1/2} |Z_{\xi - \xi'}|\\
        &\quad \quad \quad \quad \dkp^J \vkp^n \sk^m A(k')|\eta' -k't|^{1/2} |Q_{\xi'}| d\eta' d\eta dk' dk\\
        &\lesssim \mu^{-1/2} \D_\alpha^{1/2}\D_\tau^{1/2} ||\dk^J \vk^n \sk^m A(k) |\eta - kt|^{1/2} Q_\xi||_{L^2_k L^1_\eta}\\
        &\lesssim \mu^{-1/2} \D_\alpha^{1/2}\D_\tau^{1/2} (\E^{1/2} + \mu^{-\delta_*} \D_{\alpha}^{\delta_*} \E^{1/2-\delta_*})\\
        &\lesssim \mu^{-1/2- \delta_*} \D\E^{1/2}.
    \end{split}
\end{equation}
For $T_{\alpha, Q_2, HL, (H, \cdot)}^x$, we use \eqref{HL_alpha_key_equation}. Then, noting the differences between $|k| \geq 1$ and $\mu < |k| \leq 1$, we apply Young's inequality and Lemma \ref{new_energy_lemmas},
\begin{equation}
    \begin{split}
        T_{\alpha, Q_2, HL, (H, \cdot)}^x &\lesssim \int_{|k| \geq \mu}\iiint_{\R^3} 1_{HL}\dk^{2J} \langle \xi \rangle^{n} A(k)(1_{\mu \leq |k| < 1} + 1_{|k| \geq \mu})   |\eta - kt|^2  |Q_\xi| \dkm^J\langle \xi - \xi'\rangle^n   \\
        &\quad \quad \quad \quad \skm^m |k-k'| p_{\xi-\xi'}^{-1/2} |Z_{\xi - \xi'}| \vkp^{n} \skp^{-1/2} A(k')|\eta' -k't|^{1/2} |Q_{\xi'}| d\eta' d\eta dk' dk\\
        &\lesssim \mu^{-1/2} \D_\alpha^{1/2} \D_{\tau}^{1/2} ||\dk^J \vk^n \sk^{-1/2} A(k) |\eta-kt|^{1/2} Q_\xi||_{L^1_\xi}\\
        &\quad + \mu^{-1/2} \ln(1/\nu)^{1/2} \D_\alpha^{1/2} \D_{\tau}^{1/2} ||\dk^J \vk^n \sk^{-1/2} A(k) |\eta-kt|^{1/2} Q_\xi||_{L^2_k L^1_\eta}\\
        &\lesssim \mu^{-1/2}\D_\alpha^{1/2} \D_{\tau}^{1/2} (\E^{1/2} + \mu^{-\delta_*} \D_\alpha^{\delta_*} \E^{1/2 - \delta_*})\\
        &\quad + \mu^{-1/2}\ln(1/\nu)^{1/2}\D_\alpha^{1/2} \D_{\tau}^{1/2} (\E^{1/2} + \mu^{-\delta_*/2} \D_\alpha^{\delta_*/2} \E^{1/2 - \delta_*/2})\\
        &\lesssim \mu^{-1/2 - \delta_*} \D \E^{1/2}.
    \end{split}
\end{equation}
The final case is $T_{\alpha, Q_2, HL, (L, \cdot)}^x$. Here we place the $k$-factors in $L^1_{|k|<\mu}$ with Young's inequality in order to gain $\mu^{1/2}$. This enables:
\begin{equation}
    \begin{split}
        T_{\alpha, Q_2, HL, (L, \cdot)}^x &\lesssim \int_{|k| < \mu}\iiint_{\R^3} 1_{HL}\dk^{2J} \langle \xi \rangle^{n}  |\eta - kt|^2  |Q_\xi| \dkm^J\langle \xi - \xi'\rangle^n \skm^m   \\
        &\quad \quad \quad \quad \quad |k-k'| p_{\xi-\xi'}^{-1/2} |Z_{\xi - \xi'}| \vkp^n \skp^{-1/2} |\eta' -k't|^{1/2} |Q_{\xi'}| d\eta' d\eta dk' dk\\
        &\lesssim \mu^{1/2-1/2} \D_\alpha^{1/2} || \dk^J \vk^n \sk^m |k| p_\xi^{-1/2} Z_\xi||_{L^2_\xi}\\
        &\quad ||\dk^J \vk^n \sk^{-1/2} |\eta-kt|^{1/2} Q_\xi||_{L^2_k L^1_\xi}\\
        &\lesssim \D_\alpha^{1/2}\D_{\tau}^{1/2}||\dk^J \vk^n \sk^{-1/2} |\eta-kt|^{1/2} Q_\xi||_{L^2_k L^1_\eta}.
    \end{split}
\end{equation}
Applying \eqref{control_en_with_so} alongside Lemma \ref{new_energy_lemmas}, we find
\begin{equation}
    \begin{split}
        T_{\alpha, Q_2, HL, (L, \cdot)}^x \lesssim \D_\alpha^{1/2}\D_{\tau}^{1/2}(\mu^{-1/6}\E^{1/2} + \mu^{-1/3 - \delta_*} \D_{\alpha}^{\delta_*} \E^{1/2 - \delta_*}) \lesssim \mu^{-1/2}\D \E^{1/2}.
    \end{split}
\end{equation}
This completes the estimates on $T_{\alpha, Q}^x$.

\subsubsection{y-derivatives}
We split $T_{\alpha, Q}^y$ as
\begin{equation}
\begin{split}
    |T_{\alpha,Q}^y| &\lesssim \iiiint_{\R^4} \dk^{2J} \langle \xi \rangle^{n} \sk^{2m} A(k)^2 |\eta - kt|^2  \sk^{1/2} |Q_\xi| \vkm^n \skm^{-1/2} p_{\xi-\xi'}^{-3/4}\\
    &\quad \quad \quad \quad |(\eta - \eta' + (k-k')t) Z_{\xi-\xi'}| \vkp^n \skp^{-1/2} |k'| |Q_{\xi'}| d\eta' d\eta dk' dk\\
        &\quad + \iiiint_{\R^4} \dk^{2J} \langle \xi \rangle^{n} \sk^{2m} A(k)^2 |\eta - kt|^2 \sk^{1/2} |Q_\xi| \vkm^n \skm^{-1/2} p_{\xi-\xi'}^{-1/2} \\
        &\quad \quad \quad \quad \quad |(\eta - \eta' + (k-k')t) Z_{\xi-\xi'}| \vkp^n p_{\xi'}^{-1/4}\skp^{-1/2} |k'| |Q_{\xi'}| d\eta' d\eta dk' dk\\
        &\coloneqq T_{\alpha, Q_1}^y + T_{\alpha, Q_2}^y.
\end{split} 
\end{equation}
For reasons previously discussed, we will only present the bound on $T_{\alpha, Q_2}^y$. We further decompose
$$T_{\alpha, Q_2}^y = T_{\alpha, Q_2, LH}^y + T_{\alpha, Q_2, HL, (H, \cdot)}^y + T_{\alpha, Q_2, HL, (L, \cdot)}^y.$$
Using interpolation, Young's inequality, $|k| \approx |k'|$, and Lemma \ref{strange_interp} we find
\begin{equation}
    \begin{split}
        T_{\alpha, Q_2, LH}^y &\lesssim \iiiint_{\R^4} 1_{LH} \dk^{J} \langle \xi \rangle^{n} \sk^m A(k) |\eta - kt|^2 |Q_\xi| \vkm^n \skm^{-1/2}\\
        &\quad \quad \quad \quad \quad | Z_{\xi-\xi'}|\dkp^J \vkp^n \skp^m A(k') p_{\xi'}^{-1/4} |k'| |Q_{\xi'}| d\eta' d\eta dk' dk\\
        &\lesssim \mu^{-1/2} \D_\alpha^{1/2} ||\vk^n \sk^{-1/2} Z_\xi||_{L^1_k L^2_\eta} || \dk^J \vk^n \sk^m |k| A(k) p_\xi^{-1/4} Q_\xi||_{L^2_k L^1_\eta}\\
        &\lesssim \mu^{-1/2}\D_\alpha^{1/2} \E^{1/2}( \D_\tau^{1/4}\D_\beta^{1/4} + \mu^{-\delta_*}\D_\alpha^{\delta_*}\D_{\tau \alpha}^{1/4}(\D_\beta^{1/8} \D_\gamma^{1/8 - \delta_*}  + \D_\beta^{1/4 - \delta_*})).
    \end{split}
\end{equation}
Next, we consider $T_{\alpha, Q_2, HL, (H, \cdot)}^y$. Here we employ a variant of \eqref{HL_alpha_key_equation}, $|k|, |k'| \lesssim |k-k'|$, Young's inequality, interpolation, and Lemma \ref{main_D_tau_alpha_lemma} to find
\begin{equation}
    \begin{split}
        T_{\alpha, Q_2, HL, (H, \cdot)}^y &\lesssim \int_{|k| \geq \mu}\iiint_{\R^3} 1_{HL} \dk^{J} \langle \xi \rangle^{n} \sk^m A(k)(1_{\mu \leq |k| < 1}|k|^{-1/2} + 1_{|k| \geq 1}) |\eta - kt|^2 |Q_\xi| \dkm^J\\
        &\quad \quad \quad \quad \langle \xi - \xi'\rangle^n \skm^m  p_{\xi - \xi'}^{-1/2} A(k-k') |k-k'|^{1/2} |(\eta - \eta') -(k-k')t|^{1/2} | Z_{\xi-\xi'}|\\
        &\quad \quad \quad \quad  \vkp^n \skp^{-1/2} p_{\xi'}^{-1/4} |k'|^{1/2} |Q_{\xi'}| d\eta' d\eta dk' dk\\
        &\lesssim \mu^{-1/2} \D_\alpha^{1/2} ||\dk^J \langle \xi \rangle^n \sk^m A(k) |k|^{1/2}|\eta - kt|^{1/2} p_{\xi}^{-1/4} Z_\xi||_{L^2_k L^1_\eta}\\
        &\quad||\vk^n \sk^{-1/2} |k|^{1/2} p_\xi^{-1/4} Q_\xi||_{L^1_k L^2_\eta}\\
        &\quad + \mu^{-1/2} \ln(1/\nu)^{1/2}\D_\alpha^{1/2} ||\dk^J \langle \xi \rangle^n \sk^m A(k) |k|^{1/2}|\eta - kt|^{1/2} p_{\xi}^{-1/4} Z_\xi||_{L^2_k L^1_\eta}\\
        &\quad \quad ||\vk^n \sk^{-1/2} |k|^{1/2} p_\xi^{-1/4} Q_\xi||_{L^2_\xi}\\
        &\lesssim \mu^{-1/2} \D_\alpha^{1/2} \left(\D_{\tau \alpha}^{1/4} \E^{1/4} + \D_{\tau \alpha}^{1/4} \mu^{-\delta_*} \D_\gamma^{\delta_*} \E^{1/4 - \delta_*}\right) \D_\tau^{1/4} \E^{1/4}\\
        &+ \mu^{-1/2} \ln(1/\nu)^{1/2} \D_\alpha^{1/2} \left(\D_{\tau \alpha}^{1/4} \E^{1/4} + \D_{\tau \alpha}^{1/4} \mu^{-\delta_*/2} \D_\gamma^{\delta_*/2} \E^{1/4 - \delta_*/2}\right) \D_\tau^{1/4} \E^{1/4}\\
        &\lesssim \mu^{-1/2 - \delta_*} \D \E^{1/2}.
    \end{split}
\end{equation}
Finally, we have $T_{\alpha, Q_2, HL, (L, \cdot)}^y$.  We proceed in a similar fashion as \eqref{low_alpha1}, exploiting the fact that $|k| < \mu$ implies $\sk^{m + 1/2} \lesssim 1$. Additionally, we have from $|k'| \lesssim |k-k'|$, boundedness of the Riesz Transform, interpolation, and \eqref{control_en_with_so} with Lemma \ref{main_D_tau_alpha_lemma}:
\begin{equation}
    \begin{split}
        T_{\alpha, Q_2, HL, (L, \cdot)}^y &\lesssim \int_{|k| < \mu}\iiint_{\R^3} 1_{HL} \dk^{J} \langle \xi \rangle^{n} |\eta - kt|^2 |Q_\xi| \dkm^J \langle \xi - \xi'\rangle^n \skm^{-1/2} \\
        &\quad \quad \quad \quad  |k-k'| |(\eta - \eta') - (k-k')t|^{1/2} p_{\xi-\xi'}^{-1/4}| Z_{\xi-\xi'}|\\
        &\quad \quad \quad \quad \vkp^n \skp^{-1/2} p_{\xi'}^{-1/4} |k'|^{1/2} |Q_{\xi'}| d\eta' d\eta dk' dk\\
        &\lesssim  \D_\alpha^{1/2} ||\dk^J \langle \xi \rangle^n \sk^{-1/2} |k|^{1/2} |\eta -kt|^{1/2} p_\xi^{-1/4} Z_\xi||_{L^2_k L^1_\eta}\\
        &\quad ||\vk^n \sk^{-1/2} |k|^{1/2} p_\xi^{-1/4} Q_\xi||_{L^2_\xi}\\
        &\lesssim \D_\alpha^{1/2}(\mu^{-1/6} \D_{\tau \alpha}^{1/4} \E^{1/4} + \mu^{-1/6 - \delta_*} \D_{\tau \alpha}^{1/4} \D_\gamma^{-\delta_*} \E^{1/4 - \delta_*}) \D_{\tau}^{1/4} \E^{1/4}\\
        &\lesssim \mu^{-1/2 - \delta_*} \D \E^{1/2}.
    \end{split}
\end{equation}
This concludes the discussion of the $T_{\alpha, Q}^y$ terms.

\subsection{Beta Terms for Z}

Writing out $T_{\beta, Z}$ \eqref{T_main} explicitly using \eqref{conv_NL_Z}, we have
\begin{equation}
\begin{split}
        T_{\beta,Z} &= -c_\beta \iint_{\R^2} \frac{\dk^{2J}}{M_k(t)} \langle k, \eta \rangle^{2n} \sk^{2m} B(k)^2 k (\eta-kt) \sk^{1/2} p_\xi^{-1/4} 2\mathrm{Re} \biggl( \iint_{\R^2} \bar{Z}_\xi \skm^{-1/2}\\
    &\quad \quad \quad \quad \quad \quad p_{\xi -\xi'}^{-3/4} i (k-k') Z_{\xi-\xi'} \skp^{-1/2} p_{\xi'}^{1/4} i(\eta'-k't) Z_{\xi'} dk' d\eta'\biggr) d\eta dk\\
    & \quad + c_\beta \iint_{\R^2} \frac{\dk^{2J}}{M_k(t)} \langle k, \eta \rangle^{2n} \sk^{2m} B(k)^2 k (\eta-kt) \sk^{1/2} p_\xi^{-1/4} 2\mathrm{Re} \biggl( \iint_{\R^2} \bar{Z}_\xi \skm^{-1/2}\\
    &\quad \quad \quad \quad \quad \quad p_{\xi -\xi'}^{-3/4} i ((\eta-\eta') - (k-k')t) Z_{\xi-\xi'} \skp^{-1/2} p_{\xi'}^{1/4} i k' Z_{\xi'} dk' d\eta' \biggr) d\eta dk\\
    & \eqqcolon T_{\beta, Z}^x + T_{\beta, Z}^y.
\end{split}
\end{equation}
We begin with $T_{\beta, Z}^x$.

\subsubsection{x-derivatives}
By the triangle inequality and boundedness of relevant Fourier multipliers
\begin{equation}\label{split_beta_z_x_terms}
    \begin{split}
         |T_{\beta, Z}^x| &\lesssim \iiiint_{\R^4} \dk^{2J}\langle \xi \rangle^{n} \sk^{2m} B(k)^2 \langle k \rangle^{1/2} |k| |\eta - kt| |Z_\xi| 
         \vkm^n \langle k-k' \rangle^{-1/2} |k-k'|  \\
         & \quad \quad \quad \quad p_{\xi - \xi'}^{-3/4} |Z_{\xi-\xi'}| \vkp^n \langle k' \rangle^{-1/2} |\eta' -k't| |Z_{\xi'}|  d\eta' d\eta dk' dk\\
        &\quad +\iiiint_{\R^4} \dk^{2J} \langle \xi\rangle^{n} \sk^{2m} B(k)^2 \sk^{1/2} |k| p_{\xi}^{-1/4} |\eta - kt| |Z_\xi| \vkm^n \skm^{-1/2} |k-k'|\\
        & \quad \quad \quad \quad \quad p_{\xi-\xi'}^{-1/2} |Z_{\xi-\xi'}| \vkp^n \skp^{-1/2}|\eta' - k't| |Z_{\xi'}| d\eta' d\eta dk' dk\\
        &\coloneqq T_{\beta,Z_1}^x+ T_{\beta,Z_2}^x.
    \end{split}
\end{equation}
To control $T_{\beta,Z_1}^x$, we split into $LH$ and $HL$ cases:
\begin{equation}
    T_{\beta,Z_1}^x = T_{\beta,Z_1, LH}^x + T_{\beta,Z_1, HL}^x.
\end{equation}
For the $LH$ cases, we use that $|k| \approx |k'|$, $B(k)^2 |k| \lesssim A(k)$, Young's inequality, interpolation, and Lemma \ref{main_D_tau_lemma}:
\begin{equation}
    \begin{split}
        T_{\beta,Z_1, LH}^x &= \iiiint_{\R^4} 1_{LH}\dk^{J}\langle \xi \rangle^{n} A(k) |\eta - kt| |Z_\xi|  \vkm^n \langle k-k' \rangle^{-1/2} |k-k'|\\
        &\quad \quad \quad \quad  p_{\xi - \xi'}^{-3/4} |Z_{\xi-\xi'}| \dkp^J \langle \xi'\rangle^n \sk^m |\eta' -k't| |Z_{\xi'}|  d\eta' d\eta dk' dk\\
        &\lesssim \E^{1/2} ||\sk^{-1/2} |k|p_\xi^{-3/4} Z_\xi||_{L^1_\xi} \mu^{-1/2} \D_\gamma^{1/2}\\
        &\lesssim \E^{1/2} \left(\D_\tau^{1/2} + \mu^{-\delta_*} \D_\tau^{1/2-\delta_*} \D_\gamma^{\delta_*}\right)\mu^{-1/2} \D_\gamma^{1/2}.
    \end{split}
\end{equation}
In the $HL$ case, we utilize a similar approach, only now using $|k| \lesssim |k-k'|$. Moreover, we place the $k'$ terms in $L^1_k$ in order to employ interpolation as follows:
\begin{equation}
    \begin{split}
        T_{\beta,Z_1, HL}^x &= \iiiint_{\R^4} 1_{HL}\dk^{J}\langle \xi \rangle^{n} A(k) |\eta - kt| |Z_\xi| \dkm^J\vkm^n\sk^m \\
        & \quad \quad \quad \quad |k-k'| \min(|k-k'|^{1/4}, 1) p_{\xi - \xi'}^{-3/4} |Z_{\xi-\xi'}| \vkp^n  |\eta' -k't| \skp^{-1/2}\\
        &\quad \quad \quad \quad \left(|k'|^{-1/4} 1_{|k-k'| < 1} + 1_{|k-k'| \geq 1}\right)|Z_{\xi'}|  d\eta' d\eta dk' dk\\
        &\lesssim \E^{1/2} ||\dk^J \vk^n \sk^m |k|
        \min(|k|^{1/4}, 1) p_\xi^{-3/4} Z_\xi||_{L^2_k L^1_\eta}\\
        &\quad || \vk^n \sk^{-1/2} \max(|k|^{-1/4},1) |\eta - kt| Z_\xi||_{L^1_k L^2_\eta}\\
        &\lesssim \E^{1/2} \left(\D_\tau^{1/2} + \mu^{-\delta_*} \D_\tau^{1/2-\delta_*} \D_\gamma^{\delta_*}\right)\mu^{-1/2} \D_\gamma^{1/2}.
    \end{split}
\end{equation}
Turning to $T_{\beta,Z_2}^x$, we use the standard $LH/HL$ splitting:
\begin{equation}
    T_{\beta,Z_2}^x = T_{\beta,Z_2, LH}^x + T_{\beta,Z_2, HL}^x.
\end{equation}
Using $B(k)^2|k| \lesssim A(k)$, we have by interpolation, Young's inequality, $|k-k'| \lesssim |k'|$, and Lemma \ref{main_D_tau_lemma}:
\begin{equation}\label{beta_Z2_LH}
    \begin{split}
        T_{\beta,Z_2, LH}^x &= \iiiint_{\R^4} 1_{LH}\dk^{J}\langle \xi \rangle^{n} \sk^m A(k) p_\xi^{-1/4} |k|^{1/2} |\eta - kt| |Z_\xi| \vkm^n \\
        &\quad \quad \quad \quad\skm^{-1/2} |k-k'|^{1/2}
         p_{\xi - \xi'}^{-1/2} |Z_{\xi-\xi'}| \dkp^J \sk^m |\eta' -k't| |Z_{\xi'}|  d\eta' d\eta dk' dk\\
         &\lesssim ||\langle c \lambda_k t \rangle^J \langle \xi \rangle^n \sk^m A(k)|\eta-kt| |k|^{1/2} p_\xi^{-1/4} Z_\xi||_{L^2_\xi}\\
         &\quad \quad \quad \quad || \vk^n \sk^{-1/2} |k|^{1/2}  p_\xi^{-1/2} Z_\xi||_{L^1_\xi} ||  \dk^J \vk^n \sk^m |\eta -kt| Z_{\xi} ||_{L^2_\xi}\\
        &\lesssim \D_{\tau\alpha}^{1/4} \E^{1/4} (\D_\tau^{1/4} \E^{1/4} + \mu^{-\delta_*} \D_\gamma^{\delta_*} \D_\tau^{1/4 - \delta_*} \E^{1/4}) \mu^{-1/2}\D_\gamma^{1/2}.
    \end{split}
\end{equation}
For the $HL$ bound, we use $|k'| \lesssim |k-k'|$, $B(k)^2 |k| \lesssim A(k) \lesssim 1$, interpolation, Young's inequality, boundedness of the Riesz Transform, and \ref{new_energy_lemmas}:
\begin{equation}
    \begin{split}
        T_{\beta,Z_2, HL}^x &= \iiiint_{\R^4} 1_{HL}\dk^{J}\langle \xi \rangle^{n} \sk^m A(k)   |\eta - kt|^{1/2} |Z_\xi| \dkm^J \langle \xi - \xi' \rangle^n \sk^m |k-k'|  \\
        &\quad \quad \quad \quad p_{\xi - \xi'}^{-1/2} |Z_{\xi-\xi'}| \vkp^n \langle k' \rangle^{-1/2} |\eta' -k't| |Z_{\xi'}|  d\eta' d\eta dk' dk\\
        &\lesssim ||\dk^J \vk^n \sk^m A(k) |\eta -kt|^{1/2} Z_\xi||_{L^2_k L^1_\xi }||\dk^J \langle \xi \rangle^n \sk^m |k| p_\xi^{-1/2} Z_\xi||_{L^2_\xi}\\
        &\quad \quad  || \vk^n \langle k \rangle^{-1/2} |\eta -kt| Z_\xi||_{L^1_k L^2_\eta}\\
        &\lesssim  \left(\E^{1/2}+ \mu^{-\delta_*} \D_\alpha^{\delta_*} \E^{1/2 - \delta_*}\right)\D_\tau^{1/2} \mu^{-1/2} \D_\gamma^{1/2},
    \end{split}
\end{equation}
thereby completing the estimates on $T_{\beta, Z}^x$.

\subsubsection{y-derivatives}
Splitting $T_{\beta, Z}^y$ through the triangle inequality and boundedness of $N_k$, $\mathfrak{J}_k$, and $M_k^{-1}$, we have
\begin{equation}\label{split_beta_z_y_terms}
    \begin{split}
         |T_{\beta, Z}^y| &\lesssim \iiiint_{\R^4} \dk^{2J}\langle \xi \rangle^{n} \sk^{2m} B(k)^2 \sk^{1/2} |k| |\eta - kt| |Z_\xi| \vkm^n \skm^{-1/2} p_{\xi - \xi'}^{-3/4}\\
         &\quad \quad \quad \quad |(\eta - \eta') - (k-k')t||Z_{\xi-\xi'}| \vkp^n \skp^{-1/2} |k'| |Z_{\xi'}|  d\eta' d\eta dk' dk\\
        &\quad +\iiiint_{\R^4} \dk^{2J} \langle \xi\rangle^{n} \sk^{2m} B(k)^2 |k| \sk^{1/2} p_{\xi}^{-1/4} |\eta - kt| |Z_\xi| \vkm^n \skm^{-1/2} p_{\xi-\xi'}^{-1/2} \\
         &\quad \quad \quad \quad |(\eta - \eta') - (k-k')t||Z_{\xi-\xi'}| \vkp^n \skp^{-1/2} |k'| |Z_{\xi'}| d\eta' d\eta dk' dk\\
        &\coloneqq T_{\beta,Z_1}^y+ T_{\beta,Z_2}^y.
    \end{split}
\end{equation}
We split $T_{\beta, Z_1}^y$ as
\begin{equation}
    T_{\beta, Z_1}^y = T_{\beta, Z_1, LH}^y + T_{\beta, Z_1, HL, (H, \cdot)}^y + T_{\beta, Z_1, HL, (L, H)}^y + T_{\beta, Z_1, HL, (L,L)}^y.
\end{equation}
In the $LH$ case, we use $|k| \approx |k'|$, $B(k) |k| \lesssim A(k) \lesssim A(k-k')^{\delta_*} A(k')^{1-\delta_*}$, interpolation, Young's inequality, boundedness of the Riesz transform,  and Lemma \ref{alpha_energy_control} to find
\begin{equation}
    \begin{split}
        T_{\beta, Z_1, LH}^y&\lesssim \iiiint_{\R^4} 1_{LH}\dk^{J}\langle \xi \rangle^{n} \sk^m |\eta - kt| |Z_\xi|  \vkm^n \skm^{-1/2} p_{\xi - \xi'}^{-1/4} \\
        & \quad \quad \quad \quad A(k-k')^{\delta_*}|Z_{\xi-\xi'}| A(k')^{1-\delta_*} \dkp^J \langle \xi'\rangle^n \skp^m |k'| |Z_{\xi'}|  d\eta' d\eta dk' dk\\
        &\lesssim \mu^{-1/2} \D_\gamma^{1/2} || \vk^n \sk^{-1/2} A(k)^{\delta_*} p_\xi^{-1/4} Z_\xi||_{L^1_\xi} \left(\mu^{-\delta_*/2}\D_\gamma^{1/4(1+\delta_*)} \D_\beta^{1/4(1-\delta_*)} + \mu^{1/2} \D_\beta\right)\\
        &\lesssim \mu^{-1/2} \D_\gamma^{1/2} (\E^{1/2} + \mu^{-\delta_*/2} \E^{1/2})\left(\mu^{-\delta_*/2}\D_\gamma^{1/4(1+\delta_*)} \D_\beta^{1/4(1-\delta_*)} + \mu^{1/2} \D_\beta\right).
    \end{split}
\end{equation}
For the $HL$ case of $T_{\beta, Z_1}^y$, we are able to move $k$-type derivatives onto the $k-k'$ factor since $|k|, |k'| \lesssim |k-k'|$. In particular, for $T_{\beta, Z_1, HL, (H,\cdot)}$, we use that $A(k)^{1/2}|k|^{1/2} \lesssim A(k-k')^{1/2} |k-k'|^{1/2}$. Then using interpolation, Young's inequality, Lemma \ref{main_D_gamma_lemma}, and $\D_\tau \lesssim \E$, we find:
\begin{equation}
    \begin{split}
        T_{\beta, Z_1, HL, (H, \cdot)}^y
        &\lesssim \int_{|k| \geq \mu}\iiint_{\R^3} 1_{HL} \dk^{J}\langle \xi \rangle^{n} \sk^m A(k)^{1/2} \max(|k|^{-1/6},1) |\eta - kt| |Z_\xi|\\
         &\quad \quad \quad \quad \dkm^J \vkm^n\skm^m 
        A(k-k')^{1/2} |k-k'| |(\eta - \eta') - (k-k')t|^{1/2} \\
        &\quad \quad \quad \quad  p_{\xi - \xi'}^{-1/2}|Z_{\xi-\xi'}| \vkp^n \skp^{-1/2} |Z_{\xi'}|  d\eta' d\eta dk' dk\\
        &\lesssim  ||\dk^J \vk^n \sk^m A(k)^{1/2}|\eta - kt| Z_\xi||_{L^2_{|k| \geq 1}L^2_\eta}\\
        &\quad ||\dk^J \sk^m A(k)^{1/2} |k| |\eta -kt|^{1/2} p_\xi^{-1/2} Z_\xi||_{L^2_{|k| \geq 1}L^2_\eta}|| \vk^n \sk^{-1/2} Z_\xi ||_{L^1_\xi}\\
        & \quad +  ||\dk^J \vk^n A(k)^{1/2} |k|^{-1/6} |\eta - kt| Z_\xi||_{L^1_{\mu \leq |k| < 1}L^2_\eta} \\
        &\quad \quad ||\dk^J \sk^m A(k)^{1/2} |k| |\eta -kt|^{1/2} p_\xi^{-1/2} Z_\xi||_{L^2_{ |k| \geq \mu}L^2_\eta} || \vk^n \sk^{-1/2}  Z_\xi ||_{L^2_k L^1_\eta}\\
        &\lesssim  \mu^{-1/4}\D_{\gamma}^{1/4} \E^{1/4} \D_\tau^{1/4} \D_{\tau \alpha}^{1/4} \mu^{-1/4} \D_\gamma^{1/4}(\E^{1/4} + \mu^{-\delta_*}\D_{\gamma}^{\delta_*} \E^{1/4 - \delta_*})\\
        &\quad + \mu^{-1/4}\D_{\gamma}^{1/4} \E^{1/4} \D_\tau^{1/4} \D_{\tau \alpha}^{1/4} \mu^{-1/4- \delta_*} \D_\gamma^{1/4 + \delta_*}\E^{1/4 - \delta_*}\\
        &\quad + \mu^{-1/2} \D_\gamma^{1/2} \D_\tau^{1/4} \D_{\tau \alpha}^{1/4} \E^{1/2}.
    \end{split}
\end{equation}
Next, we consider the low-in-$k$ cases. For $T_{\beta, Z_1, HL, (L, H)}^y$, we note that $|k| \leq |k'| \lesssim |k-k'|$. Then we have by Young's inequality and interpolation,
\begin{equation}
    \begin{split}
        T_{\beta, Z_1, HL, (L, H)}^y&\lesssim \int_{|k| <\mu} \int_{|k'| \geq \mu}\iint_{\R^2} 1_{HL} \dk^J   \vk^n  |\eta - kt| |Z_\xi| \dkm^J \vkm^n \\
        &\quad \quad \quad \skm^m |(\eta - \eta') - (k-k')t| p_{\xi - \xi'}^{-3/4}|Z_{\xi-\xi'}|\\
        &\quad \quad \quad \quad  \vkp^n \sk^m |k'| |Z_{\xi'}|  d\eta' d\eta dk' dk\\
        &\lesssim \int_{|k| <\mu} \int_{|k| \geq \mu}\iiint_{\R^3} 1_{HL}  \dk^J \vk^n |\eta - kt| |Z_\xi| \dkm^J \vkm^n \\
        &\quad \quad \quad \skm^m |(\eta - \eta') - (k-k')t| |k-k'|^{2/3} p_{\xi - \xi'}^{-3/4}|Z_{\xi-\xi'}|\\
        &\quad \quad \quad \quad \dkp^{J} \vkp^n \skp^m |k'|^{1/3} |Z_{\xi'}| d\eta' d\eta dk' dk\\
        &\lesssim || \dk^J \vk^n |\eta -kt| Z_\xi||_{L^1_{|k| \leq \mu} L^2_\eta} ||\dk^J \vk^n  \sk^m |k|^{2/3} |\eta - kt| p_\xi^{-3/4} Z_\xi||_{L^2_{|k| \geq \mu} L^1_\eta}\\
        &\quad \quad ||  \vk^n  \sk^m |k|^{1/3} Z_\xi ||_{L^2_\xi}\\
        &\lesssim  \mu^{1/2} \E^{1/2} (\mu^{-1/12} \D_\beta^{1/4} \D_\tau^{1/4} + \mu^{-1/3-\delta_*} \D_{\tau \alpha}^{1/2 - \delta_*} \D_\alpha^{\delta_*}) \mu^{-1/6} \D_\beta^{1/2},
    \end{split}
\end{equation}
where we have used
\begin{equation}
    \begin{split}
        || \vk^n  \sk^m |k|^{2/3} |\eta - kt| p_\xi^{-3/4}& Z_\xi||_{L^2_{|k| \geq \mu} L^1_\eta}\\
        &\lesssim || \vk^n  \sk^m |k|^{2/3} |\eta - kt| \langle \eta - kt \rangle^{1/2 + 2\delta_*} p_\xi^{-3/4} Z_\xi||_{L^2_{|k| \geq \mu} L^2_\eta}\\
        &\lesssim || \vk^n  \sk^m |k|^{2/3} |\eta - kt |^{1/2} p_\xi^{-1/2} Z_\xi||_{L^2_{|k| \geq \mu} L^2_\eta}\\
        &\quad \quad + || \vk^n  \sk^m |k|^{2/3} |\eta - kt|^{1 + 2\delta_*} p_\xi^{-1/2} Z_\xi||_{L^2_{|k| \geq \mu} L^2_\eta}\\
        &\lesssim \mu^{-1/12} \D_\beta^{1/4} \D_\tau^{1/4} + \mu^{-1/3-\delta_*} \D_{\tau \alpha}^{1/2 - \delta_*} \D_\alpha^{\delta_*}.
    \end{split}
\end{equation}
Lastly, we have the $ T_{\beta, Z_1, HL, (L, L)}^y$ case where all frequencies are low. Here we use Lemma \ref{main_D_tau_lemma}, boundedness of the Riesz Transform, Young's inequality, and interpolation to find
\begin{equation}
    \begin{split}
        T_{\beta, Z_1, HL, (L, L)}^y&\lesssim \int_{|k| <\mu} \int_{|k'| < \mu}\iint_{\R^2} 1_{HL}   \dk^J \vk^n \sk^m  |\eta - kt| |Z_\xi| \dkm^J \skm^m \vk^n\\
        &\quad \quad \quad |k-k'|^{3/4}p_{\xi - \xi'}^{-1/4}|Z_{\xi-\xi'}|   \vkp^n |Z_{\xi'}|  d\eta' d\eta dk' dk\\
        &\lesssim || \dk^J \vk^n \sk^m |\eta -kt| Z_\xi||_{L^2_{|k| \leq \mu} L^2_\eta} ||\dk^J \vk^n  \sk^m |k|^{3/4}  p_\xi^{-1/4} Z_\xi||_{L^2_{|k| \lesssim \mu} L^1_\eta}\\
        &\quad \quad ||\vk^n  Z_\xi ||_{L^1_{|k| < \mu}L^2_\eta}\\
        &\lesssim \mu^{-1/2} \D_\gamma^{1/2}(\mu^{-1/4} \D_\gamma^{1/4} \D_\tau^{1/4} + \mu^{-1/4 - \delta_*} \D_{\gamma}^{1/4 + \delta_*} \D_\tau^{1/4 - \delta_*}) \mu^{1/2} \E^{1/2}.
    \end{split}
\end{equation}
To address $T_{\beta, Z_2}^y$, we perform the decomposition
\begin{equation}
    T_{\beta, Z_2}^y = T_{\beta, Z_2, LH}^y + T_{\beta, Z_2, HL,(\cdot,H)}^y + T_{\beta, Z_2, HL,(\cdot, L)}^y.
\end{equation}
For the $LH$ case, we use Young's inequality, $|k| \approx |k'|$, interpolation, boundedness of the Riesz transform, $D_{\tau \alpha} \lesssim \E$, and Lemma \ref{main_D_gamma_lemma} to find:
\begin{equation}\label{beta_Z2_LH_y}
    \begin{split}
        T_{\beta, Z_2, LH}^y&\lesssim \iiiint_{\R^4} 1_{LH}\dk^{J} \langle \xi\rangle^{n}A(k) (|k|^{1/2}1_{|k| \geq \mu} + 1_{|k| < \mu})|\eta-kt| p_\xi^{-1/4} |Z_\xi|    \\
         &\quad \quad \quad \quad  \vkm^n \skm^{-1/2} |Z_{\xi-\xi'}| \dkp^J \skp^m \\
         &\quad \quad \quad \quad \vkp^n (|k'|^{1/2}1_{|k'| \geq \mu} + |k'| 1_{|k'| < \mu})|Z_{\xi'}| d\eta' d\eta dk' dk\\
         &\lesssim ||\dk^J \vk^n \sk^m (|k|^{1/2}1_{|k| \geq \mu} + 1_{|k| < \mu}) A(k) |\eta-kt| p_\xi^{-1/4} Z_\xi||_{L^2_\xi}|| \vk^n \sk^{-1/2} Z_\xi ||_{L^1_\xi}\\
         & \quad \quad ( \mu^{-1/4} \D_\gamma^{1/8} \D_\beta^{3/8} + \mu^{1/2} \D_\beta^{1/2})\\
         &\lesssim \D_{\tau \alpha}^{1/4} \E^{1/4} (\mu^{-1/4}\E^{-1/4}\D_{\gamma}^{-1/4} + \mu^{-1/4-\delta_*} \D_\gamma^{1/4 + \delta_*} \E^{1/4 - \delta_*}) (\mu^{-1/4} \D_\gamma^{1/8} \D_\beta^{3/8} + \mu^{1/2} \D_\beta^{1/2})\\
         &\lesssim \mu^{-1/2 - \delta_*} \D \E^{1/2}.
    \end{split}
\end{equation}
We now turn to $T_{\beta, Z_2, HL, (\cdot,H)}^y$. We use $|k| B(k)^{2} \lesssim A(k)$ and $|k|, |k'| \lesssim |k-k'|$.
We now employ interpolation, Young's inequality, and Lemma \ref{new_energy_lemmas}:
\begin{equation}\label{T_beta_Z2_HL_y_H}
    \begin{split}
        T_{\beta, Z_2, HL, (\cdot, H)}^y& \lesssim \int_{\R} \int_{|k'| \geq \mu}\iint_{\R^2} 1_{HL}\dk^{J} \langle \xi\rangle^{n} \sk^m A(k) |\eta - kt|^{1/2} |Z_\xi| \dkm^J \langle \xi - \xi'\rangle^n \sk^m p_{\xi -\xi'}^{-1/2} \\ &\quad \quad \quad \quad  |(\eta-\eta') - (k-k')t|   |k-k'|^{2/3} |Z_{\xi-\xi'}|  \vkp^n \langle k' \rangle^{-1/2} |k'|^{1/3}|Z_{\xi'}| d\eta' d\eta dk' dk\\
         &\lesssim ||\dk^J \vk^n \sk^m A(k) |\eta - kt|^{1/2} Z_\xi||_{L^2_k L^1_\eta}\\
         &\quad||\dk^J \vk^n \sk^m |k|^{2/3} |\eta -kt|p_\xi^{-1/2} Z_\xi||_{L^2_{|k| \geq \mu} L^2_\eta} || \vk^n \langle k \rangle^{-1/2} |k|^{1/3} Z_\xi ||_{L^1_{|k| \geq \mu} L^2_\eta}\\
         &\lesssim (\E^{1/2} + \mu^{-\delta_*} \D_\alpha^{\delta_*} \E^{1/2 - \delta_*}) \mu^{-1/3} \D_{\tau \alpha}^{1/2} \mu^{-1/6} \D_{\beta},
    \end{split} 
\end{equation}
which suffices as $\D_{\tau \alpha} \lesssim \E$. Our final term is estimated using boundedness of the Riesz Transform, $|k'| \lesssim |k-k'|$, Lemma \ref{main_D_gamma_lemma}, and interpolation:
\begin{equation}\label{T_beta_Z2_HL_y_L}
    \begin{split}
        T_{\beta, Z_2, HL, (\cdot, L)}^y& \lesssim \int_{\R} \int_{|k'| < \mu }\iint_{\R^2} 1_{HL}\dk^{J} \langle \xi\rangle^{n} \sk^m A(k) |\eta - kt| p_{\xi}^{-1/4} |Z_\xi| \dkm^J \langle \xi - \xi'\rangle^n \sk^m \\ &\quad \quad \quad \quad  \min(|k-k'|^{1/4},1) |Z_{\xi-\xi'}|  \vkp^n \langle k' \rangle^{-1/2} |k'| \max(|k'|^{-1/4},1)|Z_{\xi'}| d\eta' d\eta dk' dk\\
         &\lesssim ||\dk^J \vk^n \sk^m A(k) |\eta - kt|^{1/2} Z_\xi||_{L^2_\xi}||\dk^J \vk^n \sk^m \min(|k|^{1/4},1) Z_\xi||_{L^2_k L^1_\eta}\\
         &\quad || \vk^n \langle k \rangle^{-1/2} |k| \max(|k|^{-1/4},1) Z_\xi ||_{L^1_{|k| < \mu} L^2_\eta}\\
         &\lesssim \mu^{-1/4} \D_\gamma^{1/4} \E^{1/4} (\mu^{-1/4} \D_{\gamma}^{1/4}\E^{1/4} + \mu^{-1/4 - \delta_*} \D_{\gamma}^{1/4 + \delta_*} \E^{1/4 - \delta_*})\\
         &\quad || \vk^n \langle k \rangle^{-1/2} |k| \max(|k|^{-1/4},1) Z_\xi ||_{L^1_{|k| < \mu} L^2_\eta},
    \end{split} 
\end{equation}
which suffices once we note that
\begin{equation}
    || \vk^n \langle k \rangle^{-1/2} |k| \max(|k|^{-1/4},1) Z_\xi ||_{L^1_{|k| < \mu} L^2_\eta} \lesssim \D_\beta^{1/2 - \delta_*}\min(\D_\beta, \E)^{\delta_*}.
\end{equation}
This completes the bounds for the $\beta$ terms involving $Z$.

\subsection{Beta Terms for Q}\label{beta_terms_for_Q}
We expand $T_{\beta, Q}$ as 
\begin{equation}
\begin{split}
    T_{\beta,Q} &= c_\beta \iint_{\R^2} \frac{\dk^{2J}}{M_k(t)} \langle k, \eta \rangle^{2n} \sk^{2m}  \beta_k k (\eta-kt) 2\mathrm{Re}(\mathbb{NL}_k^{(Q)} \bar{Q}_k) d\eta dk\\
    &= -c_\beta \iint_{\R^2} \frac{\dk^{2J}}{M_k(t)} \langle k, \eta \rangle^{2n} \sk^{2m}  B(k)^2 k (\eta-kt) \sgn(k) \langle k \rangle^{1/2} p_\xi^{1/4} 2\mathrm{Re}\biggl(\iint_{\R^2} \bar{Q}_\xi \skm^{-1/2} p_{\xi -\xi'}^{-3/4}\\
    &\quad \quad \quad \quad i (k-k') Z_{\xi-\xi'}\sgn(k') \skp^{-1/2} p_{\xi'}^{-1/4} i(\eta'-k't) Q_{\xi'}d\eta' dk' \biggr) d\eta dk\\
    & \quad +c_\beta \iint_{\R^2} \frac{\dk^{2J}}{M_k(t)} \langle k, \eta \rangle^{2n} \sk^{2m}  B(k)^2 k (\eta-kt) \sgn(k) \sk^{1/2} p_\xi^{1/4} 2\mathrm{Re}\biggl(\iint_{\R^2} \bar{Q}_\xi \skm^{-1/2} p_{\xi -\xi'}^{-3/4}\\
    &\quad \quad \quad \quad i ((\eta-\eta') - (k-k')t) Z_{\xi-\xi'} \sgn(k') \skp^{-1/2} p_{\xi'}^{-1/4} i k' Q_{\xi'}d\eta' dk' \biggr) d\eta dk\\
    &\eqqcolon T_{\beta, Q}^x + T_{\beta, Q}^y.
\end{split}
\end{equation}

\subsubsection{x-derivatives}
By the triangle inequality
\begin{equation}\label{split_beta_q_x_terms}
    \begin{split}
         |T_{\beta, Q}^x| &\lesssim \iiiint_{\R^4} \dk^{2J}\langle \xi \rangle^{n} \vk^{2m} B(k)^2 \langle k \rangle^{1/2}  |k| |\eta - kt| |Q_\xi| \vkm^n \langle k-k'\rangle^{-1/2} |k-k'| \\
         &\quad \quad \quad \quad p_{\xi - \xi'}^{-3/4} |Z_{\xi-\xi'}| \vkp^n \langle k' \rangle^{-1/2} |\eta' -k't| |Q_{\xi'}|  d\eta' d\eta dk' dk\\
        &\quad +\iiiint_{\R^4} \dk^{2J} \langle \xi\rangle^{n} \sk^{2m} B(k)^2 \sk^{1/2} |k| |\eta - kt| |Q_\xi| \vkm^n \langle k-k' \rangle^{-1/2} |k-k'| p_{\xi-\xi'}^{-1/2}\\
        & \quad \quad \quad \quad \quad |Z_{\xi-\xi'}| \vkp^n \langle k' \rangle^{-1/2} p_{\xi}^{-1/4}|\eta' - k't| |Q_{\xi'}| d\eta' d\eta dk' dk\\
        &\coloneqq T_{\beta,Q_1}^x+ T_{\beta,Q_2}^x.
    \end{split}
\end{equation}
We note that $T_{\beta, Q_1}^x$ can be controlled with the same methods as $T_{\beta, Z_1}^x$, since $Q$ and $Z$ satisfy the same linear estimates. Furthermore, if we split $T_{\beta, Q_2}^x$ as
\begin{equation}
    T_{\beta, Q_2}^x = T_{\beta, Q_2, LH}^x + T_{\beta, Q_2, HL}^x,
\end{equation}
    then $T_{\beta, Q_2, LH}^x$ is estimated in the same manner as $T_{\beta, Q_2, LH}^x$, since we may reverse the roles of $k$ and $k'$ as $|k| \approx |k'|$. The only remaining term is $T_{\beta, Q_2, HL}^x$. Here, we use $|k| \lesssim |k-k'|$, $B(k)^2 |k| \lesssim A(k)$, interpolation and boundedness of the Reisz transform. Additionally, we use Young's inequality to place the $k$ factor in $L^1_\eta$ and then $L^2_k$, along with Lemma \ref{half_energy_half_alpha_lemma}:
\begin{equation}
\begin{split}
        T_{\beta, Q_2, HL}^x &\lesssim \iiiint_{\R^4}1_{HL}\dk^{J} \langle \xi\rangle^{n} \sk^m A(k)  |\eta - kt| |Q_\xi| \dkm^{J} \langle \xi - \xi'\rangle^{n} \skm^m |k-k'| p_{\xi-\xi'}^{-1/2} \\
        &\quad \quad \quad \quad |Z_{\xi-\xi'}| \vkp^n \langle k' \rangle^{-1/2} |\eta' - k't|^{1/2} |Q_{\xi'}| d\eta' d\eta dk' dk\\
        &\lesssim ||\dk^{J} \langle \xi\rangle^{n} \sk^m A(k)  |\eta - kt| Q_\xi||_{L^2_k L^1_\eta} D_\tau^{1/2} ||  \vk^n \langle k \rangle^{-1/2} |\eta - kt|^{1/2} Q_{\xi}||_{L^1_k L^2_\eta}\\
        &\lesssim ||\dk^{J} \langle k, \eta + kt\rangle^{n} \sk^m A(k)  \eta q_\xi||_{L^2_k L^1_\eta} \D_\tau^{1/2} \mu^{-1/4} \E^{1/4} \D_\gamma^{1/4}\\
        &\lesssim ||\dk^{J} \langle k, \eta + kt\rangle^{n}\sk^m \langle \eta \rangle^{1/2 + 2 \delta_* } A(k)  \eta q_\xi||_{L^2_k L^2_\eta} \mu^{-1/4} \D_\tau^{1/2} \mu^{-1/4} \E^{1/4} \D_\gamma^{1/4}\\
        &\lesssim (\mu^{-1/4} \E^{1/4} \D_\gamma^{1/4} + \mu^{-1/4 - \delta_*} \E^{1/4 - \delta_*} \D_\alpha^{1/4 + \delta_*}) \D_\tau^{1/2} \mu^{-1/4} \E^{1/4} \D_\gamma^{1/4},
\end{split}
\end{equation}
which suffices as $\D_\tau \lesssim E$. We have now completed the $T_{\beta, Q}^x$ terms.

\subsubsection{y-derivatives}
The $T_{\beta, Q}^y$ terms will, similar to the $T_{\beta, Q}^x$ terms, require only a small number of new calculations. By boundedness of the relevant Fourier multipliers, we have
\begin{equation}
    \begin{split}
         |T_{\beta, Q}^y| &\lesssim \iiiint_{\R^4} \dk^{2J}\langle \xi \rangle^{n} \sk^{2m} B(k)^2 \langle k \rangle^{1/2} |k||\eta - kt| |Q_\xi| \vkm^n \langle k-k' \rangle^{-1/2} p_{\xi - \xi'}^{-3/4}\\
         &\quad \quad \quad \quad |(\eta - \eta') - (k-k')t||Z_{\xi-\xi'}| \vkp^n \skp^{-1/2} |k'| |Q_{\xi'}|  d\eta' d\eta dk' dk\\
        &\quad +\iiiint_{\R^4} \dk^{2J} \langle \xi\rangle^{n} \sk^{2m} B(k)^2 \sk^{1/2} |k| |\eta - kt| |Q_\xi| \vkm^n \skm^{-1/2} p_{\xi-\xi'}^{-1/2} \\
         &\quad \quad \quad \quad |(\eta - \eta') - (k-k')t||Z_{\xi-\xi'}| \vkp^n \skp^{-1/2} |k'|p_{\xi}^{-1/4} |Q_{\xi'}| d\eta' d\eta dk' dk\\
        &\coloneqq T_{\beta,Q_1}^y+ T_{\beta,Q_2}^y.
    \end{split}
\end{equation}
The term $T_{\beta, Q_1}^y$ can be estimated in the same way as $T_{\beta, Z_1}^y$. Splitting $T_{\beta,Q_2}^y$ into $LH$ and $HL$ as
\begin{equation}
    T_{\beta,Q_2}^y = T_{\beta,Q_2, LH}^y + T_{\beta,Q_2, HL}^y,
\end{equation}
we see that $T_{\beta, Q_2, LH}$ can be treated the same as $T_{\beta, Z_2, LH}^y$, utilizing $|k| \approx |k'|$. This leaves the $HL$ term, which we further split as
$$T_{\beta,Q_2, HL}^y = T_{\beta,Q_2, HL,(H, \cdot)}^y + T_{\beta,Q_2, HL, (L, \cdot)}^y.$$
For $T_{\beta,Q_2, HL,(H, \cdot)}^y$, we use $|k|, |k'| \lesssim |k-k'|$, $B(k)^2|k| \lesssim A(k)$, Young's inequality, and boundedness of the Riesz transform to compute as follows and (a variation on) Lemma \ref{main_D_tau_lemma}
\begin{equation}
\begin{split}
       T_{\beta, Q_2, HL, (H, \cdot)}^y & \lesssim \int_{|k| \geq \mu }\iiint_{\R^3} \dk^{J} \langle \xi\rangle^{n} \sk^m A(k)^{1/2} \max(|k|^{-1/6},1) |\eta - kt| |Q_\xi| \dkm^J\\
       &\quad \quad \quad \quad \langle \xi - \xi'\rangle^n \skm^m  A(k-k')^{1/2} |k-k'|^{1/2} |(\eta - \eta') - (k-k')t|^{1/2} \\
         &\quad \quad \quad \quad p_{\xi-\xi'}^{-1/4} | |Z_{\xi-\xi'}| \langle k' \rangle^{-1/2} |k'|^{1/2}p_{\xi}^{-1/4} |Q_{\xi'}| d\eta' d\eta dk' dk\\
         &\lesssim ||\dk^{J} \vk^n \sk^m A(k)^{1/2} |\eta - kt| Q_\xi||_{L^2_{|k| \geq 1} L^2_\eta} \\
         &\quad \quad ||\dk^{J}\langle \xi \rangle^n \sk^m A(k)^{1/2} |k|^{1/2} p_\xi^{-1/4} Z_\xi ||_{L^2_{\xi}} || \langle k \rangle ^{-1/2} |k|^{1/2} p_\xi^{-1/4} Q_\xi||_{L^1_\xi}\\
         &\quad + ||\dk^{J} \vk^n \sk^m A(k)^{1/2} |\eta - kt| Q_\xi||_{L^1_{\mu \leq |k| < 1} L^2_\eta} \\
         &\quad \quad \quad ||\dk^{J}\langle \xi \rangle^n \sk^m A(k)^{1/2} |k|^{1/2} |\eta - kt|^{1/2} p_\xi^{-1/4} Z_\xi ||_{L^2_{\xi}}|| \langle k \rangle^{-1/2} |k|^{1/2}  p_\xi^{-1/4} Q_\xi||_{L^2_k L^1_\eta}\\
         &\lesssim \E^{1/4} \mu^{-1/4} \D_{\gamma}^{1/4} \D_{\tau \alpha}^{1/4} \E^{1/4}(\D_\tau^{1/4} \mu^{-1/4} \D_\gamma^{1/4} + \D_\tau^{1/4} \mu^{-1/4-\delta_*} \D_\gamma^{1/4 + \delta_*})\\
         &\quad +||\dk^{J} \vk^n \sk^m A(k)^{1/2} |\eta - kt| Q_\xi||_{L^2_{\mu \leq |k| < 1} L^2_\eta} \D_{\tau \alpha}^{1/4} \E^{1/4}\\
         &\quad \quad (\D_\tau^{1/4} \mu^{-1/8} \D_\gamma^{1/8} \E^{1/8} + \D_\tau^{1/4} \mu^{-1/4-\delta_*} \D_\gamma^{1/4 + \delta_*}).
\end{split}
\end{equation}
This suffices since 
$$||\dk^{J} \vk^n \sk^m A(k)^{1/2} |\eta - kt| Q_\xi||_{L^2_{\mu \leq |k| \leq 1} L^2_\eta} \lesssim \min(\E^{1/8} \mu^{-3/8}\D_\gamma^{3/8}, \mu^{-1/4} \E^{1/4} \D_\gamma^{1/4}).$$ 
Now $T_{\beta,Q_2, HL,(L, \cdot)}^y$ is easier, and we have more simply
\begin{equation}
    \begin{split}
         T_{\beta, Q_2, HL, (L, \cdot)}^y & \lesssim \int_{|k| < \mu }\iiint_{\R^3} \dk^{J} \langle \xi\rangle^{n} \sk^m  |\eta - kt| |Q_\xi| \dkm^J \langle \xi - \xi'\rangle^n \skm^m  |k-k'|^{1/2} \\
         &\quad \quad \quad \quad  |Z_{\xi-\xi'}| \langle k' \rangle^{-1/2} |k'|^{1/2}p_{\xi}^{-1/4} |Q_{\xi'}| d\eta' d\eta dk' dk\\
         &\lesssim ||\dk^{J} \vk^n \sk^m |\eta - kt| Q_\xi||_{L^1_{|k| \leq \mu} L^2_\eta} ||\dk^{J}\langle \xi \rangle^n \sk^m |k|^{1/2}  Z_\xi ||_{L^2_{\xi}}\\
         &\quad \quad \quad || \langle k \rangle^{-1/2} |k|^{1/2} p_\xi^{-1/4} Q_\xi||_{L^2_k L^1_\eta}\\
         &\lesssim \mu^{1/8} \D_\gamma^{3/8} \E^{1/8}, \mu^{1/4} \D_{\gamma} \mu^{-1/4} \D_\gamma^{1/4} \E^{1/4}\\
         &\quad \quad ( \D_\tau^{1/4} \E^{1/8} \mu^{-1/8} \D_\gamma^{1/8} + \D_\tau^{1/4 - \delta_*} \mu^{-1/4 - \delta_*} \D_\gamma^{1/4 + \delta_*})\\
         &\lesssim \mu^{-1/2 - \delta_*} \D \E^{1/2}.
    \end{split}
\end{equation}
This completes our estimates on $T_{\beta, Q}^y$, thereby finishing the proof of Lemma \ref{key_lemma}, which in turn cmpletes the proof of Theorem \ref{main_theorem}.

\section{Discussion of the Alternate Theorem}\label{alt_section}

The first difference in the proof of Theorem \ref{alt_theorem} as compared to Theorem \ref{main_theorem} comes in the definition of the symmetrized coordinates, which we define as\begin{equation}\label{def_of_symmetrized_coords_alt}
    \begin{split}
        &z \coloneqq |\partial_x|^{1/2}|\nabla|^{-1/2} \omega, \;\; \; q \coloneqq \sqrt{R} \partial_x |\nabla|^{1/2} |\partial_x|^{-1/2} \theta,\\
        &Z \coloneqq |\partial_X|^{1/2}p^{-1/4} \Omega, \; \; \; Q \coloneqq \sqrt{R} \partial_X p^{1/4} |\partial_X|^{-1/2} \Theta.
    \end{split}
\end{equation}
In terms of the symmetrized variables, \eqref{moving_reference_equations} becomes
\begin{equation}\label{Symmetrized_moving_reference_equations_alt}
    \begin{cases}
        &\partial_t Z - \frac{1}{4}\frac{\partial_t p}{p}Z +|\partial_X|^{1/2}p^{-1/4}\left(U \cdot \nabla_t \Omega\right) = - \nu p Z - \sqrt{R} |\partial_X|p^{-1/2} Q,\\
        & \partial_t Q + \frac{1}{4}\frac{\partial_t p}{p}Q+ \sqrt{R} \partial_X p^{1/4} |\partial_X|^{-1/2}\left(U \cdot \nabla_t\Theta\right) = -\kappa p Q + \sqrt{R} |\partial_X| p^{-1/2} Z,\\
        &U = \nabla_t^\perp \Psi, \quad - p \Psi = \Omega,\\
        &Z(0,X,Y) = Z_{in}(X,Y), \; \; Q(0,X,Y) = Q_{in}(X,Y).
    \end{cases}
\end{equation}
We redefine our quantitative stability norms to be
\begin{equation}\label{updated_base_norm}
\begin{split}
        ||F||_{\tilde{H}(0)} &\coloneqq \sum_{0 \leq j \leq 1}|| \langle \partial_X, \partial_Y \rangle ^n \langle \partial_X \rangle^m \langle \frac{\partial_X}{\mu}\rangle^{-j/3}\partial_Y^j F ||_{L_{X,Y}^2} + || \langle k, \partial_Y \rangle ^n  \hat{F} ||_{L_k^\infty L_Y^2}\\
    &\eqqcolon ||F||_{\tilde{H}^{(1)}(0)} + ||F||_{\tilde{H}^{(2)}(0)},
\end{split}
\end{equation}
\begin{equation}\label{update_target_norm}
\begin{split}
        ||G||_{\tilde{H}(t)} &\coloneqq \sum_{0 \leq j \leq 1}|| \langle c \lambda(\mu, \partial_x) t \rangle^J \langle \partial_X, \partial_Y \rangle^n \langle \partial_X \rangle^m \langle \frac{\partial_X}{\mu}\rangle^{-j/3} (\partial_Y - t \partial_X)^j G||_{L_{X,Y}^2} + || \langle k, \partial_Y \rangle^n  \hat{G} ||_{L_k^\infty L_y^2}\\
        &\eqqcolon ||G||_{\tilde{H}^{(1)}(t)} + ||G||_{\tilde{H}^{(2)}(t)}.
\end{split}
\end{equation}
The key theorem to prove is almost identical to Theorem \eqref{symmetric_theorem} with $\tilde{H}(0)$ and $\tilde{H}(t)$ having the updated definitions. However, there is a small but technically significant change in the statement regarding the invisid damping term:
\begin{theorem}\label{symmetric_theorem_alt}
Suppose $(Z_{in}, Q_{in})$ are initial data for \eqref{Symmetrized_moving_reference_equations_alt} and that $R > 1/4$, $\epsilon \in (0,1/2)$ are given. Then for all $m \in (0, \infty)$, $n \in [0,\infty)$, $J \in [1,\infty)$, and $\delta_* \in (0,1/12)$, there exists a constant $\delta = \delta(n, m,J,R, \epsilon, \delta_*) > 0$ independent of $\nu$ and $\kappa$ such that if
\begin{equation}
\begin{split}
        ||Z_{in}||_{\tilde{H}(0)} + ||Q_{in}||_{\tilde{H}(0)} = \zeta \leq \delta \mu^{1/2 +  \delta_*},
\end{split}
\end{equation}
    then for all $c = c(R, \epsilon) > 0$ sufficiently small (independent of $\mu$ and $\delta$) and all $\nu, \kappa \in (0,1)$ satisfying \eqref{diffusion_assumption}, the corresponding solution $(Z,Q)$ to \eqref{vorticity_perturb_eqn} satisfies the following estimate:
\begin{equation}
\begin{split}
            ||Z||_{\tilde{H}(t)} + ||Q||_{\tilde{H}(t)} &+ C \left(\int_0^t ||\frac{\partial_X }{|\nabla_t|} Z||_{\tilde{H}^{(1)}(s)}^2 + || \frac{ \partial_X }{|\nabla_t|} Q||_{\tilde{H}^{(1)}(s)}^2 ds\right)^{1/2}\\
            &+ C || \langle k, \partial_y\rangle \frac{k}{|\nabla_t|} (\hat{Z},\hat{Q}) ||_{L^\infty_k L^2([0,t])L^2_y}
            \leq 2\zeta, \quad \forall t\in [0,\infty),
\end{split}
\end{equation}
for some constant $C = C(R, \epsilon, c) > 0$ independent of $\nu$, $\kappa$, and $\delta$.
\end{theorem}
The proof of Theorem \ref{alt_theorem} follows from Theorem \ref{symmetric_theorem_alt} in the same manner as Theorem \ref{main_theorem} follows from Theorem \ref{symmetric_theorem}, as discussed in Section \ref{proof_of_main}. We therefore confine ourselves to proving Theorem \ref{symmetric_theorem_alt}.

The system \eqref{Symmetrized_moving_reference_equations_alt} differs from \eqref{Symmetrized_moving_reference_equations} only at the nonlinear level. Thus the pointwise energy $E_k$ \eqref{linear_energy}, the pointwise dissipation $D_k$ \eqref{linear_energy}, and the crucial linear estimates of Section \ref{linear_section} and Proposition \ref{linear_proposition} remain unchanged.

However, the nonlinear energy is modified to incorporate the supremum in $k$ term. We (re)define 
\begin{equation}\label{nonlinear_energy_alt}
    \begin{split}
        &\E^{(1)}[Z, Q] \coloneqq \iint_{\R^2} \frac{\dk^{2J}}{M_k(t)} \langle k, \eta \rangle^{2n} \langle k \rangle^{2m} 
 E_k[Z_k, Q_k] d\eta dk,\\
        & \E^{(2)}[Z, Q] \coloneqq  \sup_{k \in \R} \int_\R \langle k, \eta \rangle^{2n}  \left(N_k + c_\tau \mathfrak{J}_k\right)\left(|Z_k|^2 + |Q_k|^2 + \frac{1}{2\sqrt{R}}\mathrm{Re}\left(\frac{\partial_t p}{|k| p^{1/2}} Z_k \bar{Q}_k\right) \right) d\eta ,\\
        &\E[Z,Q] \coloneqq \E^{(1)}[Z,Q] + \E^{(2)}[Z,Q].
    \end{split}
\end{equation}
where $c$ and $M_k$ are as before. We also (re)define the nonlinear dissipation. Importantly, we now consider \textit{time-integrated} dissipations, as used in \cite{arbon_bedrossian}, since the conclusion of Theorem \ref{symmetric_theorem_alt} involves taking the supremum in $k$ after integrating in time:
\begin{equation}\label{nonlinear_dissipation_alt}
    \begin{split}
        &\D^{(1)}[Z, Q](T) \coloneqq \int_0^T \int_{\R^2} \frac{\dk^{2J}}{M_k(t)} \langle k, \eta \rangle^{2n} \langle k \rangle^{2m} 
 D_k[Z_k, Q_k] d\eta dk dt,\\
        & \D^{(2)}[Z, Q](T) \coloneqq  \sup_{k\in\R} \int_0^T\int_\R \langle k, \eta \rangle^{2n} \left(D_{k, \gamma}[Z_k, Q_k] + c_\tau \left(\frac{1}{2} - \frac{1}{4\sqrt{R}}\right) D_{k, \tau}[Z_k, Q_k] + \frac{1}{2}c_\rho D_{k,\rho}\right)d\eta dt,\\
        &\D[Z,Q] \coloneqq \D^{(1)}[Z, Q] + \D^{(2)}[Z, Q].
    \end{split}
\end{equation}
Furthermore, we (re)define the terms of the nonlinear dissipation coming from different components of the pointwise dissipation:
\begin{equation}\label{definition_of_gamma_and_tau_diss_alt}
    \begin{split}
        &\D_{\gamma}[Z,Q] \coloneqq \int_0^T \iint_{\R^2} \frac{\dk^{2J}}{M_k(t)} \langle k, \eta \rangle^{2n} \langle k \rangle^{2m}  
 \D_{k,\gamma}[Z_k, Q_k] dk d\eta dt + \sup_{k \in \R} \int_0^T \int_\R \langle k, \eta \rangle^{2n}  D_{k, \gamma}[Z_k, Q_k]d \eta dt,\\
 &\D_{\tau}[Z,Q] \coloneqq \int_0^T \iint_{\R^2} \frac{\dk^{2J}}{M_k(t)} \langle k, \eta \rangle^{2n} \langle k \rangle^{2m} 
 c_\tau D_{k,\tau}[Z_k, Q_k] dk d\eta dt + \sup_{k \in \R} \int_0^T \int_\R  c_\tau \left(\frac{1}{2} - \frac{1}{4\sqrt{R}}\right) \langle k, \eta \rangle^{2n}  D_{k, \tau}[Z_k, Q_k] d \eta dt,
    \end{split}
\end{equation}
and
\begin{equation}
    \D_{*}[Z,Q] \coloneqq \int_\R \frac{\dk^{2J}}{M_k(t)} \langle k, \eta \rangle^{2n} \langle k \rangle^{2m} c_{*} D_{k,*}[Z_k, Q_k] dk
\end{equation}
for $* \in \{ \alpha, \tau \alpha, \beta ,\rho, \rho\alpha\}$, letting $c_{\tau \alpha } = c_\tau c_\alpha$ and $c_{\rho \alpha} = c_\rho c_\alpha$. We note that the bootstrap Lemma \ref{bootstrap_lemma} was expressed in a purely differential form. However, since we now have supremum terms in the energy $\E$, we express the bootstrap in the following formulation, as in \cite{arbon_bedrossian}.
\begin{lemma}\label{bootstrap_lemma_alt}
    Let $(Z_{in}, Q_{in})$ be initial datum for \eqref{Symmetrized_moving_reference_equations_alt} such that $\E[Z_{in}, Q_{in}] < \infty$. Then there exists a constant $C > 0$ depending only on $m$, $J$, $R$, $\epsilon$ and the choice of $c$ such that $\forall T \in [0,\infty)$:
    \begin{equation}
        \E[Z,Q](T) \leq 2\E[Z,Q](0) - 4 c \D(T) + \mu^{-1/2 - \delta_*} 
 C^{1/2}\sup_{t \in [0,T]}\E^{1/2}[Z,Q](t) \D(T).
    \end{equation}
\end{lemma}
Clearly, Lemma \ref{bootstrap_lemma_alt} implies Theorem \ref{symmetric_theorem_alt}. In Section \ref{nonlin_boot_sect}, we will performs computations similar to those at the beginning of Section \ref{nonlinear_section}, keeping in mind the supremum terms. We will also discuss the various components of the nonlinearity which need to be estimated to prove \ref{bootstrap_lemma_alt}.

\subsection{Nonlinear Bootstrap Set-up}\label{nonlin_boot_sect}
Maintaining the definitions of $\mathbb{L}_k^{(Z)}$ and $\mathbb{L}_k^{(Q)}$ from \eqref{lin_and_nonlin}, we redefine
\begin{equation}
    \begin{split}
        &\mathbb{NL}_k^{(Z)} \coloneqq - |k|^{1/2}p^{-1/4}\left(U \cdot \nabla_t \Omega\right)_k,\\
        &\mathbb{NL}_k^{(Q)} \coloneqq - \sqrt{R} ik p^{1/4} |k|^{-1/2}\left(U \cdot \nabla_t\Theta\right)_k.
    \end{split}
\end{equation}
Now by the fundamental theorem of calculus, we have for all $T > 0$:
\begin{equation}\label{big_energ_ineq}
    \begin{split}
    \E(t)  + &\D^{2}(T) \leq \E^{(1)}(0) + \int_0^T \frac{d}{dt} \E^{(1)}(t) dt\\
    &+ \sup_k \int_0^T \int_\R \langle k, \eta\rangle^{2n}(D_{k,\gamma} + c_\tau \left(\frac{1}{2} - \frac{1}{4\sqrt{R}}\right) D_{k,\tau} + \frac{1}{2}c_\rho D_{k,\rho}) d\eta dt\\
    &
 + \sup_k \int_\R \langle k, \eta\rangle^{2n} \biggl((N_k + c_\tau \mathfrak{J}_k)\biggl(|Z_k|^2 + |Q_k|^2 + \frac{1}{2\sqrt{R}}\mathrm{Re}\left(\frac{\partial_t p}{ |k| p^{1/2}} Z_k\bar{Q}_k \right) \biggr)(0) \\
 &\quad \quad \quad +  \int_0^T \frac{d}{dt}\biggl((N_k + c_\tau \mathfrak{J}_k)\biggl(|Z_k|^2 + |Q_k|^2 + \frac{1}{2\sqrt{R}}\mathrm{Re}\left(\frac{\partial_t p}{ |k| p^{1/2}} Z_k\bar{Q}_k \right) \biggr) \biggr)(t) dt \biggr) d\eta.
    \end{split}
\end{equation}
Relabeling the term $\mathcal{NL}$ from Section \ref{nonlinear_section} as $\mathcal{NL}^{(1)}$, We have as in \eqref{nonlin_deriv},
\begin{equation}\label{easy_non_lin}
    \frac{d}{dt}\mathcal{E}^{(1)} \leq - 4c \D^{(1)} + \mathcal{NL}^{(1)}.
\end{equation}
To handle the remaining terms in \eqref{big_energ_ineq}, we begin by noting that by $R > 1/4$, the definition of $\mu$ \eqref{mu_definition}, Young's product inequality, and Cauchy-Schwarz, we have for each $k \neq 0$ (similar to the work in Section \ref{linear_section}):
\begin{equation}\label{key_non_lin_lower_bound}
    \begin{split}
        (D_{k,\gamma} + c_\tau \left(\frac{1}{2} - \frac{1}{4\sqrt{R}}\right)D_{k,\tau} + \frac{1}{2}c_\rho D_{k,\rho}) &\leq 2 D_{k, \gamma} + \left(\frac{1}{2} - \frac{1}{4 \sqrt{R}}\right) c_\tau D_{k, \tau} + \frac{1}{2\sqrt{R}} D_{k,\rho}\\
        &\leq 2 p N_k\biggl(\nu |Z_k|^2 + \kappa |Q_k|^2 + \frac{\nu + \kappa}{4 \sqrt{R}} \frac{\partial_t p}{|k| p^{1/2}} \mathrm{Re}(Z_k \bar{Q}_k) \biggr)\\
        & \quad + c_\tau\frac{1}{2} k^2 p^{-1} \biggl(|Z_k|^2 + |Q_k|^2 + \frac{1}{2 \sqrt{R}}\frac{\partial_t p}{|k| p^{1/2}} \mathrm{Re}(Z_k \bar{Q}_k) \biggr)\\
        &\quad +N_k \frac{2 |k|^3}{p^{3/2}} \biggl(\frac{1}{2\sqrt{R}-1}(|Z_k|^2 + |Q_k|^2 + \frac{1}{2\sqrt{R}}\biggl(1 - \frac{1}{2\sqrt{R}-1} \frac{\partial_t p}{|k| p^{1/2}} \mathrm{Re}(Z_k \bar{Q}_k)\biggr) \biggr)\\
        & \eqqcolon \tilde{D}_k.
    \end{split}
\end{equation}
Applying \eqref{key_non_lin_lower_bound}, we find by the triangle inequality:
\begin{equation}\label{non_lin_comps}
    \begin{split}&\sup_k \int_0^T \int_\R \langle k, \eta\rangle^{2n}(D_{k,\gamma} + \left(\frac{1}{2} - \frac{1}{4\sqrt{R}}\right)c_\tau D_{k,\tau} + \frac{1}{2}c_\rho D_{k,\rho}) d\eta dt\\
    & \quad +
         \sup_k \int_\R \langle k, \eta\rangle^{2n} \biggl((N_k + c_\tau \mathfrak{J}_k)\biggl(|Z_k|^2 + |Q_k|^2 + \frac{1}{2\sqrt{R}}\mathrm{Re}\left(\frac{\partial_t p}{ |k| p^{1/2}} Z_k\bar{Q}_k \right) \biggr)(0)\\
 &\quad \quad \quad +  \int_0^T \frac{d}{dt}\biggl((N_k + c_\tau \mathfrak{J}_k)\biggl(|Z_k|^2 + |Q_k|^2 + \frac{1}{2\sqrt{R}}\mathrm{Re}\left(\frac{\partial_t p}{ |k| p^{1/2}} Z_k\bar{Q}_k \right) \biggr) \biggr)(t) dt \biggr) d\eta\\
 &\leq  \sup_k \int_0^T \int_\R \langle k, \eta \rangle^{2n} \tilde{D}_k d\eta dt  +
         \sup_k \int_\R \langle k, \eta\rangle^{2n} \biggl((N_k + c_\tau \mathfrak{J}_k)\biggl(|Z_k|^2 + |Q_k|^2 + \frac{1}{2\sqrt{R}}\mathrm{Re}\left(\frac{\partial_t p}{ |k| p^{1/2}} Z_k\bar{Q}_k \right) \biggr)(0)\\
 &\quad \quad \quad +  \int_0^T \frac{d}{dt}\biggl((N_k + c_\tau \mathfrak{J}_k)\biggl(|Z_k|^2 + |Q_k|^2 + \frac{1}{2\sqrt{R}}\mathrm{Re}\left(\frac{\partial_t p}{ |k| p^{1/2}} Z_k\bar{Q}_k \right) \biggr) \biggr)(t) dt \biggr) d\eta\\
 & \leq 2 \sup_k \int_\R \langle k, \eta\rangle^{2n} \biggl( (N_k + c_\tau \mathfrak{J}_k)\biggl(|Z_k|^2 + |Q_k|^2 + \frac{1}{2\sqrt{R}}\mathrm{Re}\left(\frac{\partial_t p}{ |k| p^{1/2}} Z_k\bar{Q}_k \right) \biggr)(0)\\
 &\quad \quad \quad + \int_0^T \tilde{D}_k + \frac{d}{dt}\biggl((N_k + c_\tau \mathfrak{J}_k)\biggl(|Z_k|^2 + |Q_k|^2 + \frac{1}{2\sqrt{R}}\mathrm{Re}\left(\frac{\partial_t p}{ |k| p^{1/2}} Z_k\bar{Q}_k \right) \biggr) \biggr)(t)dt \biggr) d \eta .
    \end{split}
\end{equation}
Continuing from \eqref{non_lin_comps}, we usw the triangle inequality again and the exact form of the time derivative, as computed in Section \ref{linear_section}:
\begin{equation}\label{final_non_lin_sups}
    \begin{split}
        &2 \sup_k \int_\R \langle k, \eta\rangle^{2n} \biggl( (N_k + c_\tau \mathfrak{J}_k)\biggl(|Z_k|^2 + |Q_k|^2 + \frac{1}{2\sqrt{R}}\mathrm{Re}\left(\frac{\partial_t p}{ |k| p^{1/2}} Z_k\bar{Q}_k \right) \biggr)(0)\\
 &\quad \quad \quad + \int_0^T \tilde{D}_k + \frac{d}{dt}\biggl((N_k + c_\tau \mathfrak{J}_k)\biggl(|Z_k|^2 + |Q_k|^2 + \frac{1}{2\sqrt{R}}\mathrm{Re}\left(\frac{\partial_t p}{ |k| p^{1/2}} Z_k\bar{Q}_k \right) \biggr) \biggr)(t)dt \biggr) d \eta\\
 &\leq 2 \E^{(2)}(0)\\
 &\quad \quad \quad + 2\sup_k \int_\R \langle k, \eta \rangle^{2n}\biggl| \int_0^T \tilde{D}_k + \frac{d}{dt}\biggl((N_k + c_\tau \mathfrak{J}_k)\biggl(|Z_k|^2 + |Q_k|^2 + \frac{1}{2\sqrt{R}}\mathrm{Re}\left(\frac{\partial_t p}{ |k| p^{1/2}} Z_k\bar{Q}_k \right) \biggr) \biggr)(t)dt \biggr| d\eta\\
 &\leq 2\E^{(2)}(0) + 2\sup_k \int_\R \langle k, \eta \rangle ^{2n} \int_0^T c_\tau 16 \pi \frac{R}{\epsilon}  \mu p \left(|Z_k|^2 + |Q_k|^2\right) + c_\tau \frac{\pi}{8\sqrt{R}}|k|^{3/2}p^{-3/4}\left(|Z_k|^2 + |Q_k|^2\right)\\
 & \quad \quad  |N_k + c_\tau\mathfrak{J}_k|\biggl|2\mathrm{Re}(\mathbb{NL}_k^{(Z)} \bar{Z}_k) + 2\mathrm{Re}(\mathbb{NL}_k^{(Q)} \bar{Q}_k) + \frac{1}{2\sqrt{R}}\frac{\partial_t p}{|k|p^{1/2}}\biggl( \mathrm{Re}(\mathbb{NL}_k^{(Z)} \bar{Q}_k) + \mathrm{Re}(\bar{Z}_k \mathbb{NL}_k^{(Q)})\biggr) \biggr|dt d\eta.
    \end{split}
\end{equation}
We now define
\begin{equation}\label{T_inf_def}
    T_\infty \coloneqq 2\sup_k \int_0^T \int_\R \langle k, \eta \rangle^{2n} |N_k + c_\tau\mathfrak{J}_k|\biggl|2\mathrm{Re}(\mathbb{NL}_k^{(Z)} \bar{Z}_k) + 2\mathrm{Re}(\mathbb{NL}_k^{(Q)} \bar{Q}_k) + \frac{1}{2\sqrt{R}}\frac{\partial_t p}{|k|p^{1/2}}\biggl( \mathrm{Re}(\mathbb{NL}_k^{(Z)} \bar{Q}_k) + \mathrm{Re}(\bar{Z}_k \mathbb{NL}_k^{(Q)})\biggr) \biggr|d \eta dt.
\end{equation}
Recalling that $c_\tau \leq \frac{1}{32 \pi} \min\{\frac{1}{8}, \frac{\epsilon}{R}\}$, we have from \eqref{easy_non_lin} and \eqref{final_non_lin_sups} applied to \eqref{big_energ_ineq}
\begin{equation}
\begin{split}
        \E(0) + \D^{(2)}(T) &\leq 2 \E(0) - 4 c \D^{(1)}(T) + \int_0^T \mathcal{NL}^{(1)} (t) dt + T_\infty(T)\\
    & \quad + \sup_k \int_0^T \int_\R \langle k, \eta\rangle^{2n} (\frac{1}{2}D_{k,\gamma} + \frac{1}{8} c_\rho D_{k,\rho} )d\eta dt,
\end{split}
\end{equation}
which implies that
\begin{equation}
    \begin{split}
        \E(0) + \frac{1}{2} \D^{(2)}(T) + 4c \D^{(1)}(T)\leq \int_0^T \mathcal{NL}^{(1)} (t) dt + T_\infty(T).
    \end{split}
\end{equation}
Hence the key to proving Lemma \eqref{bootstrap_lemma_alt} is to prove a bound of the form
\begin{equation}\label{pre_holder}
    \int_0^T \mathcal{NL}^{(1)}(t) dt + T_\infty(T) \lesssim \mu^{-1/2-\delta_*} \sup_{t \in [0,T]}\E^{1/2}(t) \D(T).
\end{equation}
Recall that in \eqref{T_main}, we defined the quantities $T_\gamma$, $T_\alpha$, and $T_\beta$. We now re-define these terms to be their time-integrated versions (suitably modified with the new $\mathbb{NL}_k^{(Z)}$ and $\mathbb{NL}_k^{(Q)}$) so that
\begin{equation}
     \int_0^T \mathcal{NL}^{(1)}(t) dt = T_\gamma(T) + T_\alpha(T) + T_\beta(T).
\end{equation}
Associated to each $T_{*}$, $* \in \{\gamma, \alpha, \beta\}$, we recall that $T_{*} = T_{*, Z} + T_{*,Q} + T_{*, m_1} + T_{*,m_2}$. We extend this decomposition to $T_\infty$ in the natural way and define
\begin{equation}\label{T_infty}
    \begin{split}
        T_\infty &\leq 4\sup_{k \in \R} \int_0^T \int_\R \langle k, \eta\rangle^{2n}|N_k + c_\tau\mathfrak{J}_k|\biggl|\mathrm{Re}(\mathbb{NL}_k^{(Z)} \bar{Z}_k)  \biggr| d\eta dt\\
        &\quad + 4\sup_{k \in \R} \int_0^T \int_\R \langle k, \eta\rangle^{2n}|N_k + c_\tau\mathfrak{J}_k|\biggl|\mathrm{Re}(\mathbb{NL}_k^{(Q)} \bar{Q}_k) \biggr| d\eta dt\\
        &\quad +  4\sup_{k \in \R} \int_0^T \int_\R \langle k, \eta\rangle^{2n}|N_k + c_\tau\mathfrak{J}_k| \frac{1}{2\sqrt{R}}|\frac{\partial_t p}{|k|p^{1/2}}| \biggl|\mathrm{Re}(\mathbb{NL}_k^{(Z)} \bar{Q}_k) \biggr| d\eta dt\\
        &\quad + 4\sup_{k \in \R} \int_0^T \int_\R \langle k, \eta\rangle^{2n}|N_k + c_\tau\mathfrak{J}_k| \frac{1}{2\sqrt{R}}|\frac{\partial_t p}{|k|p^{1/2}}| \biggl|\mathrm{Re}(\bar{Z}_k \mathbb{NL}_k^{(Q)})\biggr| d\eta dt \eta\\
        &\eqqcolon T_{\infty, Z} + T_{\infty, Q} + T_{\infty, m_1} + T_{\infty, m_2}.
    \end{split}
\end{equation}
The lemma corresponding to Lemma \ref{key_lemma} which suffices to prove Lemma \ref{bootstrap_lemma_alt} is then
\begin{lemma}\label{key_lemma_alt}
    For $T_\gamma$, $T_\alpha$, $T_\beta$, and $T_\infty$, we have for all $T > 0$:
\begin{equation}\label{key_eqn_alt}
        |T_\gamma| + |T_\alpha| + |T_\beta| + |T_\infty| \lesssim \mu^{-1/2 - \delta_*}  \sup_{t \in [0,T]}\E^{1/2}(t) \D(T),
    \end{equation}
    where the implicit constant is independent of $\nu, \kappa$, and $\mu$, but is allowed to depend on $\delta_*$, $\epsilon$, $m$, $J$, and $R$.
\end{lemma}
The bounds on $T_\gamma$, $T_\alpha$, and $T_\beta$ for Lemma \ref{key_lemma_alt} are proven in a similar manner to Sections \ref{gamma_terms_for_Z} -- \ref{beta_terms_for_Q}. The first difference arises in the need to consider the time integral. When all frequencies are $\geq 1$ in magnitude, then simply applying Cauchy-Schwarz in time at the end suffices. Larger differences arise in the consideration of low frequencies ($|k| < \mu$) and intermediate frequencies ($\mu \leq |k| < 1)$. The intermediate regime tends to be beneficial, and in fact removes the logarithmic loss present in, for instance, \eqref{alpha_Z1_HL}. Furthermore, the terms like $\langle k \rangle^{1/2} p_\xi^{-1/4}$ previously posed a challenge at low frequencies, since we needed to find an additional half-derivative to cancel out $p_\xi^{-1/4}$. Now however, these are replaced with $|k|^{1/2} p_\xi^{-1/4}$, which naturally gives rise to factors like $\D_{\tau}^{1/4}$, which is one of the most beneficial dissipation factors.

Other differences arise in the consideration of low frequencies. Here, we encounter terms containing $|k|^{-1/2}$ and measured in $L^1$. Rather than interpolate these frequencies to $L^2_k$, we interpolate to $L^\infty_k$. Furthermore, greater care must be taken to ensure that the time integral and the supremum-in-$k$ estimates occur in the correct order.

To demonstrate these changes, we will include the bound on the $T_{\gamma, Z}^x$ terms, as this will show the key ideas needed to bound $T_\gamma$, $T_\alpha$, and $T_\beta$ in the new setting. Furthermore, we will show that $T_\infty$ obeys \eqref{key_eqn_alt}, as this is a new term which will require some subtly in handling the time-integration. 

\subsection{Bound on Gamma Terms for Z: x derivatives}

With the new notation, $T_{\gamma, Z}^x$ is given as
\begin{equation}
\begin{split}
        T_{\gamma, Z}^x &= - \int_0^T \iint_{\R^2} \frac{\dk^{2J}}{M_k(t)} \langle \xi \rangle^{2n}\langle k \rangle^{2m}  (N_k + c_\tau\mathfrak{J}_k)\mathrm{Re}\biggl(\iint_{\R^2} |k|^{1/2} p_\xi^{-1/4} \bar{Z}_\xi\\
    &\quad \quad \quad \quad \quad \quad | k-k' |^{-1/2} p_{\xi -\xi'}^{-3/4} i (k-k') Z_{\xi-\xi'} |k'|^{-1/2} p_{\xi'}^{1/4} i(\eta'-k't) Z_{\xi'}d\eta' dk'\biggr) d\eta dk dt,
\end{split}
\end{equation}
which we bound by the following, using the the triangle inequality:
\begin{equation}
    \begin{split}
        |T_{\gamma, Z}^x| &\lesssim \int_0^T\iiiint_{\R^4}  \dk^{2J} \langle \xi \rangle^{n} \langle k \rangle^{2m} |k|^{1/2} |Z_\xi| \vkm^n |k-k'|^{1/2} p_{\xi-\xi'}^{-3/4} |Z_{\xi -\xi'}| |k'|^{-1/2}\\
        & \quad \quad \quad \quad \quad \quad \langle \xi' \rangle^{n}|\eta' -k't| Z_{\xi'} d\eta' d\eta dk' dk dt\\
        &\quad + \int_0^T\iiiint_{\R^4} \dk^{2J}\langle \xi \rangle^{n}  \sk^{2m}|k|^{1/2} p_\xi^{-1/4} |Z_\xi| \langle \xi - \xi'\rangle^n |k-k'|^{1/2} p_{\xi-\xi'}^{-1/2} |Z_{\xi -\xi'}| \\
        & \quad \quad \quad \quad \quad \quad \quad \vkp^n |k'|^{-1/2} |\eta' -k't| Z_{\xi'} d\eta' d\eta dk' dk dt\\
        &\eqqcolon T_{\gamma, Z_1}^x + T_{\gamma, Z_2}^x.
    \end{split}
\end{equation}
Splitting as in Section \ref{gamma_terms_for_Z}, we decompose into $LH$ and $HL$ terms as
\begin{equation}
    T_{\gamma, Z_1}^x \eqqcolon T_{\gamma, Z_1, LH}^x + T_{\gamma, Z_1, HL}^x, \; \; \; T_{\gamma, Z_2}^x \eqqcolon T_{\gamma, Z_2, LH}^x + T_{\gamma, Z_2, HL}^x.
\end{equation}
We will start with the $T_{\gamma, Z_1}^x$ terms. In the $LH$ case, we estimate using Young's inequality, $|k| \approx |k'|$, and H\"older's inequality in time:
\begin{equation}\label{T_gamma_mod_Z1_LH}
    \begin{split}
        T_{\gamma, Z_1, LH}^x &\lesssim  \int_0^T \iiiint_{\R^4} 1_{LH} \dk^{J}\langle \xi \rangle^{n} \sk^m |Z_\xi|  \vkm^n  |k-k'|^{1/2} p_{\xi-\xi'}^{-3/4} |Z_{\xi-\xi'}|\\
        &\quad \quad \quad \quad \dkp^J \vkp^n \skp^m |(\eta' -k't) Z_{\xi'}|  d \eta' d\eta dk' dk dt\\
        &\lesssim \int_0^T ||\dk^{J}\langle \xi \rangle^{n} \sk^m  Z_\xi||_{L^2_\xi} || \vk^n |k|^{1/2} p_\xi^{-3/4} Z_{\xi}||_{L^1_\xi}\\
        &\quad \quad ||\dk^{J}\langle \xi \rangle^{n} \sk^m (\eta -kt) Z_\xi||_{L^2_\xi} dt\\
        &\lesssim \sup_{t \in [0,T]}\E^{1/2}(t)\left(\int_0^T || \vk^n |k|^{1/2} p_\xi^{-3/4} Z_{\xi}||_{L^1_\xi}^2 dt\right)^{1/2}\\
        &\quad \quad \left(\int_0^T||\dk^{J}\langle \xi \rangle^{n} \sk^m (\eta -kt) Z_\xi||_{L^2_\xi}^2 dt\right)^{1/2}.
    \end{split}
\end{equation}
The final factor of \eqref{T_gamma_mod_Z1_LH} is controlled by $\mu^{-1/2} \D_\gamma(T)$. To handle the $L^1_\xi$ factor, we need a variation on the themes of Lemmas \ref{new_energy_lemmas} and \ref{main_D_tau_lemma}. Indeed, by Minkowski's inequality, interpolating in $k$, switching to stationary coordinates, and interpolating in $\eta$, we find:
\begin{equation}\label{replacement_lemma_1}
    \begin{split}
    || \vk^n |k|^{1/2} p_\xi^{-3/4} Z_\xi||_{L^2_t([0,T])L^1_\xi} &\lesssim || \vk^n |k|^{1/2} p_\xi^{-3/4} Z_\xi||_{L^1_\xi L^2_t([0,T])}\\
    &\lesssim || \vk^n \langle k \rangle^{m-m} |k|^{1/2} (|k|^{3/4 - 3/4} 1_{|k| < 1} + |k|^{1/2-1/2}1_{|k| \geq 1}) p_\xi^{3/4} Z_\xi||_{L^1_\xi L^2_t([0,T])}\\
    &\lesssim || \vk^n |k| \max(|k|^{1/4},1) p_\xi^{-3/4}Z_\xi||_{L^\infty_k L^1_\eta L^2_t([0,T])}\\
    &\quad + || \vk^n \sk^m |k| \max(|k|^{1/4},1) p_\xi^{-3/4}Z_\xi||_{L^2_k L^1_\eta L^2_t([0,T])}\\
    &\lesssim || \vk^n |k|^{1/2}\max(|k|^{1/4},1) \max(|\eta-kt|^{1/4}, | \eta-kt |^{1/2 + 2 \delta_*}) p_\xi^{-3/4}Z_\xi||_{L^\infty_k L^2_\eta L^2_t([0,T])}\\
    &\quad + || \vk^n \sk^m |k|^{1/2} \max(|k|^{1/4},1)\max(|\eta-kt|^{1/4}, | \eta-kt |^{1/2 + 2 \delta_*}) p_\xi^{-3/4}Z_\xi||_{L^2_\xi L^2_t([0,T])}\\
    &\lesssim \D_\tau^{1/2} + \mu^{-\delta_*} \D_\gamma^{-\delta_*}\D_{\tau}^{1/2 - \delta_*}.
    \end{split}
\end{equation}
Applying \eqref{replacement_lemma_1} to \eqref{T_gamma_mod_Z1_LH}, we obtain
\begin{equation}
    T_{\gamma, Z_1, LH}^x \lesssim \mu^{-1/2 - \delta_*}\sup_{t \in [0,T]}\E^{1/2}(t) \D.
\end{equation}
Next for the $HL$ term, we utilize $|k|,|k'| \lesssim |k-k|'$, Young's inequality, and H\"older's inequality to compute
\begin{equation}
    \begin{split}
        T_{\gamma, Z_1, HL}^x &\lesssim \int_0^T \iiiint_{\R^4} 1_{HL} \dk^{J} \langle \xi \rangle^{n} \langle k \rangle^{m} |Z_\xi| \dkm^J \vkm^n \skm^m |k-k'| \min(|k-k'|^{1/4},1) p_{\xi-\xi'}^{-3/4}\\
        & \quad \quad \quad \quad \quad  \quad  |Z_{\xi-\xi'}| \vkp^n |k'|^{-1/2} (1 + 1_{|k-k'| \leq 1}|k'|^{-1/4})|(\eta'-k't)Z_{\xi'}| d\eta' d\eta dk' dk dt\\
        &\lesssim \sup_{t \in [0,T]}|| \dk^J \vk^n \sk^{m} Z_\xi||_{L_\xi^2}|| \dk^J \sk^m \min(|k|,|k|^{5/4}) p^{-3/4} Z_\xi ||_{L^2_t([0,T])L_k^2 L_\eta^1}\\& \quad \quad || \vk^n \sk^{-1/2}\max(|k|^{-1/4},1) |\eta - kt| Z_\xi ||_{L^2_t([0,T])L^1_k L^2_\eta}.
    \end{split}
\end{equation}
Now interpolation in $\eta$ in the style of Lemma \ref{main_D_tau_lemma} gives
$$|| \dk^J \sk^m \min(|k|,|k|^{5/4}) p^{-3/4} Z_\xi ||_{L^2_t([0,T])L_k^2 L_\eta^1} \lesssim \D_\tau^{1/2} + \mu^{-\delta_*}\D_\gamma^{\delta_*} \D_\tau^{1/2 - \delta_*},$$
while Minkowski's inequality and interpolation in $k$ yield
\begin{equation}
    \begin{split}
    || \vk^n |k|^{-1/2}\max(|k|^{-1/4},1) |\eta - kt| Z_\xi ||_{L^2_t([0,T])L^1_k L^2_\eta} & \lesssim || \vk^n |k|^{-1/2}\max(|k|^{-1/4},1) |\eta - kt| Z_\xi ||_{L^1_k L^2_t([0,T]) L^2_\eta}\\
    &\lesssim || \vk^n  |\eta - kt| Z_\xi ||_{L^\infty_k L^2_t([0,T]) L^2_\eta}\\
    &\quad \quad  + || \vk^n \sk^m |\eta - kt| Z_\xi ||_{L^2_k L^2_t([0,T]) L^2_\eta}\\
    &\lesssim \mu^{-1/2}\D_\gamma^{1/2},
    \end{split}
\end{equation}
so that altogether
\begin{equation}
    T_{\gamma, Z_1, HL}^x \lesssim \mu^{-1/2 - \delta_*}\sup_{t \in [0,T]}\E^{1/2}(t) \D.
\end{equation}
Having completed the necessary bounds on $T_{\gamma, Z_1}^x$, we turn our attention to $T_{\gamma, Z_2}^x$, starting with the $LH$ case. Here, we start by employing Young's inequality and H\"older's inequality:
\begin{equation}
    \begin{split}
        T_{\gamma, Z_2, LH}^x &\lesssim \int_0^T \iiiint_{\R^4} 1_{LH}\dk^{J} \langle \xi\rangle^{n} \sk^m |k|^{1/2} p_\xi^{-1/4} |Z_\xi| \vkm^{n} p_{\xi -\xi'}^{-1/2} |Z_{\xi-\xi'}|\\
        &\quad \quad \quad \quad \dkp^J \vkp^n \skp^m |(\eta' - k't) Z_{\xi'}| d\eta' d\eta dk' dk dt\\
        & \lesssim || \dk^J \langle \xi \rangle^n \langle k \rangle^m \min(|k|^{1/4}, 1) p_\xi^{-1/4} Z_\xi ||_{L^4_t([0,T])L^2_\xi} \\
        & \quad \quad || \vk^n   p_\xi^{-1/2} Z_\xi||_{L^4_t([0,T]) L^1_\xi}|| \dk^J \vk^n \sk^m (\eta-kt) Z_\xi ||_{L^2_t([0,T])L^2_\xi}\\
        &\lesssim \sup_{t \in [0,T]}\E^{1/4}(t) \D_\tau^{1/4} || \vk^n   p_\xi^{-1/2} Z_\xi||_{L^4_t([0,T]) L^1_\xi}\mu^{-1/2} \D_\gamma^{1/2}
    \end{split}
\end{equation}
Next, we use the following interpolation argument (incorporating Minkowski's inequality) which replaces Lemma \ref{main_D_tau_lemma}:
\begin{equation}\label{replacement_lemma}
    \begin{split}
    || \vk^n p_\xi^{-1/2} Z_\xi||_{L^4_t([0,T])L^1_\xi} &\lesssim || \vk^n 1_{|k| \geq 1} p_\xi^{-1/2} Z_\xi||_{L^4_t([0,T])L^1_\xi} + || \vk^n 1_{|k| < 1} p_\xi^{-1/2} Z_\xi||_{L^1_k L^4_t([0,T]) L^1_\eta}\\
    &\lesssim || \vk^n \sk^m |k|^{1/2} \min(|k|^{1/4},1) p_\xi^{-1/2} Z_\xi||_{L^4_t([0,T])L^2_k L^1_\eta} \\
    &\quad \quad || \vk^n |k|^{3/4} p_\xi^{-1/2} Z_\xi||_{L^\infty_k L^4_t([0,T]) L^1_\eta}\\
    &\lesssim || \vk^n \sk^m |k|^{1/2} \min(|k|^{1/4},1) \max(|\eta-kt|^{1/4}, |\eta -kt|^{1/2 +2\delta_*})p_\xi^{-1/2} Z_\xi||_{L^4_t([0,T])L^2_k L^2_\eta} \\
    &\quad \quad || \vk^n |k|^{3/4} p_\xi^{-1/2} \max(|\eta-kt|^{1/4}, |\eta -kt|^{1/2 +2\delta_*})Z_\xi||_{L^\infty_k L^4_t([0,T]) L^2_\eta}\\
    &\lesssim \sup_{t \in [0,T]}\E^{1/4}(t)\D_\tau^{1/4} + \mu^{-\delta_*} \D_\gamma^{\delta_*}  \D_\tau^{1/4 - \delta_*} \sup_{t \in [0,T]}\E^{1/4}(t).
    \end{split}
\end{equation}
We therefore find
\begin{equation}
    T_{\gamma, Z_2, LH}^x \lesssim \sup_{t \in [0,T]}\E^{1/4}(t) \D_\tau^{1/4}\sup_{t \in [0,T]}\E^{1/4}(t) ( \D_\tau^{1/4} + \mu^{-\delta_*} \D_\gamma^{\delta_*}  \D_\tau^{1/4 - \delta_*}) \mu^{-1/2} \D_\gamma^{1/2}.
\end{equation}
Our final estimate concerns $T_{\gamma, Z_2, HL}^x$. Unlike in the proof of Theorem \ref{main_theorem}, there is no need to split along $|k| \leq |k'|$ and $|k| > |k'|$. Instead we have more simply, through Young's inequalty, H\"older's inequality, and interpolation in $\eta$ similar to that done in \eqref{replacement_lemma}:
\begin{equation}
    \begin{split}
    T_{\gamma, Z_2, HL}^x &\lesssim \int_0^T \iiiint_{\R^4} 1_{HL}\dk^{J} \langle \xi\rangle^{n} \sk^m |k|^{1/2} p_\xi^{-1/4} |Z_\xi| \dkm^J \skm^m \vkm^{n}\\
    &\quad \quad \quad \quad |k-k'|^{1/2} p_{\xi -\xi'}^{-1/2} |Z_{\xi-\xi'}|
        |k'|^{-1/2} \vkp^n  |(\eta' - k't) Z_{\xi'}| d\eta' d\eta dk' dk dt\\
        &\lesssim || \dk^J \langle \xi \rangle^n \langle k \rangle^m \min(|k|^{1/4}, 1) p_\xi^{-1/4} Z_\xi ||_{L^4_t([0,T])L^2_\xi}\\
        &\quad \quad || \dk^J \langle \xi \rangle^n \langle k \rangle^m |k|^{1/2} \min(|k|^{1/4},1) p_\xi^{-1/2} Z_\xi ||_{L^4_t([0,T])L^2_k L^1_\eta}\\
        &\quad \quad ||\vk^n |k|^{-1/2} \max(|k|^{-1/4},1) (\eta-kt) Z_\xi||_{L^2_t([0,T])L^1_k L^2_\eta}\\
        &\lesssim  \sup_{t \in [0,T]}\E^{1/4}(t) \D_\tau^{1/4}( \sup_{t \in [0,T]}\E^{1/4}(t)\D_\tau^{1/4} + \mu^{-\delta_*} \D_\gamma^{\delta_*}  \D_\tau^{1/2 - \delta_*} \sup_{t \in [0,T]}\E^{1/4}(t))\\
        &\quad \quad ||\vk^n |k|^{-1/2} \max(|k|^{-1/4},1) (\eta-kt) Z_\xi||_{L^2_t([0,T])L^1_k L^2_\eta}.
    \end{split}
\end{equation}
The estimate is complete upon noting that
\begin{equation}
    \begin{split}
    ||\vk^n |k|^{-1/2} \max(|k|^{-1/4},1) (\eta-kt) Z_\xi||_{L^2_t([0,T])L^1_k L^2_\eta} & \lesssim ||\vk^n |k|^{-1/2} \max(|k|^{-1/4},1) (\eta-kt) Z_\xi||_{L^1_k L^2_t([0,T])L^2_\eta}\\
    &\lesssim ||\vk^n \sk^m  (\eta-kt) Z_\xi||_{L^2_k L^2_t([0,T])L^2_\eta}\\
    &\quad \quad + ||\vk^n   (\eta-kt) Z_\xi||_{L^\infty_k L^2_t([0,T])L^2_\eta}\\
    &\lesssim \mu^{-1/2} \D_\gamma^{1/2}.
    \end{split}
\end{equation}
This completes the work for $T_\gamma^x$.

\subsection{Bound on \texorpdfstring{$T_\infty$}{Supremum terms} for Z}
The proofs in this section will mirror in many ways the proofs in Section \ref{gamma_terms_for_Z}. Aside from the modified change of coordinates, the main differences lies in our application of Young's inequality and in the usage of H\"older's inequality in time. In Section \ref{gamma_terms_for_Z}, we were free to place two factors among $k$, $k'$, or $k-k'$ in $L^2_k$, with the remaining factor being placed in $L^1_k$. This $L^1_k$ factor was generally the relatively ``low" frequency between $k'$ and $k-k'$. However, here we need to place the $k$ factor in $L^\infty_k$, with the remaining factors going in $L^p_k$ and $L^{p'}_k$, where $p$ and $p'
$ are H\"older conjugates. Generally, they will be either $2$ and $2$ or $1$ and $\infty$. This matters since interpolation from $L^\infty_k$ is more difficult, as there are fewer valid terms of $\D$ and $\E$ which are $L^\infty_k$ based. Furthermore, we will need to take care that the relative ordering of the time-norm and the frequency norms are correct in all estimates. This is further complicated by the heavy use of interpolation. We begin by considering $T_{\infty, Z}$:
\begin{equation}
    \begin{split}
        T_{\infty,Z} &\leq  4\sup_{k \in\R} \int_0^T \int_{\R}  \langle \xi \rangle^{2n}  |N_k + c_\tau\mathfrak{J}_k|\mathrm{Re}\biggl|\iint_{\R^2} |k| p_\xi^{-1/4} \bar{Z}_\xi\\
    &\quad \quad \quad \quad \quad \quad |k-k'|^{-1/2} p_{\xi -\xi'}^{-3/4} i (k-k') Z_{\xi - \xi'} |k'|^{-1/2} p_{\xi'}^{1/4} i(\eta'-k't) Z_{\xi'}d\eta' dk'\biggr| d\eta dt\\
    &\quad +4\sup_{k \in \R}\int_0^T \int_{\R}  \langle \xi \rangle^{2n}  |N_k + c_\tau\mathfrak{J}_k|\mathrm{Re}\biggl|\iint_{\R^2} |k| p_\xi^{1/4} \bar{Z}_\xi\\
    &\quad \quad \quad \quad \quad \quad |k-k'|^{-1/2} p_{\xi -\xi'}^{-3/4} i ((\eta-\eta') - (k-k')t) Z_{\xi - \xi'} |k'|^{-1/2} p_{\xi'}^{1/4} i k' Z_{k'}d\eta' dk'\biggr| d\eta dt\\
    &\eqqcolon T_{\infty, Z}^x + T_{\infty, Z}^y.
    \end{split}
\end{equation}

\subsubsection{x-derivatives}
We start with the following decomposition:
\begin{equation}
    \begin{split}
        |T_{\infty, Z}^x| &\lesssim \sup_{\R} \int_0^T \iiint_{\R^3} \langle \xi \rangle^{n} |k|^{1/2} |Z_\xi| \vkm^n |k-k'|^{1/2} p_{\xi - \xi'}^{-3/4} |Z_{\xi-\xi'}| \vkp^n |k'|^{-1/2} |(\eta' - k't) Z_{\xi'}|  d \eta' d \eta dk' dt \\
        &\quad +\sup_{\R} \int_0^T \iiint_{\R^3} \langle \xi\rangle^{n} |k|^{1/2} p_\xi^{-1/4} |Z_\xi| \vkm^n |k-k'|^{1/2} p_{\xi-\xi'}^{-1/2} |Z_{\xi-\xi'}| \vkp^n |k'|^{-1/2}|(\eta'-k't)Z_{\xi'}|  d \eta' d\eta dk' dt\\
        &\coloneqq T_{\infty,Z_1}^x+ T_{\infty,Z_2}^x.
    \end{split}
\end{equation}
We now perform the further split,
\begin{equation}
    \begin{split}
        T_{\infty,Z_2}^x &\lesssim \sup_{\R} \int_0^T\iiint_{\R^3} 1_{LH}(1_{|k-k'| < 1} + 1_{|k-k'| \geq 1})\langle \xi\rangle^{n} |k|^{1/2} p_\xi^{-1/4} |Z_\xi| |k-k'|^{1/2} p_{\xi-\xi'}^{-1/2} \\
        &\quad \quad \quad \quad \quad |Z_{\xi-\xi'}| \vkp^n |k'|^{-1/2}|(\eta'-k't)Z_{\xi'}|  d \eta' d\eta dk' dt\\
        & \quad + \sup_{\R} \int_0^T \iiint_{\R^3} 1_{HL}\langle \xi\rangle^{n} |k|^{1/2} p_\xi^{-1/4} |Z_\xi| |k-k'|^{1/2} p_{\xi-\xi'}^{-1/2} |Z_{\xi-\xi'}| \vkp^n |k'|^{-1/2}\\
        &\quad \quad \quad \quad \quad \quad|(\eta'-k't)Z_{\xi'}|  d \eta' d\eta dk'dt\\
        &\eqqcolon T_{\infty,Z_2, LH}^{x, a} + T_{\infty, Z_2, LH}^{x,b} + T_{\infty,Z_2, HL}^x.
    \end{split}
\end{equation}
We begin by treating $T_{\infty,Z_2, LH}^{x, a}$, where we are in the $LH$ regime and $|k-k'| < 1$. Here, we use $|k-k'| \lesssim |k'|$. Then, we employ Young's inequality in $\eta$, followed by H\"older's inequality in time, and the Young's inequality in $k$. Together with interpolation, this gives us
\begin{equation}\label{infty_z2_LH}
    \begin{split}
        T_{\infty, Z_2, LH}^{x,a} &\lesssim \sup_\R \int_0^T \iiint_{\R^3} 1_{LH} 1_{|k-k'| < 1}\langle \xi\rangle^{n} |k|^{1/2} p_\xi^{-1/4} |Z_\xi| \vkm^n p_{\xi -\xi'}^{-1/2}\\
        &\quad \quad \quad \quad \quad |Z_{\xi-\xi'}| |(\eta' - k't) Z_{\xi'}| d\eta' d\eta dk' dt \\
        & \lesssim ||\langle \xi \rangle^n \frac{|k|^{1/2}}{p_\xi^{1/4}} Z_\xi ||_{L^\infty_k L^4_t([0,T]) L^2_\eta}  || \vk^n 1_{|k| < 1} p_\xi^{-1/2} Z_\xi||_{L^1_k L^4_t([0,T]) L^1_\eta}|| \vk^n (\eta-kt) Z_\xi ||_{L^\infty_k L^2_t([0,T]) L^2_\xi}\\
        &\lesssim \D_\tau^{1/4} \sup_{t \in [0,T]}\E^{1/4}(t) || \vk^n 1_{|k| < 1}p_\xi^{-1/2} Z_\xi||_{L^1_k L^4_t([0,T]) L^1_\eta} \mu^{-1/2}\D_\gamma^{1/2}.
        \end{split}
    \end{equation}
To handle the $L^1_k L^4_t([0,T]) L^1_\eta$ factor, we use a variant of Lemma \ref{main_D_tau_lemma} in the style of \eqref{replacement_lemma}:
\begin{equation}
 || \vk^n 1_{|k| < 1} p_\xi^{-1/2} Z_\xi||_{L^1_k L^4_t([0,T]) L^1_\eta} \lesssim \sup_{t \in [0,T]}\E^{1/4}(t)\D_\tau^{1/4} + \mu^{-\delta_*} \D_\gamma^{\delta_*}  \D_\tau^{1/4 - \delta_*} \sup_{t \in [0,T]}\E^{1/4}(t),
\end{equation}
which yields the desired bound on $T_{\gamma, Z_2, LH}^{x,a}$. For $T_{\gamma, Z_2, LH}^{x,b}$, we again use Young's inequality in $\eta$ and then H\"older's inequality in time. Next, we use H\"older's inequality in $k$, then Minkowski's inequality. Lastly we use Young's inequality in $k$ followed by H\"older's inequality in time, although now we place the factors in different $L^p_k$ spaces:
\begin{equation}\label{infty_z2_LH_b}
    \begin{split}
        T_{\infty, Z_2, LH}^{x,b} &\lesssim \sup_\R \int_0^T \iiint_{\R^3} 1_{LH} 1_{|k-k'| \geq 1}\langle \xi\rangle^{n} |k|^{1/2} p_\xi^{-1/4} |Z_\xi| \vkm^n p_{\xi -\xi'}^{-1/2}\\
        &\quad \quad \quad \quad \quad |Z_{\xi-\xi'}| |(\eta' - k't) Z_{\xi'}| d\eta' d\eta dk' dt\\
        & \lesssim ||\langle \xi \rangle^n \frac{|k|^{1/2}}{p_\xi^{1/4}} Z_\xi ||_{L^\infty_k L^4_t([0,T]) L^2_\eta} \sup_\R \biggl(\int_0^T \biggl( \int_\R || \langle k-k', \eta\rangle^n 1_{|k-k'| \geq 1} p_{k-k',\eta}^{-1/2} Z_{(k-k',\eta)}||_{L^1_\eta}\\
        &\quad \quad \quad \quad \quad \quad \quad \quad \quad \quad|| \langle k' , \eta\rangle^n (\eta-k't) Z_{(k',\eta)} ||_{L^2_\eta}dk'\biggr)^4 dt\biggr)^{3/4}\\
        &\lesssim \D_\tau^{1/4} \sup_{t \in [0,T]}\E^{1/4}(t) \biggl(\int_0^T \biggl(  \sup_\R \int_\R || \langle k-k', \eta\rangle^n 1_{|k-k'| \geq 1} p_{k-k',\eta}^{-1/2} Z_{(k-k',\eta)}||_{L^1_\eta}\\
        &\quad \quad \quad \quad \quad \quad \quad \quad \quad \quad|| \langle k' , \eta\rangle^n (\eta-k't) Z_{(k',\eta)} ||_{L^2_\eta}dk'\biggr)^4 dt\biggr)^{3/4}\\
        &\lesssim \D_\tau^{1/4} \sup_{t \in [0,T]}\E^{1/4}(t)  || \vk^n 1_{|k| \geq 1} p_\xi^{-1/2} Z_\xi||_{L^4_t([0,T]) L^2_k  L^1_\eta}|| \vk^n (\eta-kt) Z_\xi ||_{L^2_t([0,T]) L^2_k L^2_\xi}\\
        &\lesssim \D_\tau^{1/4} \sup_{t \in [0,T]}\E^{1/4}(t) || \vk^n 1_{|k| \geq 1} p_\xi^{-1/2} Z_\xi||_{L^4_t([0,T]) L^2_k  L^1_\eta} \mu^{-1/2}\D_\gamma^{1/2}.
        \end{split}
    \end{equation}
Applying a variation of Lemma \ref{main_D_tau_lemma} following the pattern of \eqref{replacement_lemma} gives the desired bound on $T_{\infty, Z_2, LH}^{x,b}$. For $T_{\infty, Z_2, HL}$, we start by using $|k'| \lesssim |k-k'|$. Then we have by Young's inequality, H\"older's inequality and interpolation in $k$ with $m > 0$,
\begin{equation}
    \begin{split}
        T_{\infty, Z_2, HL}^x &\lesssim \sup_\R \int_0^T \iiint_{\R^3}  1_{HL} \langle \xi\rangle^{n} |k|^{1/2} p_\xi^{-1/4} |Z_\xi| \vkm^n |k-k'|^{1/2} \min(|k-k'|^{1/4},1) p_{\xi-\xi'}^{-1/2} \\
        &\quad\quad\quad\quad |Z_{\xi-\xi'}|
     |k'|^{-1/2}\min(|k'|^{-1/4},1)|(\eta' - k't) Z_{\xi'}| d\eta' d\eta dk' dt\\
        &\lesssim || \langle \xi \rangle^n \frac{|k|^{1/2}}{p_\xi^{1/4}} Z_\xi ||_{L^\infty_k L^4_t([0,T]) L^2_\eta}  ||\langle \xi \rangle^n |k|^{1/2} \min(|k|^{1/4},1) p_\xi^{-1/2} Z_\xi||_{L^\infty_k L^4_t([0,T]) L^1_\eta}\\
        &\quad \quad || \vk^n|k|^{-1/2} \max(|k|^{-1/4},1) (\eta-kt) Z_\xi ||_{L^1_k L^2_t([0,T])L^2_\eta}\\
        & \lesssim D_\tau^{1/4} \sup_{t \in [0,T]}\E^{1/4}(t) \left(||\langle \xi \rangle^n |k|^{1/2} \min(|k|^{1/4},1) p_\xi^{-1/2} Z_\xi||_{L^\infty_k L^4_t([0,T]) L^1_\eta}\right)\mu^{-1/2} \mathcal{D}^{1/2}_\gamma.
    \end{split}
\end{equation}
Applying interpolation in $\eta$ similar to that performed in \eqref{replacement_lemma}, we find
\begin{equation}
\begin{split}
    T_{\infty, Z_2, HL}^x & \lesssim D_\tau^{1/4} \sup_{t\in [0,T]} \E^{1/2}(t)\left(\D_\tau^{1/4} + \mu^{-\delta_*} \D_\gamma^{\delta_*}\D_{\tau}^{1/4 - \delta_*} \right)\mu^{-1/2} \mathcal{D}^{1/2}_\gamma\\
    &\lesssim \mu^{-1/2-\delta_*} \D\sup_{t \in [0,T]}\E^{1/2}(t).
\end{split}
\end{equation}
We now turn to the $Z_1$ terms. Here, we split according to
\begin{equation}
    T_{\infty, Z_1}^x \leq T_{\infty, Z_1, LH} + T_{\infty, Z_1, HL}^x.
\end{equation}
In the $HL$ case, we use $|k|, |k'| \lesssim |k-k'|$, alongside Young's inequality, H\"older's inequality, interpolation in $k$, and a variation on Lemma \ref{main_D_tau_lemma}:
\begin{equation}
    \begin{split}
        T_{\infty, Z_1, HL}^x & \lesssim \sup_{\R} \int_0^T\iiint_{\R^3} 1_{HL} \langle \xi \rangle^{n} |Z_\xi| \vkm^n \min(|k-k'|, |k-k'|^{5/4})  p_{\xi - \xi'}^{-3/4} \\
        &\quad \quad \quad \quad|Z_{\xi-\xi'}| \vkp^n |k'|^{-1/2}\left(|k'|^{-1/4}1_{|k-k'| \leq 1} + 1_{|k-k'|>1}\right) |(\eta' - k't) Z_{\xi'}|  d \eta' d \eta dk' dt\\
        &\lesssim ||\langle \xi \rangle^{n} Z_\xi ||_{L^\infty_k L^4_t([0,T]) L^2_\eta} || \vk^n \min(|k|, |k|^{5/4}) p_\xi^{-3/4} Z_\xi||_{L^\infty_k L^4_t([0,T]) L^1_\eta}  \\
        &\quad \quad \quad || \vk^n |k|^{-3/4} (\eta - kt) Z_\xi||_{L^1_k L^4_t([0,T]) L^2_\eta}\\
        &\lesssim \sup_{t \in [0,T]}\E^{1/2}(t)(\D_\tau^{1/2} +  \mu^{-\delta_*} \D_\tau^{1/2 - \delta_*} \D_\gamma^{\delta_*} )\mu^{-1/2} D_\gamma^{1/2}.
    \end{split}
\end{equation}
The $LH$ case is not significantly harder. We use $|k| \approx |k'|$, Young's inequality, interpolation, H\"older's inequality and a variant of Lemma \ref{main_D_tau_lemma} in the spirit of \eqref{replacement_lemma_1} to find:
\begin{equation}
    \begin{split}
        T_{\infty, Z_1, LH}^x & \lesssim \sup_{\R}\int_0^T \iiint_{\R^3} 1_{LH} \langle \xi \rangle^{n} |Z_\xi| |k-k'|^{1/2} p_{\xi - \xi'}^{-3/4} |Z_{\xi-\xi'}| \langle \xi' \rangle^n |(\eta' - k't) Z_{\xi'}|  d \eta' d \eta dk' dt\\
        &\lesssim ||\langle \xi \rangle^{n} Z_\xi ||_{L^\infty_k L^\infty_t([0,T]) L^2_\eta} || \min(|k|, |k|^{5/4}) |k|^{-3/4} p_\xi^{-3/4} Z_\xi||_{L^1_k L^2_t([0,T]) L^1_\eta}\\
        &\quad \quad \quad || \langle \xi \rangle^n (\eta - kt) Z_\xi||_{L^\infty_k L^2_t([0,T]) L^2_\eta}\\
        &\lesssim \sup_{t \in [0,T]}\E^{1/2}(t) (\D_\tau^{1/2} +  \mu^{-\delta_*} \D_\tau^{1/2 - \delta_*} \D_\gamma^{\delta_*} )\mu^{-1/2} D_\gamma^{1/2},
    \end{split}
\end{equation}
which suffices.

\subsubsection{y-derivatives}
We need to bound
\begin{equation}
    \begin{split}
        |T_{\infty, Z}^y| &\lesssim \sup_\R \int_0^T \iiint_{\R^3}  \langle \xi \rangle^{n} |k|^{1/2} |Z_\xi| \vkm^n |k-k'|^{-1/2} p_{\xi-\xi'}^{-3/4} |\eta - \eta' + (k-k')t|\\
        &\quad \quad \quad \quad |Z_{\xi-\xi'}| \vkp^n |k'|^{1/2} |Z_{\xi'}| d\eta' d\eta dk' dt \\
        &\quad + \sup_\R \int_0^T \iiint_{\R^3} \langle \xi \rangle^{n} |k|^{1/2} p_{\xi}^{-1/4}|Z_\xi| |k-k'|^{-1/2} \vkm^n p_{\xi-\xi'}^{-1/2}\\
        &\quad \quad \quad \quad \quad |(\eta - \eta' + (k-k')t) Z_{\xi-\xi'}| \vkp^n|k'|^{1/2} |Z_{\xi'}| d\eta' d\eta dk' dt\\
        &\eqqcolon T_{\infty, Z_1}^y + T_{\infty, Z_2}^y.
    \end{split}
\end{equation}
We split according to
\begin{equation}
    T_{\infty, Z_1}^y \leq T_{\infty, Z_1, LH}^y + T_{\infty, Z_1, HL, (\cdot, H)}^y +  T_{\infty, Z_1, HL, (\cdot, L)}^y.
\end{equation}
In the $LH$ setting, we exploit $|k| \approx |k'|$, $Z_{\xi'} = p_{\xi'} p_{\xi'}^{-1} Z_{\xi'}$, the triangle inequality, and boundedness of the Riesz Transform to write
\begin{equation}\label{base_infty_Z1_LH_y}
    \begin{split}
        T_{\infty, Z_1, LH}^y &\lesssim \sup_\R \int_0^T \iiint_{\R^3}  1_{LH}\langle \xi \rangle^{n} |\eta - kt| |Z_\xi| \vkm^n |k-k'|^{-1/2} p_{\xi-\xi'}^{-1/4} \\
        &\quad \quad \quad \quad \quad \quad | Z_{\xi-\xi'}| \langle \xi' \rangle^n |k'| |\eta' - k't| p_{\xi'}^{-1} |Z_{\xi'}|\\
        &\quad \quad \quad \quad +1_{LH}\langle \xi \rangle^{n} \min(|k|^{1/4},1) |Z_\xi| \vkm^n |k-k'|^{-1/2} \max(|k-k'|^{1/4},1) \\
        &\quad \quad \quad \quad \quad \quad |(\eta - \eta') - (k-k')t|^{1/2} |Z_{\xi-\xi'}| \langle \xi' \rangle^n |k'| |\eta' - k't| p_{\xi'}^{-1} |Z_{\xi'}|\\
        & \quad \quad \quad \quad + 1_{LH}\langle \xi \rangle^{n} |k||Z_\xi| \vkm^n |k-k'|^{-1/2} p_{\xi - \xi'}^{-1/4} \\
        &\quad \quad \quad \quad \quad \quad |Z_{\xi-\xi'}| \langle \xi' \rangle^n |k'
        |^2 p_{\xi'}^{-1} |Z_{\xi'}| d\eta' d\eta dk' dt\\
        & \eqqcolon T_{\infty, Z_1, LH}^{y,a} + T_{\infty, Z_1, LH}^{y,b} + T_{\infty, Z_1, LH}^{y,c}.
    \end{split}
\end{equation}
Here we see some of the additional complications which arise because of the time-integration interacting with the interpolation. In particular, to estimate each term of \eqref{base_infty_Z1_LH_y}, we will follow a similar procedure to \eqref{infty_z2_LH} and \eqref{infty_z2_LH_b} by splitting according to $|k-k'| < 1$ and $|k-k'| \geq 1$. Treating each component with Young's ineuqality, H\"older's inequality, and Minkowski's inequality as needed, we obtain
\begin{equation}
    \begin{split}
    T_{\infty, Z_1, LH}^{y,a} &\lesssim || \langle \xi \rangle^n |\eta-kt| Z_\xi||_{L^\infty_k L^2_t([0,T]) L^2_\eta}\biggl( || \vk^n |k|^{-1/2} p_\xi^{-1/4} 1_{|k| < 1}Z_{\xi}||_{L^1_k L^{1/\delta_*}_t([0,T]) L^1_\eta}\\
    & \quad \quad\quad\quad\quad\quad\quad\quad\quad\quad\quad\quad\quad\quad\quad||\langle \xi \rangle^n |k'| p_\xi^{-1/2} Z_\xi||_{L^{\infty}_k L^{2/(1-2\delta_*)}_t([0,T])L^2_\eta}\\
    &\quad \quad\quad\quad\quad\quad\quad\quad\quad\quad\quad\quad\quad +|| \vk^n |k|^{-1/2} p_\xi^{-1/4} 1_{|k| > 1}Z_{\xi}||_{L^{1/\delta_*}_t([0,T]) L^2_k  L^1_\eta}\\
    &\quad \quad\quad\quad\quad\quad\quad\quad\quad\quad\quad\quad\quad\quad\quad ||\langle \xi \rangle^n |k'| p_\xi^{-1/2} Z_\xi||_{L^{2/(1-2\delta_*)}_t([0,T]) L^{2}_k L^2_\eta}\biggr).
    \end{split}
\end{equation}
Now we carefully interpolate in $\eta$ and $k$ (similar to \eqref{replacement_lemma}) to obtain
\begin{equation}
    \begin{split}
        T_{\infty, Z_1, LH}^{y,a} &\lesssim \mu^{-1/2} \D_\gamma^{1/2}\sup_{t \in [0,T]}\E^{1/2 - \delta_*}(t) \mu^{-\delta_*} \D_\gamma^{\delta_*}\sup_{t \in [0,T]}\E^{\delta_*}(t)\D_\tau^{1/2-\delta_*},
    \end{split}
\end{equation}
which suffices. For $T_{\infty, Z_1, LH}^{y,b}$, we also split $|k-k'| < 1$ and $|k-k'| \geq 1$. Applying Young's and H\"older's inequalities, then interpolating, we find
\begin{equation}
    \begin{split}
        T_{\infty, Z_1, LH}^{y,b} &\lesssim ||\langle \xi \rangle^n \min(|k|^{1/4},1) Z_\xi||_{L^\infty_k L^{4/(1-4\delta_*)}_t([0,T]) L^1_\eta}\\
        &\quad  \biggl(||\vk^n |k|^{-3/4}|\eta -kt|^{1/2} 1_{|k| < 1} Z_\xi||_{L^1_k L^{4}_t([0,T]) L^2_\eta} ||\langle \xi \rangle^n |k| p_{\xi}^{-1/2} Z_\xi||_{L^\infty_k L^{2/(1+2\delta_*)}_t([0,T]) L^2_\eta}\\
        &\quad \quad + ||\vk^n |k|^{-3/4}|\eta -kt|^{1/2} 1_{|k| \geq 1} Z_\xi||_{L^{4}_t([0,T]) L^2_k L^2_\eta} ||\langle \xi \rangle^n |k| p_{\xi}^{-1/2} Z_\xi||_{ L^{2/(1+2\delta_*)}_t([0,T]) L^2_k L^2_\eta}\biggr)\\
        &\lesssim \sup_{t\in [0,T]} \E^{1/4-\delta_*}(t) \mu^{-1/4 - \delta_*} \D_\gamma^{1/4 + \delta_*}  \sup_{t\in [0,T]} \E^{1/4}(t) \mu^{-\delta_*}\D_\gamma^{1/4} \sup_{t\in [0,T]} \E^{\delta_*}(t) \D_\tau^{1/2 - \delta_*}.
    \end{split}
\end{equation}
Our last step in controlling $T_{\infty, Z_1, LH}^y$ is to estimate the $T_{\infty, Z_1, LH}^{y,c}$ term. We apply Young's inequality in $\eta$, H\"older's inequality in time, Young's inequality in $k$, and interpolate in $k$ and $\eta$:
\begin{equation}
    \begin{split}
    T_{\infty, Z_1, LH}^{y,c} & \lesssim || \langle \xi \rangle^n|k| Z_\xi||_{L^\infty_k L^2_t([0,T]) L^2_\eta}\\
    &\quad \biggl( ||\vk^n |k|^{-1/2} p_\xi^{-1/4} 1_{|k| < 1} Z_\xi||_{L^1_k L^{1/\delta_*}_t([0,T]) L^1_\eta} ||\langle \xi \rangle^n |k| p_\xi^{-1/2} Z_\xi ||_{L^\infty_k L^{2/(1-\delta_*)}_t([0,T])L^2_\eta}\\
    &\quad \quad +  ||\vk^n |k|^{-1/2} p_\xi^{-1/4} 1_{|k| \geq 1} Z_\xi||_{L^{1/\delta_*}_t([0,T]) L^2_k  L^1_\eta} ||\langle \xi \rangle^n |k| p_\xi^{-1/2} Z_\xi ||_{L^{2/(1-\delta_*)}_t([0,T]) L^2_k L^2_\eta}\biggl)\\
    &\lesssim \mu^{-1/2} \D_\gamma^{1/2} \sup_{t \in [0,T]}\E^{1/2 - \delta_*}(t) \mu^{-\delta_*}\D_\gamma^{\delta_*} \sup_{t \in [0,T]}\E^{\delta_*}(t) \D_\tau^{1/2 + \delta_*}.
    \end{split}
\end{equation}
Having completed the $LH$ terms for $T_{\infty, Z_1}^y$, we turn to the $HL$ terms. In the $HL$ cases, we treat high and low frequencies in $k'$ differently. For high frequencies, we use a variant of Lemma \ref{main_D_tau_lemma}, Young's inequality, and $|k|, |k'| \lesssim |k-k'|$ to compute
\begin{equation}
    \begin{split}
        T_{\infty, Z_1, HL, (\cdot, H)}^y &\lesssim \sup_\R \int_0^T \int_{|k'| \geq \mu}\iint_{\R^2}  \langle \xi \rangle^{n}  |Z_\xi| \vkm^n |k-k'|^{1/2} (1_{|k-k'| \geq 1} + |k-k'|^{1/4}) 1_{|k-k'| < 1}p_{\xi-\xi'}^{-1/4}\\
        &\quad \quad \quad \quad \quad \quad |Z_{\xi-\xi'}| \vkp^n |k'|^{-1/2}(1_{|k-k'| \geq 1} + |k'|^{-1/4} 1_{|k-k'| < 1})|k'|^{1/2} |Z_{\xi'}| d\eta' d\eta dk' dt\\
        &\lesssim \sup_{t \in [0,T]}\E^{1/2} ||\vk^n|k|^{1/2} \min(1, |k|^{1/4}) p_\xi^{-1/4} Z_\xi||_{L^\infty_k L^2_t([0,T])L^1_\eta}\\
        &\quad \quad \quad || \langle \xi \rangle^n \max(|k|^{-1/2}, |k|^{-3/4}) |k|^{1/2} Z_\xi||_{L^1_{|k| \geq \mu} L^2_t([0,T]) L^2_\eta}\\
        &\lesssim \E^{1/2}(\mu^{-1/4} \D_\gamma^{1/4} \D_\tau^{1/4} + \D_\tau^{1/4 - \delta_*} \mu^{-1/4 - \delta_*} \D_\gamma^{1/4 + \delta_*}) \mu^{-1/8} \D_\beta^{3/8} \mu^{-1/8} \D_\gamma^{1/8}.
    \end{split}
\end{equation}
For low-in-$k'$ frequencies, we use $p_\xi^{1/4} \lesssim p_{\xi-\xi'}^{1/4} + p_{\xi'}^{1/4}$ to write
\begin{equation}
    \begin{split}
         T_{\infty, Z_1, HL, (\cdot, L)}^y & \lesssim \sup_\R \int_0^T \int_{|k'| < \mu}\iint_{\R^2}  1_{HL}\langle \xi \rangle^{n} |k|^{1/2} p_{\xi}^{-1/4} \langle \xi - \xi' \rangle^n  |Z_\xi| |k-k'|^{-1/2} p_{\xi-\xi'}^{-1/2} \\
         & \quad \quad \quad \quad \quad |(\eta - \eta' + (k-k')t) Z_{\xi-\xi'}|  \vkp^n |k'|^{1/2} |Z_{\xi'}| d\eta' d\eta dk'\\
         &\quad + \sup_\R\int_{|k'| < \mu}\iint_{\R^3}  1_{HL}\langle \xi \rangle^{n} |k|^{1/2} p_{\xi}^{-1/4} |Z_\xi| \vkm^n |k-k'|^{-1/2} p_{\xi-\xi'}^{-3/4}\\
         & \quad \quad \quad \quad \quad |(\eta - \eta' + (k-k')t) Z_{\xi-\xi'}| \vkp^n |k'|^{1/2} p_{\xi'}^{1/4} |Z_{\xi'}| d\eta' d\eta dk' dt\\
         &\lesssim T_{\infty, Z_2, HL}^y + (I),
    \end{split}
\end{equation}
where
\begin{equation}
\begin{split}
    (I) &\coloneqq \sup_\R \int_0^T \int_{|k'| < \mu}\iint_{\R^3}  1_{HL}\langle \xi \rangle^{n} |k|^{1/2} p_{\xi}^{-1/4} |Z_\xi| \vkm^n |k-k'|^{-1/2} p_{\xi-\xi'}^{-3/2} \\
    &\quad \quad \quad \quad |(\eta - \eta' + (k-k')t) Z_{\xi-\xi'}| \vkp^n |k'|^{1/2} p_{\xi'}^{1/4} |Z_{\xi'}| d\eta' d\eta dk' dt.
    \end{split}
\end{equation}
The term $T_{Z_2, HL}^y$ is controlled in \eqref{infty_Z2_HL_y}. Then to treat $(I)$, we employ $|k'| \lesssim |k-k'|$,  alongside Young's inequality and a variant of Lemma \ref{main_D_tau_lemma} as follows:
\begin{equation}
    \begin{split}
        (I) & \lesssim \sup_\R \int_0^T \int_{|k'| < \mu}\iint_{\R^3}  1_{HL}\langle \xi \rangle^{n} |k|^{1/2} p_{\xi}^{-1/4}|Z_\xi| \vkm^n |k-k'|^{1/2} \min(1, |k-k'|^{1/4})\\
        &\quad \quad \quad \quad p_{\xi - \xi'}^{-1/4}|Z_{\xi-\xi'}| \vkp^n (1_{|k -k'| < 1}|k'|^{-1/4}+ 1_{|k-k'| \geq 1}) p_{\xi'}^{1/4} |Z_{\xi'}| d\eta' d\eta dk' dt\\
        &\lesssim ||\langle \xi \rangle^{n} |k|^{1/2} p_{\xi}^{-1/4} Z_\xi||_{L^\infty_k L^2_t([0,T]) L^2_\eta} || \vk^n  \min(|k|^{1/2}, |k|^{3/4} )p_{\xi}^{-1/4} Z_\xi||_{L^\infty_k L^4_t([0,T]) L^1_\eta}\\
        &\quad \quad \quad \quad || \vk^n |k|^{-3/4} p_{\xi}^{1/4} Z_\xi||_{L^1_{|k|<\mu} L^4_t([0,T]) L^2_\eta}\\
        &\lesssim \D_\tau^{1/4} \sup_{t \in [0,T]}\E^{1/4}(t)\mu^{-1/4}\left(\D_\tau^{1/4} \D_\gamma^{1/4} + \mu^{-\delta_*} \D_\tau^{1/4 - \delta_*} \D_\gamma^{1/4 + \delta_*}\right) \mu^{-1/4} \D_\gamma^{1/4} \sup_{t \in [0,T]}\E^{1/4}(t).
        \end{split}
\end{equation}
To complete our estimate of $T_{\infty, Z}^y$, it remains to bound $T_{\infty, Z_2}^y$. We split into $LH$ and $HL$ terms as usual:
\begin{equation}
    T_{\infty, Z_2}^y \leq T_{\infty, Z_2, LH}^y + T_{\infty,Z_2, HL}^y.
\end{equation}
For the $LH$ terms, we again use $p_\xi^{1/4} \lesssim p_{\xi-\xi'}^{1/4} + p_{\xi'}^{1/4}$. Furthermore, we split each component of the resulting decomposition along $|k-k'| < 1$ and $|k-k'| \geq 1$. Then we proceed by interpolation, $|k| \approx |k'|$, Young's inequality, and a variant of Lemma \ref{main_D_gamma_lemma} adapted to the new setting:
\begin{equation}
    \begin{split}
        T_{\infty, Z_2, LH}^y &\lesssim \sup_\R \int_0^T \iiint_{\R^3} 1_{LH} \langle \xi \rangle^{n} |k| p_{\xi}^{-1/2}|Z_\xi| |k-k'|^{-1/2} \vkm^n \vkp^n \\
        &\quad\quad \quad \quad \biggl( p_{\xi-\xi'}^{1/4}|Z_{\xi-\xi'}||Z_{\xi'}| + (|k-k'|^{-1/4} |k'|^{1/4} 1_{|k'| \leq 1} + 1_{|k'| > 1})|Z_{\xi-\xi'}|p_{\xi'}^{1/4} |Z_{\xi'}|\biggr)d\eta' d\eta dk' dt\\
        &\lesssim ||\vk^n k p
        _\xi^{-1/2} Z_\xi||_{L^\infty_k L^{2/(1-2\delta_*)}_t([0,T]) L^2_\eta}\\
        &\quad \biggl( || \vk^n |k|^{-3/4} p_\xi^{1/4} 1_{|k| < 1}Z_\xi||_{L^1_k L^{4}_t([0,T]) L^2_\eta}  || \vk^n \min(1, |k|^{1/4}) Z_\xi||_{L^\infty_k L^{4/(1+4 \delta_*)}L^1_\eta}\\ 
        &\quad \quad + || \vk^n |k|^{-1/2} p_\xi^{1/4} 1_{|k| \geq 1}Z_\xi||_{ L^{4}_t([0,T]) L^2_k L^2_\eta}  || \vk^n \min(1, |k|^{1/4}) Z_\xi||_{ L^{4/(1+4 \delta_*)} L^2_k L^1_\eta}\\
        &\quad \quad +
        || \vk^n |k|^{-1/2} 1_{|k| < 1}Z_\xi||_{L^1_k L^{4/(1+4 \delta_*)} L^1_\eta} ||\vk^n p_\xi^{1/4} Z_\xi||_{L^\infty_k L^4_t([0,T])L^2_\eta}\\
        &\quad \quad  +
        || \vk^n |k|^{-1/2} 1_{|k| \geq 1}Z_\xi||_{L^{4/(1+4 \delta_*)} L^2_k 
L^1_\eta} ||\vk^n p_\xi^{1/4} Z_\xi||_{ L^4_t([0,T]) L^2_k L^2_\eta} \biggr)\\
        &\lesssim \sup_{t\in [0,T]}\E^{\delta_*}(t) \D_\tau^{1/2 - \delta_*}\biggl(\mu^{-1/4}\D_\gamma^{1/4}\sup_{t \in [0,T]} \E^{1/4}(t) (\mu^{-1/4 - \delta_*} \D_\gamma^{1/4 + \delta_*} \sup_{t \in [0,T]}\E^{1/4 - \delta_*}(t) + \mu^{-1/4} \D_\gamma^{1/4} \E^{1/4})\\
        &\quad \quad \quad \quad \quad + (\mu^{-1/4 - \delta_*} \D_\gamma^{1/4 + \delta_*} \sup_{t \in [0,T]}\E^{1/4 - \delta_*}(t) + \mu^{-1/4} \D_\gamma^{1/4} \E^{1/4}(t))\mu^{-1/4}\D_\gamma^{1/4} \sup_{t \in [0,T]}\E^{1/4}(t) \biggr)\\
&\lesssim \mu^{-1/2 - \delta_*} \D \sup_{t \in [0,T]}\E^{1/2}(t).
    \end{split}
\end{equation}
For the $HL$ terms, we use a similar strategy to that employed in \eqref{decomposed_HL_gamma_y} of Section \ref{gamma_terms_for_Z}, namely $1 = p_{\xi'}^{-1} p_{\xi'}$. We then distribute by the triangle inequality, and employ $|k'| \lesssim |k-k'|$ and boundedness of the Riesz Transform to arrive at
\begin{equation}\label{infty_Z2_HL_y}
    \begin{split}
         T_{\infty, Z_2, HL}^y & \lesssim \sup_\R \int_0^T \iiint_{\R^3} 1_{HL}\langle \xi \rangle^{n} \langle \xi - \xi'\rangle^n \vkp^n \biggl( |k|^{1/2} p_{\xi}^{-1/4}|Z_\xi| |k-k'|^{1/2} Z_{\xi-\xi'}| |k'|^{3/2} p_{\xi'}^{-1} |Z_{\xi'}|\\
         &\quad \quad \quad \quad \quad + |k|^{1/2} p_{\xi}^{-1/4}|\eta -kt|^{1/2}|Z_\xi| \min(1,|k-k'|^{1/4})\\
         &\quad \quad \quad \quad \quad \quad \quad \quad |Z_{\xi-\xi'}| |k'|^{-1/2}(1_{|k-k'| \geq 1} + |k'|^{-1/4}1_{|k-k'| < 1}) |k'|^{1/2} |\eta' -k't|^{3/2}p_{\xi'}^{-1} |Z_{\xi'}|\\
         &\quad \quad \quad \quad \quad + |k|^{1/2} p_{\xi}^{-1/4}|Z_\xi| |(\eta -\eta')-(k-k')t|^{1/2}|Z_{\xi-\xi'}|\\
          &\quad \quad \quad \quad \quad \quad \quad \quad |k'|^{-1/2} |k'|^{1/2} |\eta' -k't|^{3/2}p_{\xi'}^{-1} |Z_{\xi'}|\biggr)d\eta' d\eta dk'.
        \end{split}
\end{equation}
Continuing from \eqref{infty_Z2_HL_y}, by Young's inequality, H\"older's inequality, and interpolation in $k$ (using $m > 0$):
\begin{equation}\label{T_infty_Z2_HL_y}
    \begin{split}
        T_{\infty, Z_2, HL}^y &\lesssim || \langle \xi \rangle^n |k|^{1/2} p_\xi^{-1/4} Z_\xi||_{L^\infty_k L^4_t ([0,T])L^2_\eta} || \langle \xi \rangle^n |k|^{1/2} Z_\xi||_{L^\infty_k L^4_t ([0,T])L^2_\eta}\\
        &\quad \quad \quad \quad || \vk^n |k|^{1/2} p_\xi^{-1/2} Z_\xi||_{L^1_k L^{2}_t([0,T])L^1_\eta}\\
        & \quad + || \langle \xi \rangle^n |k|^{1/2} p_\xi^{-1/4} |\eta-kt|^{1/2} Z_\xi||_{L^\infty_k L^4_t([0,T]) L^2_\eta} \\
        &\quad \quad \biggl( || \langle \xi \rangle^n \min(1,|k|^{1/4}) Z_\xi||_{L^\infty_k L^{4/(1+4\delta_*)}_t([0,T]) L^1_\eta} || \vk^n |k|^{-3/4} |k|^{1/2} |\eta -kt|^{1/2} p_\xi^{-1/2} 1_{|k| < 1}Z_\xi||_{L^1_k L^{2/(1-2\delta_*)}_t([0,T]) L^2_\eta}\\
        &\quad \quad +  || \langle \xi \rangle^n \min(1,|k|^{1/4}) Z_\xi||_{ L^{4/(1+4\delta_*)}_t([0,T]) L^2_k L^1_\eta} || \vk^n |k|^{1/2} |\eta -kt|^{1/2} p_\xi^{-1/2} Z_\xi||_{ L^{2/(1-2\delta_*)}_t([0,T]) L^2_k L^2_\eta}\biggr)\\
        &\quad + || \langle \xi \rangle^n |k|^{1/2} p_\xi^{-1/4} Z_\xi||_{L^\infty_k L^4_t([0,T]) L^2_\eta} || \langle \xi \rangle^n |\eta -kt|^{1/2} Z_\xi||_{L^\infty_k L^4_t([0,T]) L^2_\eta}\\
        &\quad \quad \quad \quad ||  \vk^n |k|^{-1/2} |k|^{1/2} p_\xi^{-1/4}Z_\xi||_{L^1_k L^2_t([0,T])L^1_\eta}\\
        &\lesssim \D_\tau^{1/4} \sup_{t \in [0,T]}\E^{1/4}(t) \mu^{-1/4} \D_\gamma^{1/4} \sup_{t \in [0,T]}\E^{1/4}(t) || \vk^n |k|^{1/2} p_\xi^{-1/2} Z_\xi||_{L^1_k L^{2}_t([0,T])L^2_\eta}\\
        &\quad + \sup_{t \in [0,T]}\E^{1/4}(t)\mu^{-1/4}\D_\gamma^{1/4}|| \langle \xi \rangle^n \min(1,|k|^{1/4}) Z_\xi||_{L^\infty_k L^{4/(1+4\delta_*)}_t([0,T]) L^1_\eta}\\
        &\quad \quad \quad \quad\sup_{t \in [0,T]}\E^{\delta_*}(t)\D_\tau^{1/4-\delta_*}\mu^{-1/4}\D_\gamma^{1/4}\\
        &\quad + \D_\tau^{1/4} \sup_{t \in [0,T]}\E^{1/4}(t) \mu^{-1/4} \D_\gamma^{1/4} \sup_{t \in [0,T]}\E^{1/4}(t) ||  \vk^n |k|^{-1/2} |k|^{1/2} p_\xi^{-1/4}Z_\xi||_{L^1_k L^2_t([0,T])L^1_\eta}.
        \end{split}
\end{equation}
Interpolating in $\eta$ and $k$, following the ideas of \eqref{replacement_lemma} and Lemmas \ref{main_D_tau_lemma} and \ref{main_D_gamma_lemma} gives, using $m > 0$:
\begin{equation}\label{T_infty_combo_interp}
\begin{split}
    &|| \vk^n |k|^{1/2} p_\xi^{-1/2} Z_\xi||_{L^1_k L^{2}_t([0,T])L^2_\eta} \lesssim \mu^{-1/4-\delta_*} \D_{\gamma}^{1/4+\delta_*}\D_\tau^{1/4 - \delta_*} + \D_\tau^{1/2},\\
    &|| \langle \xi \rangle^n \min(1,|k|^{1/4}) Z_\xi||_{L^\infty_k L^{4/(1+4\delta_*)}_t([0,T]) L^1_\eta} \lesssim \mu^{-1/4-\delta_*} \D_\gamma^{1/4 + \delta_*} \sup_{t \in [0,T]}\E^{1/4 - \delta_*}(t),\\
    &|| \langle \xi \rangle^n \min(1,|k|^{1/4}) Z_\xi||_{ L^{4/(1+4\delta_*)}_t([0,T]) L^2_k L^1_\eta} \lesssim \mu^{-1/4-\delta_*} \D_\gamma^{1/4 + \delta_*} \sup_{t \in [0,T]}\E^{1/4 - \delta_*}(t),\\
    &||  \vk^n |k|^{-1/2} |k|^{1/2} p_\xi^{-1/4}Z_\xi||_{L^1_k L^2_t([0,T])L^1_\eta} \lesssim \mu^{-1/4} \D_\gamma^{1/4}(\D_\tau^{1/4} + \mu^{-\delta_*} \D_\gamma^{\delta_*} \D_\tau^{1/4 - \delta_*}).
\end{split}
\end{equation}
Combining \eqref{T_infty_Z2_HL_y} with \eqref{T_infty_combo_interp},
we obtain
\begin{equation}
    T_{\infty, Z_2, HL}^y \lesssim \mu^{-1/2 - \delta_*} \D \sup_{t \in [0,T]}\E^{1/2}(t),
\end{equation}
as desired. This completes the bounds on $T_{\infty,Z}$ for the purposes of Lemma \ref{key_lemma_alt}.
\subsection{Bound on \texorpdfstring{$T_\infty$}{Supremum terms} for Q}
Our final terms to bound for the purposes of Lemma \ref{bootstrap_lemma} are $T_{\infty, Q}^x$ and $T_{\infty, Q}^y$, which are defined via
\begin{equation}
    \begin{split}
        T_{\infty,Q} &\leq  4\sup_{k \in \R} \int_0^T\int_{\R}  \langle \xi \rangle^{2n}  |N_k + c_\tau\mathfrak{J}_k|\mathrm{Re}\biggl|\iint_{\R^2} k p_\xi^{1/4} \bar{Q}_\xi\\
    &\quad \quad \quad \quad \quad \quad |k-k'|^{-1/2} p_{\xi -\xi'}^{-3/4} i (k-k') Z_{\xi - \xi'} k'^{-1} |k'|^{1/2} p_{\xi'}^{-1/4}  i(\eta'-k't) Q_{\xi'}d\eta' dk'\biggr| d\eta dt \\
    &\quad +4 \sup_{k \in \R}\int_0^T \int_{\R}  \langle \xi \rangle^{2n}  |N_k + c_\tau\mathfrak{J}_k|\mathrm{Re}\biggl|\iint_{\R^2} k p_\xi^{1/4} |k|^{-1/2}|k-k'|^{-1/2} \bar{Q}_\xi\\
    &\quad \quad \quad \quad \quad \quad  p_{\xi -\xi'}^{-3/4} i ((\eta-\eta') - (k-k')t) Z_{\xi - \xi'} k'^{-1} |k'|^{1/2} p_{\xi'}^{-1/4} i k' Q_{\xi'}d\eta' dk'\biggr| d\eta dt\\
    &\eqqcolon T_{\infty, Q}^x + T_{\infty, Q}^y.
    \end{split}
\end{equation}
\subsubsection{x-derivatives}
We begin by noting that
\begin{equation}
    \begin{split}
    |T_{\infty, Q}^x| & \lesssim \sup_\R \int_0^T \iiint_{\R^3} |k|^{1/2} \langle \xi \rangle^{n}  |Q_\xi| \vkm^n |k-k'|^{1/2} p_{\xi-\xi'}^{-3/4} |Z_{\xi-\xi'}|  \\
    &\quad \quad \quad \quad \vkp^n |k'|^{-1/2}| \eta' -k't| |Q_{\xi'}| d\eta' d\eta dk' dt \\
        &\quad + \sup_\R \int_0^T \iiint_{\R^3}  \langle \xi \rangle^{n} |k|^{1/2} |Q_\xi|\vkm^n |k-k'|^{1/2} p_{\xi-\xi'}^{-1/2} |Z_{\xi-\xi'}| \\
    &\quad \quad \quad \quad \vkp^n p_{\xi'}^{-1/4} |k'|^{-1/2} |\eta'-k't||Q_{\xi'}| d\eta' d\eta dk' dt\\
        &\eqqcolon T_{\infty, Q_1}^x + T_{\infty, Q_2}^x.
    \end{split}
\end{equation}
As we have seen in previous sections, $T_{\infty, Q_1}^x$ can be estimated in the same manner as $T_{\infty, Z_1}^x$. For $T_{\infty, Q_2}^x$, we split
\begin{equation}
    T_{\infty, Q_2}^x \leq T_{\infty, Q_2, LH}^x + T_{\infty, Q_2, HL}^x.
\end{equation}
The $HL$ and $LH$ cases have similar structures but slightly different technical details. Starting with $LH$, we use $|k| \approx |k'|$, $|k-k'| \lesssim |k|$ boundedness of the Riesz transform, Young's inequality, H\"older's inequality, interpolation, and the second line of \eqref{T_infty_combo_interp} to obtain:
\begin{equation}\label{T_infty_Q2_LH_x}
    \begin{split}
        T_{\infty, Q_2, LH}^x & \lesssim \sup_\R \int_0^T \iiint_{\R^3}  1_{LH}\langle \xi \rangle^{n} \min(|k|^{1/4}, 1)|Q_\xi| \vkm^n |k-k'|^{1/2}\\
        &\quad \quad \quad \quad \quad (|k-k'|^{-1/4}1_{|k| \leq 1} + 1_{|k| > 1})  p_{\xi-\xi'}^{-1/2} |Z_{\xi-\xi'}|  \vkp^n |\eta'-k't|^{1/2} |Q_{\xi'}| d\eta' d\eta dk' dt\\
        &\lesssim ||\langle \xi \rangle^n \min(|k|^{1/4}, 1) Q_\xi||_{L^\infty_k L^{4/(1+4\delta_*)}_t([0,T]) L^1_\eta}\\
        &\quad  \biggl(|| \vk^n |k|^{-3/4} |k| p_\xi^{-1/2} 1_{|k| < 1} Z_\xi||_{L^1_k L^{2/(1-2\delta_*)}_t([0,T])L^2_\eta}|| \vk^n |\eta -kt|^{1/2} Q_\xi||_{L^\infty_k L^4_t([0,T])L^2_\eta}\\
        & \quad  + || \vk^n |k|^{-1/2} |k| p_\xi^{-1/2} 1_{|k| \geq 1} Z_\xi||_{L^{2/(1-2\delta_*)}_t([0,T]) L^2_k L^2_\eta}|| \vk^n |\eta -kt|^{1/2} Q_\xi||_{L^4_t([0,T]) L^2_k L^2_\eta} \biggr)\\
        &\lesssim \mu^{-1/4} \D_\gamma^{1/4 + \delta_*}\sup_{t \in [0,T]}\E^{1/4 - \delta_*}(t) \sup_{t\in [0,T]}\E^{\delta_*}(t) \D_\tau^{1/2 - \delta_*} \sup_{t \in [0,T]}\E^{1/4 }(t) \mu^{-1/4} \D_\gamma^{1/4}.
    \end{split}
\end{equation}
For the $HL$ case, we have using $|k|, |k'| \lesssim |k-k'|$, Young's inequality and interpolation similar to \eqref{T_infty_Q2_LH_x}:
\begin{equation}
    \begin{split}
        T_{\infty, Q_2, HL}^x & \lesssim \sup_\R \int_0^T \iiint_{\R^3}  1_{HL }\langle \xi \rangle^{n} \min(|k|^{1/4},1)|Q_\xi| \langle \xi - \xi'\rangle^n |k-k'|p_{\xi-\xi'}^{-1/2} |Z_{\xi-\xi'}|\\
        &\quad \quad \quad \quad \vkp^n |k'|^{-1/2}(|k'|^{-1/4}1_{|k| < 1} + 1_{|k| \geq 1}) |\eta'-k't|^{1/2}|Q_{\xi'}| d\eta' d\eta dk' dt\\
        &\lesssim ||\langle \xi \rangle^n \min(|k|^{1/4}, 1) Q_\xi||_{L^\infty_k L^{4/(1+4 \delta_*)}_t([0,T]) L^1_\eta}\\
        &\quad \biggl( || \langle \xi \rangle^n k p_\xi^{-1/2} Z_\xi||_{L^\infty_k L^{2/(1-2\delta_*)}_t([0,T]) L^2_\eta} || \vk^n |k|^{-3/4} 1_{|k| < 1} |\eta -kt|^{1/2} Q_\xi||_{L^1_k L^4_t([0,T]) L^2_\eta}\\
        &\quad \quad + || \langle \xi \rangle^n k p_\xi^{-1/2} Z_\xi||_{ L^{2/(1-2\delta_*)}_t([0,T]) L^2_k L^2_\eta} || \vk^n |k|^{-1/2} 1_{|k| \geq 1} |\eta -kt|^{1/2} Q_\xi||_{ L^4_t([0,T]) L^2_k L^2_\eta} \biggr)\\
        &\lesssim \mu^{-1/4 - \delta_*} \D_\gamma^{1/4+\delta_*} \E^{1/4-\delta_*} \sup_{t \in [0,T]}\E^{\delta_*}(t)\D_\tau^{1/2 - \delta_*} \sup_{t \in [0,T]}\E^{1/4 }(t) \mu^{-1/4} \D_\gamma^{1/4}.
    \end{split}
\end{equation}
This completes the estimates on $T_{\infty, Q}^x$.

\subsubsection{y-derivatives}
We turn last of all to $T_{\infty, Q}^y$. We have
\begin{equation}
    \begin{split}
        |T_{\infty, Q}^y|
        &\lesssim \sup_\R \int_0^T \iiint_{\R^3} \langle \xi \rangle^{n} |k|^{1/2} |Q_\xi| \vkm^n |k-k'|^{-1/2} p_{\xi -\xi'}^{-3/4} \\
        &\quad \quad \quad \quad |\eta - \eta' - (k-k')t| |Z_{\xi -\xi'}| \vkp^n |k'|^{1/2}|Q_{\xi'}| d\eta' d\eta dk' dt\\
        &\quad + \sup_\R \int_0^T\iiint_{\R^3} \langle \xi \rangle^{n} |k|^{1/2} |Q_\xi| \vkm^n |k-k'|^{-1/2} p_{\xi -\xi'}^{-1/2}\\
        &\quad \quad \quad \quad \quad |\eta - \eta' - (k-k')t| |Z_{\xi -\xi'}| \vkp^np_{\xi'}^{-1/4}|k'|^{1/2}|Q_{\xi'}| d\eta' d\eta dk' dt \\
        &\eqqcolon T_{\infty, Q_1}^y + T_{\infty, Q_2}^y.
    \end{split}
\end{equation}
Due to the similarities with $T_{\infty, Z}^y$, we will only explicitly present the bound on $T_{\infty, Q_2, HL}^y$. Using a similar method as \eqref{infty_Z2_HL_y}, we have by boundedness of the Riesz Transform and Young's inequality:
\begin{equation}\label{T_infty_Q2_HL_y}
    \begin{split}
        T_{\infty, Q_2, HL}^y & \lesssim \sup_\R \int_0^T \iiint_{\R^3} 1_{HL}\langle \xi \rangle^{n} \langle \xi - \xi'\rangle^n \vkp^n \biggl( |k|^2 p_{\xi}^{-1}|Q_\xi| |k-k'|^{1/2} |Z_{\xi-\xi'}| p_{\xi'}^{-1/4} |Q_{\xi'}|\\
         &\quad \quad \quad \quad \quad + |k|^{1/2} |\eta- kt|^{3/2} p_\xi^{-1} |Q_\xi| \min(1,|k-k'|^{1/4})\\
         &\quad \quad \quad \quad \quad \quad \quad \quad |Z_{\xi-\xi'}| \max(|k'|^{-1/4},1) |k'|^{1/2} |k'|^{-1/2} |\eta' -k't|^{1/2} p_{\xi'}^{-1/4} |Q_{\xi'}|\\
         &\quad \quad \quad \quad \quad + |k|^{1/2} |\eta -kt|^{3/2}p_{\xi}^{-1}|Z_\xi| |(\eta -\eta')-(k-k')t|^{1/2}|Z_{\xi-\xi'}|\\
          &\quad \quad \quad \quad \quad \quad \quad \quad |k'|^{-1/2}|k'|^{1/2} p_{\xi'}^{-1/4}  |Q_{\xi'}|d\eta' d\eta dk' dt\\
          &\lesssim || \langle \xi \rangle^n |k|p_\xi^{-1/2} Q_\xi||_{L^\infty_k L^{2/(1-2\delta_*)} L^2_\eta}\\
          &\quad \biggl( || \langle \xi \rangle^n |k|^{1/2} Z_\xi||_{L^\infty_k L^4_t([0,T]) L^2_\eta} || \vk^n  1_{|k| < 1}p_\xi^{-1/4} Q_\xi||_{L^1_k L^{4/(1+4\delta_*)}_t([0,T]) L^1_{\eta}}\\
          &\quad + || \langle \xi \rangle^n |k|^{1/2} Z_\xi||_{L^4_t([0,T]) L^2_k L^2_\eta} || \vk^n  1_{|k| \geq 1}p_\xi^{-1/4} Q_\xi||_{ L^{2/(1+2\delta_*)}_t([0,T]) L^2_k L^1_{\eta}}\biggr)\\
          & \quad + || \langle \xi \rangle^n |k|^{1/2} p_\xi^{-1/4} Q_\xi||_{L^\infty_k L^{4/(1-4\delta_*}_t([0,T])L^2_\eta} || \langle \xi \rangle^n \min(1,|k|^{1/4}) Z_\xi||_{L^\infty_k L^{4/(1+4\delta_*}_t([0,T])L^1_\eta} \\
        &\quad \quad \quad || \vk^n \max(|k|^{-3/4},|k|^{-1/2})|k|^{1/2}  p_\xi^{-1/4} |\eta-kt|^{1/2} Q_\xi||_{L^1_k L^2_t([0,T])L^2_\eta}\\
        &\quad + || \langle \xi \rangle^n |k|^{1/2} p_\xi^{-1/4} Q_\xi||_{L^\infty_k L^4_t([0,T]) L^2_\eta} || \langle \xi \rangle^n |\eta -kt|^{1/2} Z_\xi||_{L^\infty_k L^4_t([0,T]) L^2_\eta}\\
        &\quad \quad \quad \quad || \vk^n |k|^{-1/2} |k|^{1/2} p_\xi^{-1/4}Q_\xi||_{L^1_k L^2_t([0,T])L^1_\eta}.
    \end{split}
\end{equation}
Applying interpolation to \eqref{T_infty_Q2_HL_y} (similar to \eqref{T_infty_combo_interp}) , along with $\D_\tau \lesssim \E$, we find
\begin{equation}
    \begin{split}
        T_{\infty, Q_2, HL}^y & \lesssim \D_\tau^{1/2 - \delta_*} \sup_{t\in [0,T]}\E^{\delta_*}(t) \mu^{-1/4} \D_\gamma^{1/4} \sup_{t\in [0,T]}\E^{1/4}(t) \sup_{t\in [0,T]}\E^{1/4-\delta_*}(t) (\D_\tau^{1/4} + \mu^{-1/4} \D_\gamma^{1/4})\mu^{-\delta_*}\D_\gamma^{\delta_*}\\
        &\quad + \D_\tau^{1/4 - \delta_*} \sup_{t\in [0,T]}\E^{1/4+\delta_*} \mu^{-1/4 - \delta_*} \D_\gamma^{1/4 + \delta_*}\sup_{t\in [0,T]}\E^{1/4-\delta_*}(t)\D_\tau^{1/4} \mu^{-1/4} \D_\gamma^{1/4} \\
        &\quad + \D_\tau^{1/4} \sup_{t\in [0,T]}\E^{1/4}(t) \mu^{-1/4} \D_\gamma^{1/4} \sup_{t\in [0,T]}\E^{1/4}(t) \mu^{-1/4} \D_\gamma^{1/4}(\D_\tau^{1/4} + \mu^{-\delta_*} \D_\gamma^{\delta_*} \D_\tau^{1/4 - \delta_*})\\
        &\lesssim \mu^{-1/2 -\delta_*} \D \sup_{t \in [0,T]}\E^{1/2}(t).
    \end{split}
\end{equation}
This concludes the proof of Lemma \ref{bootstrap_lemma_alt} and hence of Theorem \ref{alt_theorem}.

\section{Appendix}

We collect here a series of interpolation Lemmas used throughout Section \ref{nonlinear_section}. There are many sub-varieties of these lemmas, and the exact details of each are not particularly mathematically insightful. The key in all of them is to interpolate from $L^1_\eta$ to $H^s_\eta$ for some $s >1/2$. For frequencies $|\eta| < 1$, we can gain powers of $|\eta|$ up to $|\eta|^{s'}$ for $s' < 1/2$. The main variations in the lemmas come from what powers $s$ and $s'$ to choose, depending on the desired estimates and the presence of powers of $|k|$.
We note that the norms used in Section \ref{nonlinear_section} contain various powers of $\dk$ and $\vk$. However, the presence or absence of these terms does not impact the final estimate. Hence for the sake of space and readability, we exclude them from the following lemmas. Regarding $\sk$, we will use $\sk^m$ to interpolate from $L^1_k$ to $L^2_k$, and so its presence cannot be ignored.

\begin{lemma}\label{new_energy_lemmas}
For any $\delta_* \in (0,1/12)$, the following estimates hold:
\begin{subequations}
\begin{equation}\label{energy_with_gamma}
    \begin{split}
        || \langle k \rangle^{-1/2} p_\xi^{-1/4} Z_\xi ||_{L^1_k L^1_\eta} &\lesssim ||  \langle k \rangle^m \min(|k|^{1/4}, 1) p_\xi^{-1/4} Z_\xi ||_{L^2_k L^1_\eta}\\
        &\lesssim \E^{1/2} + \mu^{-\delta_*} \D_\gamma^{\delta_*} \E^{1/2 - \delta_*},
\end{split}
\end{equation}
\begin{equation}\label{energy_with_alpha}
    ||\sk^m A(k) |\eta - kt|^{1/2} Z_\xi||_{L^2_k L^1_\eta} \lesssim \E^{1/2} + \mu^{-\delta_*}\D_\alpha^{\delta_*} \E^{1/2 - \delta_*}.
\end{equation}
\end{subequations}

\end{lemma}
\begin{proof}
We begin with the proof of \eqref{energy_with_gamma}.The estimate
    $$||  \langle k \rangle^{-1/2} p_\xi^{-1/4} Z_\xi ||_{L^1_k L^1_\eta} \lesssim  || \langle k \rangle^m \min(|k|^{1/4}, 1) p_\xi^{-1/4} Z_\xi ||_{L^2_k L^1_\eta}$$
    follows immediately from interpolation in $k$ using $m > 0$. We then switch from moving coordinates to stationary coordinates and interpolate in $\eta$ as follows:
    \begin{equation}
        \begin{split}
             ||\langle k \rangle^m \min(|k|^{1/4}, 1) p_\xi^{-1/4} Z_\xi ||_{L^2_k L^1_\eta} &=  || \langle k \rangle^m \min(|k|^{1/4}, 1) |\xi|^{-1/2} z_\xi ||_{L^2_k L^1_\eta}\\
             &= || (|\eta|^{1/4}|\eta|^{-1/4}1_{|\eta|<1} + |\eta|^{1/2 + 2\delta_*}|\eta|^{-1/2 - \delta_*}1_{|\eta| \geq 1} ) \\
             &\quad \quad \quad  \langle k \rangle^m \min(|k|^{1/4}, 1) |\xi|^{-1/2} z_\xi ||_{L^2_k L^1_\eta}\\
             &\lesssim ||\langle k \rangle^{m} \max(|\eta|^{1/4}, |\eta|^{1/2 + 2\delta_*}) \min(|k|^{1/4},1) |\xi|^{-1/2} z_\xi||_{L^2_\xi}.
        \end{split}
    \end{equation}
    Now we distinguish between low and high frequencies in $\eta$ to find, via interpolation, boundedness of the Riesz Transform, and returning to the moving coordinates
    \begin{equation}\label{energy_lemma_key_eqn}
        \begin{split}
            ||\langle k \rangle^{m} \max(|\eta|^{1/4}, |\eta|^{1/2 + 2\delta_*})& \min(|k|^{1/4},1) |\xi|^{-1/2} z_\xi||_{L^2_\xi}\\ &= ||1_{|\eta| < 1}\langle k \rangle^{m} |\eta|^{1/4} \min(|k|^{1/4},1) |\xi|^{-1/2} z_\xi||_{L^2_\xi}\\&\quad + ||1_{|\eta| \geq 1}\langle k \rangle^{m} |\eta|^{1/2 + 2\delta_*} \min(|k|^{1/4},1) |\xi|^{-1/2} z_\xi||_{L^2_\xi}\\
            &\lesssim ||1_{|\eta| < 1}\langle k \rangle^{m} |\eta|^{1/4}|k|^{1/4} |\xi|^{-1/2} z_\xi||_{L^2_\xi}\\&\quad + ||1_{|\eta| \geq 1}\langle k \rangle^{m} |\eta|^{1/2 + 2\delta_*} |\xi|^{-1/2} z_\xi||_{L^2_\xi}\\
            &\lesssim ||\langle k \rangle^{m}  z_\xi||_{L^2_\xi} + ||\langle k \rangle^{m} |\eta|^{2\delta_*} z_\xi||_{L^2_\xi}\\
            &\lesssim \E^{1/2} + ||\langle k \rangle^{m} |\eta - kt| Z_\xi||_{L^2_\xi}^{2 \delta_*} ||\langle k \rangle^{m}  Z_\xi||_{L^2_\xi}^{1-2 \delta_*}\\
            &\lesssim \E^{1/2} + \mu^{-\delta_*} \D_\gamma^{\delta_*} \E^{1/2 - \delta_*},
        \end{split}
    \end{equation}
    as desired. To prove \eqref{energy_with_alpha}, we transform coordinates and interpolate as in \eqref{energy_lemma_key_eqn}:
    \begin{equation}
        \begin{split}
            ||\sk^m A(k) |\eta - kt|^{1/2} Z_\xi||_{L^2_k L^1_\eta} & = ||\sk^m A(k) |\eta|^{1/2} z_\xi||_{L^2_k L^1_\eta}\\
            &\lesssim ||\sk^m A(k) |\eta|^{1/2} \langle \eta \rangle^{1/2 + 2 \delta_*} z_\xi||_{L^2_\xi}\\
            &\lesssim ||\sk^m A(k) |\eta - kt|^{1/2} Z_\xi||_{L^2_\xi} + ||\sk^m A(k) |\eta - kt|^{1 + 2 \delta_*} Z_\xi||_{L^2_\xi}\\
            &\lesssim  \E^{1/2} + \mu^{-\delta_*}\D_\alpha^{\delta_*} \E^{1/2 - \delta_*}.
        \end{split}
    \end{equation}
\end{proof}

\begin{lemma}\label{main_D_tau_lemma}
    For any $\delta_* \in (0,1/12)$ fixed, the following estimates hold:
    \begin{subequations}
    \begin{equation}
        \begin{split}
            ||\sk^{-1/2} |k| p_\xi^{-3/4} Z_\xi||_{L^1_\xi} &\lesssim ||\sk^{m} \min(|k|^{1/4},1) |k| p_\xi^{-3/4} Z_\xi||_{L^1_\xi}\\
            & \lesssim \D_\tau^{1/2} + \mu^{-\delta^*} \D_\gamma^{\delta_*} \D_\tau^{1/2 - \delta_*},
        \end{split}
    \end{equation}
        \begin{equation}
            ||\sk^m |k| p_\xi^{-1/2} Z_\xi||_{L^2_k L^1_\eta} \lesssim D_\tau^{1/2} + \mu^{-1/4-\delta_*} \D_\gamma^{1/4 + \delta_*} \D_{\tau}^{1/4 - \delta_*},
        \end{equation}
        \begin{equation}
            ||\sk^m \min(|k|^{1/4},1) |k| p_\xi^{-1/4} Z_\xi||_{L^2_k L^1_\eta} \lesssim \mu^{-1/4} \D_\gamma^{1/4}(\D_\tau^{1/4} + \mu^{-\delta_*} \D_\gamma^{\delta_*} \D_\tau^{1/4 - \delta_*}),
        \end{equation}
        \begin{equation}
            || \vk^n \sk^{-1/2} |k|^{1/2}  p_\xi^{-1/2} Z_\xi||_{L^1_\xi} \lesssim \D_\tau^{1/4} \E^{1/4} + \mu^{-\delta_*} \D_\gamma^{\delta_*} \D_\tau^{1/4 - \delta_*} \E^{1/4}.
        \end{equation}
    \end{subequations}
\end{lemma}
\begin{proof}
    The proofs of all of the estimates of Lemma \ref{main_D_tau_lemma} are variations on the ideas utilized in the proof of Lemma \ref{new_energy_lemmas}.
\end{proof}

\begin{lemma}\label{main_D_gamma_lemma}
For any $\delta_* \in (0,1/12)$ fixed, the following estimates hold:
\begin{subequations}
\begin{equation}\label{main_D_gamma_1}
    \begin{split}
    || \langle k \rangle^{-1/2} Z_\xi||_{L^1_\xi} & \lesssim || \langle k \rangle^{m} \min(|k|^{1/4},1) Z_\xi||_{L^2_k L^1_\eta}\\
    &\lesssim \mu^{-1/4}D_\gamma^{1/4}(\E^{1/4} + \mu^{-\delta_*} \D_\gamma^{\delta_*} \E^{1/4 - \delta_*}),
    \end{split}
\end{equation}
\begin{equation}\label{main_D_gamma_2}
        || \langle k \rangle^{m}  Z_\xi||_{L^2_k L^1_\eta} \lesssim \mu^{-1/8} \D_\gamma^{1/8}(\E^{3/8} +  \mu^{-1/4 - \delta_*} \D_\gamma^{1/4 + \delta_*} \E^{1/4 - \delta_*} ),
    \end{equation}
\begin{equation}\label{main_D_gamma_3}
        ||A(k)^{2\delta_*} \langle k \rangle^{-1/2}  Z_{\xi} ||_{L^1_\xi} \lesssim \mu^{-1/4}(1+\mu^{-\delta_*})\D_\gamma^{1/4}\E^{1/4}.
    \end{equation}
\end{subequations}
\end{lemma}
\begin{proof}
    The estimate \eqref{main_D_gamma_1} has an interpolation proof similar to the proof of \eqref{energy_with_gamma}. For \eqref{main_D_gamma_2}, we lack the factor of $\min(|k|^{1/4},1)$, and so we have by a slightly different proof:
    \begin{equation}
        \begin{split}
            || \langle k \rangle^{m}  Z_\xi||_{L^2_k L^1_\eta} &= || \langle k \rangle^{m}  z_\xi||_{L^2_k L^1_\eta}\\
            &\lesssim || \langle k \rangle^{m} |\eta|^{1/4} z_\xi||_{L^2_\xi} + || \langle k \rangle^{m} |\eta|^{1/2 + 2 \delta_*} z_\xi||_{L^2_\xi}\\
            &\lesssim \mu^{-1/8} \D_\gamma^{1/8}(\E^{3/8} +  \mu^{-1/4 - \delta_*} \D_\gamma^{1/4 + \delta_*} \E^{1/4 - \delta_*}).
        \end{split}
    \end{equation}
    The proof of \eqref{main_D_gamma_3} is slightly more interesting. Starting with an interpolation in $k$, we have by $m > 1/2$:
    \begin{equation}
        ||A(k)^{2\delta_*} \langle k \rangle^{-1/2}  Z_{\xi} ||_{L^1_\xi} \lesssim || \langle k \rangle^{m} A(k)^{2\delta_*} \min(|k|^{1/4},1)  Z_{\xi} ||_{L^2_k L^1_\eta}.
    \end{equation}
    Then by interpolation in stationary coordinates, as we have previously seen,
    \begin{equation}
        \begin{split}
            || \langle k \rangle^{m} A(k)^{2\delta_*} \min(|k|^{1/4},1)  Z_{\xi} ||_{L^2_k L^1_\eta} & \lesssim || \langle k \rangle^{m} A(k)^{2\delta_*} \min(|k|^{1/4},1) |\eta-kt|^{1/4} Z_{\xi} ||_{L^2_\xi}\\
            &\quad + || \langle k \rangle^{m} A(k)^{2\delta_*} \min(|k|^{1/4},1) |\eta-kt|^{1/2 + 2 \delta_*} Z_{\xi} ||_{L^2_\xi}\\
            &\lesssim \mu^{-1/4} \D_\gamma^{1/4} \E^{1/4}\\
            &\quad + || \langle k \rangle^{m} |\eta-kt| Z_{\xi} ||_{L^2_\xi}^{1/2 - 2 \delta_*}|| \langle k \rangle^{m} |\eta-kt| Z_{\xi} ||_{L^2_\xi}^{1/2}\\
            &\quad \quad || \langle k \rangle^{m} A(k) |\eta-kt| Z_{\xi} ||_{L^2_\xi}^{2\delta_*}\\
            &\lesssim \mu^{-1/4}(1+\mu^{-\delta_*})\D_\gamma^{1/4} \E^{1/4},
        \end{split}
    \end{equation}
    proving \eqref{main_D_gamma_3}.
\end{proof}

\begin{lemma}\label{alpha_energy_control}
For any $\delta_* \in (0,1/12)$, the following estimate, purely in terms of the energy $\E$ holds:
\begin{equation}
    ||  \sk^{-1/2} A(k)^{2\delta_*} p_\xi^{-1/4} Z_\xi||_{L^1_\xi} \lesssim ||  \sk^{m} A(k)^{2\delta_*} \min(|k|^{1/4},1)p_\xi^{-1/4} Z_\xi||_{L^2_k L^1_\eta} \lesssim \E^{1/2} + \mu^{-\delta^*} \E^{1/2}.
\end{equation}
\end{lemma}
\begin{proof}
We do not present the proof of Lemma \ref{alpha_energy_control}, since it follows similar ideas to those of Lemma \ref{new_energy_lemmas} and \eqref{main_D_gamma_3} from Lemma \ref{main_D_gamma_lemma}.
\end{proof}

\begin{lemma}\label{strange_interp}
The following estimate holds:
    \begin{equation}
        ||  \sk^m |k| A(k) p_\xi^{-1/4} Z_\xi||_{L^2_k L^1_\eta} \lesssim \D_\tau^{1/4}\D_\beta^{1/4} + \mu^{-\delta_*}\D_\alpha^{\delta_*}\D_{\tau \alpha}^{1/4}(\D_\beta^{1/8} \D_\gamma^{1/8 - \delta_*}  + \D_\beta^{1/4 - \delta_*}).
    \end{equation}
\end{lemma}
\begin{proof}
    To prove Lemma \ref{strange_interp}, we will split into $|k| \geq \mu$ and $|k| < \mu$. For $|k| < \mu$, $A(k) = 1$. Hence
    \begin{equation}
        \begin{split}
            ||  \sk^m |k| A(k) p_\xi^{-1/4} Z_\xi||_{L^2_{|k| < \mu} L^1_\eta} = ||  \sk^m |k| p_\xi^{-1/4} Z_\xi||_{L^2_{|k| < \mu} L^1_\eta}.
        \end{split}
    \end{equation}
    Now by interpolation in stationary coordinates we find
    \begin{equation}
        \begin{split}
            ||  \sk^m |k| p_\xi^{-1/4} Z_\xi||_{L^2_{|k| < \mu} L^1_\eta} & \lesssim ||  \sk^m |k| p_\xi^{-1/4} \langle \eta - kt \rangle^{1/2 + \delta_*}Z_\xi||_{L^2_{|k| < \mu} L^2_\eta}\\
            &\lesssim ||\sk^m |k| p_\xi^{-1/4} Z_\xi||_{L^2_{|k| < \mu} L^2_\eta}\\
            &\quad + ||\sk^m |k| p_\xi^{-1/4} |\eta -kt|^{1/2 + 2 \delta_*} Z_\xi||_{L^2_{|k| < \mu} L^2_\eta}\\
            &\lesssim \D_\tau^{1/4} \D_\beta^{1/4}\\
            &\quad + ||\sk^m (|k| |\eta -kt| p_\xi^{-1/2})^{1/2} |k|^{1/2} |\eta -kt|^{2 \delta_*} Z_\xi||_{L^2_{|k| < \mu} L^2_\eta}\\
            &\lesssim \D_\tau^{1/4} \D_\beta^{1/4} + \D_{\tau \alpha}^{1/4} \mu^{-\delta_*} \D_\alpha^{\delta_*} \D_\beta^{1/4 - \delta_*}.
        \end{split}
    \end{equation}
    For high-in-$k$ frequencies, $A(k) = \mu^{1/3} |k|^{-1/3}$. Thus we obtain
    \begin{equation}
        \begin{split}
            ||  \sk^m |k| A(k) p_\xi^{-1/4} Z_\xi||_{L^2_{|k| \geq \mu} L^1_\eta} & \lesssim ||  \sk^m |k| A(k) \langle \eta - kt \rangle^{1/2 + 2 \delta_*} p_\xi^{-1/4} Z_\xi||_{L^2_{|k| \geq \mu} L^2_\eta}\\
            &\lesssim ||  \sk^m (|k| p_\xi^{-1/4})^{1/2} \mu^{1/3} |k|^{1/6} Z_\xi||_{L^2_{|k| \geq \mu} L^2_\eta}\\
            &\quad + ||  \sk^m (|k| A(k) |\eta-kt| p_\xi^{-1/2})^{1/2} \mu^{1/6} |k|^{1/3} |\eta-kt|^{ 2\delta_*} Z_\xi||_{L^2_{|k| \geq \mu} L^2_\eta}\\
            &\lesssim \D_\tau^{1/4} \D_\beta^{1/4} + \mu^{1/6}\D_{\tau \alpha}^{1/4}||\sk^m |k|^{1/3} Z_\xi||_{L^2_{|k| \geq \mu}L^2_\eta}^{1/4}||\sk^m |k| Z_\xi||_{L^2_{|k| \geq \mu}L^2_\eta}^{1/4 - 2 \delta_*}\\
            &\quad \quad \quad \quad ||\sk^m p_\xi^{1/4} |\eta -kt| Z_\xi||_{L^2_{|k| \geq \mu}L^2_\eta}^{2 \delta_*}\\
            &\lesssim \D_\tau^{1/4} \D_\beta^{1/4} + \mu^{-\delta_*}\D_{\tau \alpha}^{1/4}\D_\beta^{1/8} \D_\gamma^{1/8 - \delta_*} \D_\alpha^{\delta_*},
        \end{split}
    \end{equation}
    which completes the proof.
\end{proof}
\begin{lemma}\label{D_taualpha_and_alpha}
For any $\delta_* \in (0,1/12)$, the following holds
\begin{equation}
     ||\sk^m  \min(|k|^{1/4},1) |k|^{2/3} |\eta - kt| p_\xi^{-3/4} Z_\xi||_{L^2_{|k| \geq \mu} L^1_\eta} \lesssim \mu^{-1/3}(\D_{\tau\alpha}^{1/2} + \mu^{-\delta_*} \D_\alpha^{\delta_*} \D_{\tau \alpha}^{1/2 - \delta_*})
\end{equation}
\end{lemma}
\begin{proof}
We omit the proof for the sake of space.
\end{proof}
\begin{lemma}\label{half_energy_half_alpha_lemma}For any $\delta_* \in (0,1/12)$, the following holds:
    \begin{equation}
        ||\dk^J \sk^m  A(k) |\eta -kt| Z_\xi||_{L^2_k L^1_\eta} \lesssim \mu^{-1/4}(\E^{1/4} \D_\gamma^{1/4} + \mu^{-\delta_*}\E^{1/4-\delta_*} \D_{\alpha}^{1/4 + \delta_*})
    \end{equation}
\end{lemma}
\begin{proof}
    The proof does not have new ideas beyond those used in Lemma \ref{main_D_gamma_lemma}.
\end{proof}
\begin{lemma}\label{main_D_tau_alpha_lemma} For any $\delta_* \in (0,1/12)$, the following holds:
            \begin{equation}
        || \sk^m A(k) |k|^{1/2}|\eta - kt|^{1/2} p_{\xi}^{-1/4} Z_\xi||_{L^2_k L^1_\eta} \lesssim \D_{\tau \alpha}^{1/4} \E^{1/4} + \D_{\tau \alpha}^{1/4} \mu^{-\delta_*} \D_\gamma^{\delta_*} \E^{1/4 - \delta_*}
    \end{equation}
\end{lemma}
\begin{proof}
We present the following straightforward interpolation:
    \begin{equation}
    \begin{split}
        || \sk^m A(k) &|k|^{1/2}|\eta - kt|^{1/2} p_{\xi}^{-1/4} Z_\xi||_{L^2_k L^1_\eta}\\
        &=|| \sk^m A(k) |k|^{1/2}|\eta|^{1/2} |\xi|^{-1/2} z_\xi||_{L^2_k L^1_\eta}\\
        &\lesssim || \sk^m A(k) |k|^{1/2}|\eta|^{1/2} |\xi|^{-1/2} z_\xi||_{L^2_\xi}\\
        &+ \quad || \sk^m A(k) |k|^{1/2}|\eta|^{1 + 2\delta_*} |\xi|^{-1/2} z_\xi||_{L^2_k L^2_\xi}\\
        &\lesssim \D_{\tau \alpha}^{1/4} \E^{1/4}\\
        &\quad + || \sk^m (A(k) |\eta-kt| |k|p_\xi^{-1/2})^{1/2}\\
        &\quad \quad \quad \quad \quad \quad (A(k)|\eta-kt|)^{1/2} |\eta- kt|^{2\delta_*}Z_\xi||_{L^2_{xi}}\\
        &\lesssim \D_{\tau \alpha}^{1/4} \E^{1/4} + \D_{\tau \alpha}^{1/4} \mu^{-\delta_*} \D_\gamma^{\delta_*} \E^{1/4 - \delta_*}.
    \end{split}
\end{equation}
\end{proof}

\section*{Acknowledgement}

This material is based upon work supported by the National Science Foundation Graduate Research Fellowship Program under Grant No. DGE-2034835. Any opinions, findings, and conclusions or recommendations expressed in this material are those of the authors and do not necessarily reflect the views of the National Science Foundation.

The author would like to thank J. Bedrossian for helpful conversations and suggestions.\\


\bibliographystyle{plain} 
\bibliography{citations} 

\end{document}